\documentclass
{amsart}

\usepackage{color}
\usepackage{amsmath}
\usepackage{amssymb}
\usepackage[abbrev]{amsrefs}
\usepackage{stmaryrd}
\usepackage{graphicx}
\usepackage{bm}

\newtheorem{theorem}{Theorem}[section]
\newtheorem{proposition}[theorem]{Proposition}
\newtheorem{lemma}[theorem]{Lemma}

\theoremstyle{definition}
\newtheorem{definition}[theorem]{Definition}
\newtheorem{example}[theorem]{Example}

\theoremstyle{definition}
\newtheorem{remark}[theorem]{Remark}
\newtheorem{assumption}[theorem]{Assumption}

\numberwithin{equation}{section}

\def\Inf{\operatornamewithlimits{inf\vphantom{p}}}

\newcommand{\bfone}{{\mathbf{1}}}
\newcommand{\eps}{{\varepsilon}}

\newcommand{\abs}[1]{\left|{#1}\right|}
\newcommand{\norm}[1]{\lVert{#1}\rVert}

\newcommand{\N}{{\mathbb{N}}}

\newcommand{\R}{{\mathbb{R}}}

\newcommand{\wx}{{\mathbf{x}}}
\newcommand{\wy}{{\mathbf{y}}}
\newcommand{\wz}{{\mathbf{z}}}

\newcommand{\bfp}{{\mathbf{p}}}
\newcommand{\bfq}{{\mathbf{q}}}

\begin{document}
\title[Path-dependent Hamilton-Jacobi equations]
{Path-dependent Hamilton-Jacobi equations  in infinite dimensions}
\author{Erhan Bayraktar}
\address
{Department of Mathematics, 
University of  Michigan,
Ann Arbor, MI 48109, United States}
\thanks{The research of the first author was
 supported in part by the National Science Foundation under grant DMS-1613170}
\email{erhan@umich.edu}
\author{ Christian Keller}
\address
{Department of Mathematics, 
University of  Central Florida,
Orlando, FL 32816, United States}
\email{christian.keller@ucf.edu}

\date{July 19, 2018}


\begin{abstract}
We propose notions of minimax and viscosity solutions
 for a class of
fully nonlinear path-dependent PDEs with nonlinear,
monotone, and coercive operators on Hilbert space.
Our main result is well-posedness (existence, uniqueness, and stability)
for minimax solutions. 
{\color{black} A particular novelty 
is a suitable combination of minimax and
viscosity solution techniques in the proof of the
comparison principle.} 
 {\color{black} One of the main difficulties, the lack of 
compactness in infinite-dimensional Hilbert spaces,
is circumvented by working with suitable compact subsets of our path space.}
As an application, our theory  makes it possible to employ the dynamic programming approach 
to study 
optimal control problems for a fairly general class of (delay) evolution equations 
 in the variational
framework.
Furthermore, differential games  associated to such evolution equations
can be  investigated following the Krasovski\u\i-Subbotin approach similarly
as in finite dimensions. 
\end{abstract}
\maketitle
\pagestyle{plain}
\tableofcontents
\section{Introduction}
Let $V\subseteq H\subseteq V^\ast$ be a Gelfand triple, i.e., 
$V$ is a separable reflexive Banach space with a continuous and dense
embedding into a Hilbert space $H$. Moreover, this embedding is assumed to
be compact. We study
fully nonlinear so-called path-dependent PDEs (PPDEs) of the form
\begin{equation}\label{E:PPDE_intro}
\begin{split}
&\partial_t u(t,x)-\langle A(t,x(t)),\partial_x u(t,x)\rangle +F(t,x,\partial_x u(t,x))=0,\\
&\qquad\qquad\qquad\qquad\qquad\qquad\qquad\qquad(t,x)\in [0,T)\times  C([0,T],H).
\end{split}
\end{equation}
Here, 
$A(t,\cdot): V\to V^\ast$, $t\in [0,T]$, are nonlinear, monotone, coercive operators.
The derivatives $\partial_t u$ and $\partial_x u$ are certain path derivatives on 
the path space
$C([0,T],H)$. A definition will be provided later. Note that they are not Fr\'echet derivatives.
 Let us also mention that any kind of 
solution $u$ of \eqref{E:PPDE_intro} should at least be \emph{non-anticipating}, i.e.,
for every $x$, $y\in C([0,T],H)$,
whenever $x=y$ on $[0,t]$, we have $u(t,x)=u(t,y)$.

A particular example  of \eqref{E:PPDE_intro}
is the Bellman equation associated to  distributed
control problems for  divergence-form quasi-linear parabolic PDEs of order $2m$ on some bounded
domain $G\subset\R^n$ with smooth boundary $\partial G$, e.g., 
the problem of minimizing a
cost functional
\begin{align*}
\int_{t_0}^T \ell(t, \{x(s,\cdot)\}_{s\le t},a(t,\cdot))\,dt+ h(\{x(t,\cdot)\}_{t\le T}) 
\end{align*}
over some class of admissible controls $a=a(t,\xi)$ 
subject to $x=x^{t_0,x_0,a}$ solving
\begin{align*}
\frac{\partial x}{\partial t}(t,\xi)&+\sum_{\abs{\alpha}\le m} 
(-1)^{\abs{\alpha}} D^\alpha A_\alpha(t,\xi,D x(t,\xi))\\
&= f(t,\xi,\{x(s,\xi)\}_{s\le t},a(t,\xi)),
&& (t,\xi)\in (t_0,T)\times G,\\
D^{\beta} x(t,\xi)&=0, &&(t,\xi)\in [t_0,T]\times\partial G,\quad \abs{\beta}\le m-1,\\
x(t,\xi)&=x_0(t,\xi), && (t,\xi)\in [0,t_0]\times G,
\end{align*}
where $\alpha$ and $\beta$ are multi-indices, $D x= (D^{\beta} x)_{\abs{\beta}\le m}$,
$D^\alpha$ and $D^\beta$ are partial derivatives with respect to $\xi$, and $\ell$, $h$, $f$, and
$A_\alpha$ are appropriate functions (see \cite[Section~30.4]{ZeidlerIIB} for details regarding $A_\alpha$).
Note that we allow delays in the data of our problem with the exception of the differential operator.

The objective of this paper is to establish wellposedness for the terminal-value problem
related to \eqref{E:PPDE_intro}
under an appropriate notion of solution. Our problem is set up in path-dependent framework,
i.e., a solution $u$ depends at any fixed time $t$ on  a``path segment" $\{x(s)\}_{s\le t}$,
which can be understood as history until time~$t$. Formally, the usual state space
$[0,T]\times H$ is replaced by $[0,T]\times C([0,T],H)$ (or a subset) and a solution is required
to be non-anticipating. The original motivation to work in such a framework
comes from problems involving delays in areas such
as optimal control, differential games, and mathematical finance. 
In this work however, the path-dependent approach provides a new methodology for
infinite-dimensional PDEs, which is of relevance even if our data are not path-dependent, 
i.e., when the function $F$ and a terminal datum  $h$ are of the form
\begin{align*}
F(t,x,z)=\tilde{F}(t,x(t),z),\quad h(x)=\tilde{h}(x(T)),\quad (t,x,z)\in [0,T]\times C([0,T],H)\times H,
\end{align*}
for some functions $\tilde{F}:[0,T]\times H\times H\to\R$ and $\tilde{h}: H\to\R$.
The reason is that, thanks to our additional structural assumptions for $A$ in contrast to 
the related works
\cite{Tataru92JMAA,Tataru94JDE_simplified}
and \cite{CrandallLionsVI}, we can identify suitable compact subspaces of $C([0,T],H)$ and this
allows us to avoid the use of  perturbed optimization
based on
Ekeland's variational principle. 
We want to remark that our approach
crucially depends on the coercivity of $A$ and thus for different or more general situations
perturbed optimization is an indispensable tool. Nevertheless, we  want to advocate
our approach when it is applicable as it often simplifies the analysis and many methods
from the finite-dimensional theory can be easily carried over to the infinite-dimensional case. 
\subsection{Related research}
{\color{black} Pioneering work on Hamilton-Jacobi equations on Hilbert space
 has been done by Barbu and Da Prato (see, e.g., their monograph \cite{BarbuDaPrato83book}).}

Viscosity solutions for PDEs in infinite dimensions of the form
\begin{align*}
 v_t(t,x)-\langle Ax,v_x(t,x)\rangle+G(t,x,v_x(t,x))=0,\quad (t,x)\in (0,T)\times H,
\end{align*}
(and their stationary counterparts), where $A$ is an unbounded operator,
 have been investigated by
Crandall and Lions in \cite{CrandallLionsI, CrandallLionsII, CrandallLionsIII} for the case $A=0$,
in \cite{CrandallLionsIV, CrandallLionsV, CrandallLionsVII} for linear $A$,
and in \cite{CrandallLionsVI} for nonlinear $A$. The research on such equations with 
nonlinear $A$ 
has been started earlier by Tataru \cite{Tataru92JMAA,Tataru92AMO}
(see also \cite{Tataru94JDE_simplified}
for subsequent work). For the case of $A$ being the subdifferential of a convex function,
we refer to Ishii~\cite{Ishii92JFA_HJ_Hilbert} and Shimano~\cite{Shimano02AMO}.
{\color{black} 
Cannarsa, Gozzi, and Soner \cite{CannarsaGozziSoner91AMO} studied
Hamilton-Jacobi equations on Hilbert space  (here with $A=0$) related to optimal
control problems with exit time and state constraints.
Motivated by boundary control problems for parabolic PDEs,
 Cannarsa, Gozzi, and Soner \cite{CannarsaGozziSoner93JFA} and  Cannarsa and Tessitore
\cite{CannarsaTessitore96SICON} studied Hamilton-Jacobi equations, where the Hamiltonian $G$
involves fractional powers of $A$.}

We also want to mention the works of Lions \cite{Lions89_Zakai} and Gozzi and {\'Swi\polhk ech}
\cite{GSJFA00} as well as the works  of Gozzi, Sritharan, and {\'Swi\polhk ech}
 \cite{GSS_2D_NS,GSS_2D_SNS}, respectively,
where Bellman equations associated to control problems for variational solutions of
the Zakai equation
and the Navier-Stokes equation, respectively, are studied.
Furthermore, we want to single out Soner's paper \cite{Soner88}, which influenced the literature
on viscosity solutions for fully nonlinear 
1st and 2nd order PPDEs via \cite{Lukoyanov07_viscosity}.

{\color{black}
For the  viscosity theory of  second-order PDEs in infinite dimensions, we refer to the works of Lions
\cite{Lions88Acta,Lions89JFA},
of Ishii \cite{Ishii93CommPDE},   of  {\'Swi\polhk ech}, \cite{Swiech94CommPDE}  
and of Gozzi, Rouy, and {\'Swi\polhk ech}  
\cite{GRS00SICON}
besides the already mentioned papers  \cite{Lions89_Zakai}, \cite{GSJFA00}, \cite{GSS_2D_SNS}.
}

For further works about viscosity solutions for PDEs in infinite dimensions, 
one can consult the bibliographies of \cite{FabbriThesis} and \cite{FabbriGozziSwiech}.

Results on minimax solutions in infinite dimensions can be found in the work
\cite{ClarkeLedyaev94TAMS} by Clarke and Ledyaev on mean value inequalities.
We refer also to Carja \cite{Carja00JOTA, Carja12SICON}
regarding the related notion of contingent solution.

Fully nonlinear 1st order PPDEs have been studied by Lukoyanov in the minimax
\cite{Lukoyanov01minimax,Lukoyanov03a, Lukoyanov03b}
as well as in the viscosity solution framework
\cite{Lukoyanov07_viscosity}. In those works, so-called
coinvariant derivatives by Kim \cite{KimBook} were used.
The investigation of 2nd order PPDEs has been motivated much 
by Peng's ICM plenary talk \cite{PengICM}
and by Dupire's functional It\^o calculus \cite{dupirefunctional}
(see also subsequent work by Cont and Fourni\'e
\cite{ContFournie10_JFA}). A  viscosity solution approach
has been successfully initiated by
 Ekren, Keller, Touzi, and Zhang \cite{EKTZ}. 
 In a similar spirit,
Isaacs equations are studied by Pham and Zhang \cite{PhamZhang14},
  fully nonlinear equations by Ekren, Touzi, and Zhang \cite{ETZ_I, ETZ_II},
 non-local equations by Keller \cite{Keller16},  elliptic equations by Ren \cite{Ren16},
  obstacle problems by Ekren \cite{Ekren2017}, {\color{black} and degenerate second-order
  equations by Ren, Touzi, and Zhang \cite{RTZ17}. In contrast to the papers mentioned first,
  the latter work covers also the first-order case.}
 
 A class of semilinear 2nd order PPDEs 
 with a linear unbounded operator on Hilbert space has been studied by
 Cosso, Federico, Gozzi, Rosestolato, and Touzi \cite{cosso2015path}.
 Their work is closely related to the semigroup framework, in particular to
 mild solutions of stochastic evolution equations involving linear operators, whereas the present work
 is closely related to variational solutions of (deterministic)
 evolution equations involving nonlinear operators.
 
 For optimal control problems in the variational framework,
 we refer to  Hu and Papageorgiou \cite[Chapter~IV]{HPHandbookII}
 and the references therein. 
 
 Differential games in infinite dimensions have been investigated by 
 Kocan, Soravia, and  {\'Swi\polhk ech}
 \cite{KocanSoraviaSwiech97JMAA},
 by Kocan and Soravia
 \cite{KocanSoravia00SICON},
 by Shaiju \cite{Shaiju03},
 by Ghosh and Shaiju
 \cite{GhoshShaiju04},
 by Ramaswamy and Shaiju
 \cite{RamaswamyShaiju09},
 by Nowakowska and Nowakowski
 \cite{NowakowskaNowakowski11},
 by  {\'Swi\polhk ech} \cite{Swiech11games},
 by Banks and Shuhua
\cite{BanksShuhua12}, and
 by Vlasenko, Rutkas, and {Chikri\u\i}
 \cite{VlasenkoRutkasChikrii15}.
 {\color{black} Stochastic differential games in infinite dimensions and the related
 Isaacs equations have been studied
 by Fleming and Nisio \cite{FlemingNisio93}, by Nisio \cite{Nisio93, Nisio94,Nisio96,Nisio98,
 Nisio99a,Nisio99b},
 and by   {\'Swi\polhk ech} \cite{Swiech02NA}.}
 
 Let us remark that the literature on optimal control problems for PDEs is immense
(see, e.g, the recent survey article \cite{ControlPDE_survey} 
or the monographs \cite{AhmedTeo},
\cite{BensoussanEtAl07}, \cite{FabbriGozziSwiech}, 
\cite{Fattorini}, \cite{Fursikov}, \cite{LasieckaTriggiani_I}, \cite{LasieckaTriggiani_II}
\cite{LiYong}, 
\cite{LionsJL_control}, 
\cite{Tiba}, and \cite{Troltzsch}; note that this list
is not exhaustive).

{\color{black}
Finally, let us conclude this literature review with mentioning two approaches
how to treat  optimal control problems with delays via the dynamic programming approach.
A very common approach in the literature
is to rewrite a controlled delay differential equation as an ordinary differential
equation without delays but with values in a Hilbert space $H$ such as, e.g., $H=L^2(-\tau,0;\R^d)$.
This would then naturally lead to an investigation of Bellman equations on domains
of the form $[0,T]\times H$ or $H$; see, e.g, the 
 works 
of Federico \cite{Federico11FinStoch},  
of Federico, Goldys, and Gozzi 
\cite{FedericoGoldysGozzi10SICON,FedericoGoldysGozzi11SICON}, 
 of Federico and Tacconi \cite{FedericoTacconi14SICON}, 
of Fuhrman, Masiero, and Tessitore \cite{FuhrmanFedericaTessitore10SICON}, 
of Zhou and Liu \cite{ZhouLiu12JDE}, and
of Zhou and Zhang \cite{ZhouZhang11IJC}.
In the PPDE approach however, the controlled delay differential equations is not rewritten as
an (infinite-dimensional) evolution equation but is considered as a ``path-dependent"
differential equation, i.e., a differential equation of the form
$x^\prime(t)=f(t, \{x(s)\}_{s\le t}))$
in the deterministic case. This view would lead  to a path-dependent Bellman equation.
 Furthermore, solutions of those PPDE should be non-anticipating.
A particular feature of this work is that it also allows the study of optimal control problems
associated to PDEs with delays.
}

\subsection{Our approach and the main difficulties} 
We propose
 an appropriate
adaptation of the notion of minimax
solutions in the sense of Subbotin \cite{Subbotin_book} 
and Lukoyanov \cite{Lukoyanov03a} as generalized solution
for \eqref{E:PPDE_intro}. Under this notion, we establish well-posedness.
Moreover, minimax solutions admit  similarly as viscosity solutions
an infinitesimal characterization, which is of  importance in
control theory, in particular, to establish verification theorems 
(cf.~\cite{FabbriGozziSwiech10JConvex}) and for the
 synthesis of optimal feedback controls
(see, e.g., \cite{Frankowska89AMO}, \cite[Chapter~III]{BardiCapuzzoDolcetta}, or
\cite[Chapter~II]{Subbotina2006}).

The main difficulties are the lack of compactness in infinite-dimensional
spaces and the interpretation of the duality pairing $\langle A(t,x(t)),\partial_x u(t,x)\rangle$
in \eqref{E:PPDE_intro} due to incompatibilities regarding the involved spaces:
We have
$A(t,x(t))\in V^\ast$ and thus  $\partial_x u(t,x)$ should belong to $V$ but, in general,
we only have
$\partial_x u(t,x)\in H$ and $H$ is in most
relevant situations  a strict superset of $V$.
The first difficulty is overcome by the path-dependent
approach itself  and the use of minimax solutions
as we are able to identify
suitable compact subsets of the path space $C([0,T],H)$.
The second difficulty actually disappears in our formulation of minimax 
solutions. In the formulation of viscosity solutions the problem is resolved by
replacing $\langle A(t,x(t)),\partial_x u(t,x)\rangle$ with a term of the form
\begin{align*}
\varliminf_{\delta\downarrow 0}\text{ (or $\varlimsup_{\delta\downarrow 0}$) } \frac{1}{\delta}\int_t^{t+\delta}
\langle A(s,x(s)),\partial_x u(s,x)\rangle\,ds,
\end{align*}
in which case the space incompatibilities disappear as,
for a.e.~$s\in (t,T)$, we have  $x(s)\in V$ 
and thus $\partial_x u(s,x)\in V$  (cf.~Definition~\ref{D:C111}~(i) and
Definition~\ref{PDC11}~(i))
for suitable trajectories $x$.

Without rendering our problem irrelevant or trivial,
we also consider smaller domains than $[0,T]\times C([0,T],H)$ such as
 (compact!) sets  of the form $[0,T]\times\Omega^L$, $L\ge 0$, and the 
(locally compact!) set $[0,T]\times\Omega$, where $\Omega:=\cup_{L\ge 0} \Omega^L$.
The sets $\Omega^L$ can be thought of as spaces of \emph{realizable trajectories}
(or possible histories) corresponding to a suitable family of control problems.
From this point of view, it strikes us at least as 
debatable why a larger class of trajectories should
be considered.

In the proof of the comparison principle for viscosity solutions
of 1st order PPDEs in finite dimensions by Lukoyanov \cite{Lukoyanov07_viscosity}
(see also \cite[Theorem~8.4]{ETZ_I})  time and space variables are doubled, i.e.,
functions of the form $u(t,x)-v(s,y)+\text{ ``penalty term"}$ are used.
However,  this approach  seems to be problematic in our setting as it is
not clear how one can then exploit the monotonicity of the operator $A$
(even if $A$ is time-independent!). 
To avoid this issue, we double only the space variable 
following  Ishii \cite{IshiiRemark} 
and Crandall and Lions
\cite{CrandallLionsVI}.  To obtain a comparison
principle for viscosity solutions we  need then to establish
a so-called doubling theorem, i.e., a statement of the following form:
\begin{quotation}
If $u$ is a  viscosity subsolution and $v$ is a viscosity supersolution,
then $w$ defined by $w(t,x,y):=u(t,x)-v(t,y)$ is a viscosity subsolution
of some ``doubled equation."
\end{quotation}
Due to the particular nature of our path derivatives we have been
unable to prove such a statement. (The problem is that in the path-dependent
case the time
derivative  does not depend on time alone, which can also be seen  in alternative,
explicit definitions of path derivatives such as in \cite{dupirefunctional}.) 
To get around this obstacle, we utilize the stronger notion of minimax solutions.
By doing so,
we can prove a doubling theorem but now with $u$  being a \emph{minimax} subsolution
and $v$ being a \emph{minimax} supersolution. However,
$w$ will still be a {\emph{viscosity}} subsolution.  Our doubling
theorem will then be used to prove a comparison principle for \emph{minimax} solutions.

\subsection{Notions, main results, and methodology in a simplified setting}
Suppose that there exists a constant $L\ge 0$ such that, for every $z$, $\tilde{z}\in H$,
\begin{align}\label{E:F_lip_intro}
\abs{F(t,x,z)-F(t,x,\tilde{z})}\le L\norm{z-\tilde{z}}_H,\quad (t,x)\in [0,T]\times C([0,T],H).
\end{align}
Also note that hypothesis \textbf{H}($A$) from Section~\ref{S:A}
is tacitly assumed to be satisfied.

In the finite-dimensional minimax solution theory, 
spaces of Lipschitz continuous functions with common Lipschitz
constant play an important role. Their counterpart in our work are the trajectory spaces
\begin{align*}
\tilde{\mathcal{X}}^L(t_0,x_0)&=\{x \in C([0,T],H): x=x_0\text{ on $[0,t_0]$ and
 $\exists f^x\in L^2(t_0,T;H):$}\\
&\qquad x^\prime(t)+A(t,x(t))=f^x(t),\,\abs{f^x(t)}\le L\text{ a.e.~on $(t_0,T)$}\},
\end{align*}
where $(t_0,x_0)\in [0,T)\times C([0,T],H)$. 
(For a precise definition, see Remark~\ref{R:Compact}; cf.~also
Definition~\ref{D:trajectorySpace}.)
These spaces are compact, connected,
and satisfy a certain continuous-dependence type 
property (cf.~Proposition~\ref{P:CompactII}).  The last property is needed
to prove existence for solutions of \eqref{E:PPDE_intro}.

Fix $x_\ast\in H$. Here, our state space is $[0,T]\times\tilde{\Omega}$, where
$\tilde{\Omega}=\tilde{\mathcal{X}}^L(0,x_\ast)$.
In the main part of this paper, we use a larger set as state space.
In particular, this set does not depend on
the Lipschitz constant $L$. However,  our notions will  then become slightly more
complicated.

Assume for a moment that $u$ is smooth in some sense and satisfies \eqref{E:PPDE_intro}.
Then, for every $x\in\tilde{\mathcal{X}}^L(t_0,x_0)$, 
we should have 
\begin{equation}\label{E:classical_intro}
\begin{split}
u(t,x)-u(t_0,x_0)&=\int_{t_0}^t \partial_t u(s,x)+\langle x^\prime(s),\partial_x u(s,x)\rangle\,ds\\
&=\int_{t_0}^t \partial_t u(s,x) +\langle -A(s,x(s))+f^x(s),\partial_x u(s,x)\rangle\,ds\\
&=\int_{t_0}^t  -F(s,x,\partial_x u(s,x))+\langle f^x(s),\partial_x u(s,x)\rangle\,ds.
\end{split}
\end{equation}
(A precise definition of the path derivatives can be found in Section~\ref{S:path_derivatives}.)
Because of \eqref{E:F_lip_intro}, one can show that
\eqref{E:classical_intro} also holds if
 the terms $\partial_x u(s,x)$, $s\in [t_0,t]$, are replaced by an arbitrary $z\in H$.
 This motivates the following
definition.

\begin{definition}
A function $u\in C([0,T]\times\tilde{\Omega})$ is a \emph{minimax solution} of
\eqref{E:PPDE_intro} if, for every  $(t_0,x_0,z)\in [0,T)\times\tilde{\Omega}\times H$,
there exists an $x\in\tilde{\mathcal{X}}^L(t_0,x_0)$ such that
\begin{align*}
u(t,x)-u(t_0,x_0)=\int_{t_0}^t -F(s,x,z)+(f^x(s),z)\,ds,\quad t\in [t_0,T].
\end{align*}
\end{definition}
Similarly, minimax semisolutions (i.e., sub- and supersolutions) can be defined.

The next statement is a weaker version of our main result.
\begin{theorem}
Let $h:C([0,T],H)\to\R$ be continuous.
Let $F$ be continuous and suppose that besides \eqref{E:F_lip_intro},
for every $t\in [0,T]$, $x$, $y\in \tilde{\Omega}$, and $z\in H$,
\begin{align*}
\abs{F(t,x,z)-F(t,\tilde{x},z)}\le L(1+\norm{z}_H)\sup_{s\le t} \norm{x(s)-y(s)}_H.
\end{align*}
 Then there exists a unique minimax solution $u$ of \eqref{E:PPDE_intro}
that satisfies $u(T,\cdot)=h$.
\end{theorem}

To prove this result, we employ viscosity subsolutions
(see  Definition~\ref{D:Viscsub_doubled})
of  the so-called ``doubled equation"
\begin{equation}\label{E:PPDEdoubled_intro}
\begin{split}
&\partial_t w(t,x,y)-\langle A(t,x(t)),\partial_x w(t,x,y)\rangle 
-\langle A(t,y(t)),\partial_y w(t,x,y)\rangle\\ &\qquad+F(t,x,\partial_x w(t,x,y))
-F(t,y,-\partial_y w(t,x,y))=0,\\
&\qquad\qquad\qquad\qquad\qquad\qquad (t,x,y)\in [0,T)\times  C([0,T],H)\times C([0,T],H).
\end{split}
\end{equation}
Our methodology can then be summarized in the following five-step scheme:\\
\noindent\fbox{%
    \parbox{\textwidth}{%
(i) (``Doubled comparison"). Let $w$ be a viscosity subsolution of \eqref{E:PPDEdoubled_intro}.
Then $w(T,x,x)\le 0$ for all $x\in\tilde{\Omega}$ implies $w(t,x,x)\le 0$ 
for all $(t,x)\in [0,T]\times\tilde{\Omega}$.

(ii) (Doubling theorem). If $u$ is a minimax sub-, and $v$ a minimax supersolution, then $w$ defined by
$w(t,x,y):=u(t,x)-v(t,y)$ is a viscosity subsolution of \eqref{E:PPDEdoubled_intro}.

(iii) (Comparison).  ``(i)+(ii)" yields  comparison  for minimax semisolutions.

(iv) (``Perron"). Define a function $u_0$ by
\begin{align*}
u_0(t,x):=\inf\{u(t,x):\text{ $u$ ``improper" minimax supersolution}\}
\end{align*}
and consider its semi-continuous envelopes $u_0^-$ and $u_0^+$. By definition
$u_0^-\le u_0\le u_0^+$. Moreover,  $u_0^-$ is an l.s.c.~minimax super- and
$u_0^+$ is a u.s.c.~minimax subsolution.

(v) (Existence). ``Comparison + Perron" yields $u_0^+\le u_0^-$, i.e., $u_0^+=u_0=u_0^-$.
}}\\

Next, we present a special case of the distributed control problem
mentioned at the beginning of this introduction
 to illustrate how those problems fit into our abstract setting.

\begin{example}
This example has been borrowed from \cite[Chapter~23]{ZeidlerIIA}.
For the sake of simplicity, we consider the problem of distributed control
for the heat equation, i.e., $x=x^{t_0,x_0,a}$ is supposed to be a
weak solution of 
\begin{align*}
\frac{\partial x}{\partial t} (t,\xi)-\Delta_\xi x(t,\xi)&=a(t,\xi), && (t,\xi)\in (t_0,T)\times G,\\
x(t,\xi)&=0, && (t,\xi)\in [t_0,T]\times\partial G,\\
x(0,\xi)&=x_0(\xi), && (t,\xi)\in [0,t_0]\times G.
\end{align*}
In other words, for every $\varphi\in C_0^\infty(G)$ (the space of smooth functions with bounded support),
\begin{align*}
\frac{d}{dt} \int_G x(t,\xi)\,\varphi(\xi)\,d\xi&+
\int_G \sum_{i=1}^n D_i x(t,\xi)\,D_i \varphi(\xi)\,d\xi\\&=\int_G a(t,\xi)\,\varphi(\xi)\,d\xi
&&\text{a.e.~on $(t_0,T)$,}\\
x(t,\xi)&=0, && (t,\xi)\in [t_0,T]\times\partial G,\\
x(0,\xi)&=x_0(\xi), && (t,\xi)\in [0,t_0]\times G.
\end{align*}
An abstract formulation of this Cauchy-Dirichlet problem is
\begin{align*}
x^\prime(t)+Ax(t)&=a(t)&&\text{a.e.~on $(t_0,T)$,}\\
x(t)&=x_0(t)&&\text{on $[0,t_0]$,}
\end{align*}
and we require that  a solution $x$ together with its generalized derivative $x^\prime$ satisfies
\begin{align*}
x\in C([0,T],H)\cap L^2(t_0,T; V)\text{ and }
x^\prime\in  L^2(t_0,T;  V^\ast).
\end{align*}
Here, $H=L^2(G)$, $V=W_0^{1,2} (G)$ is the closure of $C_0^\infty(G)$ under $\norm{\cdot}_{W^{1,2}(G)}$,
$ V^\ast$ is the dual of  $V$,  and the right-hand side $a:[t_0,T]\to V^\ast$  and the operator 
$A:V\to V^\ast $, $x\mapsto Ax$, are defined by
\begin{align*}
\langle a(t),v \rangle &:=\int_G a(t,\xi) v(\xi)\,d\xi, \quad v\in V,\\
\langle Ax,v\rangle&:=\int_G \sum_{i=1}^n D_i x(\xi) D_i v(\xi)\,d\xi,\quad v \in V.
\end{align*}
\end{example}

\subsection{Organization of the rest of the paper}
In Section~2, we introduce the setting for this work
including most of the notation and the standing assumptions.
Furthermore, some of the preliminary results concerning evolution equations
will be proven. Moreover, a definition for the  already mentioned path derivatives will 
be provided. In Section~3, we state our main object of study,
namely the terminal-value problem related to \eqref{E:PPDE_intro},
together with the standing hypotheses for the data.
Then we present our notion of minimax solutions for this problem and prove
basic results regarding this notion.
In Section~4, we prove a comparison principle 
for minimax solutions using viscosity solution techniques.
In Section~5, Perron's method is used to establish existence for minimax solutions.
In Section~6, a stability result for minimax solutions is proven.
In Section~7, we apply the dynamic programming approach to optimal control problems associated 
with nonlinear evolution equations, establish (local) Lipschitz continuity for
the corresponding value function in space as well as in time, and
show that the value function is a minimax solution of the corresponding
Bellman equation. The proof of the last statement also uses a combination
of minimax and viscosity solution techniques, which is new to our best knowledge.
In Section~8, we adapt the Krasovski\u\i-Subbotin approach for differential
games to our setting. Under the Isaacs condition, 
it is shown that our differential games have value
and the value function is a minimax solution of the corresponding Isaacs equation.
Appendix~A contains the proofs for some of the crucial results about solution sets of 
evolution equations.
In Appendix~B, we present first a notion of classical solution and discuss
their relation to minimax solutions. Then we present a notion of viscosity solutions
for which we have existence but cannot prove uniqueness.

\section{Setting and preliminary results}
Let $V\subseteq H\subseteq V^\ast$ be a Gelfand triple,
 i.e., $V$ is a separable reflexive Banach space with a continuous and dense
embedding into a Hilbert space $H$. 
Denote by $\norm{\cdot}$, $\abs{\cdot}$,  and $\norm{\cdot}_\ast$ the norms on $V$, $H$, and $V^\ast$.
(With slight abuse of notation, we also denote by $\abs{\cdot}$ the Euclidean norm on $\R^n$.)
Denote by $\langle \cdot,\cdot\rangle$ the duality pairing between $V$ and its dual $V^\ast$, 
i.e., $\langle x^\ast,x\rangle=x^\ast(x)$, $x^\ast\in V^\ast$, $x\in V$.
Denote by $(\cdot,\cdot)$ the inner product on $H$. Moreover, using appropriate identifications, 
we have $\langle x,y\rangle=(x,y)$ for every $x\in H$ and $y\in V$.

\begin{assumption}\label{A:compact}
The embedding from $V$ into $H$ is compact.
\end{assumption}

Fix $p\ge 2$ and $q>1$ with $p^{-1}+q^{-1}=1$.
Given $t_0\in [0,T)$,
let $W_{pq}(t_0,T)$ be the space of all $x\in L^p(t_0,T;V)$ possessing a generalized derivative
$x^\prime\in L^q(t_0,T;V^\ast)$. Equipped with  the norm $\norm{\cdot}_{W_{pq}(t_0,T)}$ defined by
\begin{align*}
\norm{x}_{W_{pq}(t_0,T)}:=\norm{x}_{L^p(t_0,T;V)}+\norm{x^\prime}_{L^q(t_0,T;V^\ast)},
\end{align*}
the space $W_{pq}(t_0,T)$ becomes a Banach space.

\begin{remark}\label{R:Wpq:continuous}
The space $W_{pq}(t_0,T)$ is continuously embedded into $C([t_0,T],H)$
(see, e.g., \cite[Proposition~23.23~(ii)]{ZeidlerIIA}).
\end{remark}

\begin{remark}\label{R:compact}
 As a consequence of Assumption~\ref{A:compact}, the embedding from
 $W_{pq}(t_0,T)$ into $L^p (t_0,T;H)$
is compact (see \cite[Th\'eor\`eme~I.12.1]{JLLions1969}). This is crucial
for the proof of Proposition~\ref{P:compact} below.
However, also note that, according to \cite{Migorski95PAMS}, it is, in general, not true that 
the embedding from $W_{pq}(0,T)$ into $C([0,T],H)$ is compact.
\end{remark}

We shall frequently use the following result
(see, e.g., \cite[Proposition~23.23~(iv)]{ZeidlerIIA}).

\begin{proposition}[Integration-by-parts] \label{P:parts}
Let $t_0\in [0,T)$. If $x$, $y\in W_{pq}(t_0,T)$ and $t_0\le t\le s\le T$, then
\begin{align*}
(x(s),y(s))-(x(t),y(t))=\int_{t}^{s} \langle x^\prime(r),y(r)\rangle+\langle y^\prime(r),x(r)\rangle\,dr. 
\end{align*}
On the left-hand side, $x$ and $y$ are to be understood as elements of $C([t_0,T],H)$.
\end{proposition}

\subsection{Path space and related topologies}
The domain for  solutions to our path-dependent PDE will of the form $[0,T]\times\tilde{\Omega}$, 
where $T>0$ and  $\tilde{\Omega}$ is a subset of 
the  path space $C([0,T], H)$, which we equip 
with the uniform norm
$\norm{\cdot}_\infty$ defined by
\begin{align*}
\norm{x}_\infty:=\sup_{t\in [0,T]} \abs{x(t)}. 
\end{align*}
If not mentioned otherwise, any subset of $C([0,T],H)$ is to
be understood as a metric space with respect to $\norm{\cdot}_\infty$.

\begin{definition}
A function $u$ defined on some subset $S$ of $[0,T]\times C([0,T],H)$ is 
\emph{non-anticipating} if 
$u(t,x)=u(t,y)$ whenever $x=y$ on $[0,t]$, $(t,x)$, $(t,y)\in S$.
\end{definition}
To ensure that (semi-) continuous functions are non-anticipating,
we equip $[0,T]\times\Omega$ with the pseudo-metric~$\mathbf{d}_\infty$ 
defined by
\begin{align*}
\mathbf{d}_\infty((t,x),(s,y))&:=
\abs{t-s}+\sup_{r\in [0,T]} 
\abs{ x(r\wedge t)-y(r\wedge s)}.
\end{align*}
Here, $a\wedge b:=\min\{a,b\}$.
Given a non-empty subset $S$ of $[0,T]\times C([0,T],H)$, we shall employ the following 
function spaces:
\begin{align*}
C(S)&:=\{u:S\to \R\text{ continuous with respect to $\mathbf{d}_\infty$}\},\\
\mathrm{USC}(S)&:=\{u:S\to \R\text{ upper semi-continuous with respect to $\mathbf{d}_\infty$}\},\\
\mathrm{LSC}(S)&:=\{u:S\to \R\text{ lower semi-continuous with respect to $\mathbf{d}_\infty$}\}.
\end{align*}
Clearly, the elements of those function spaces are non-anticipating.

\subsection{The operator $A$, related trajectory spaces, and evolution equations}
\label{S:A}

Throughout this work,  we fix an operator $A:[0,T]\times V\to V^\ast$ and assume that
 the following hypotheses hold.

\textbf{H}($A$):  
(i) The mappings
$t\mapsto \langle A(t,x),v\rangle$, 
$v\in V$,
are measurable.

(ii) Monotonicity: 
For a.e.~$t\in (0,T)$ and every $x$, $y\in V$, 
\begin{align*}
\langle A(t,x)-A(t,y),x-y\rangle \ge 0.
\end{align*}

(iii) Hemicontinuity: 
For a.e.~$t\in (0,T)$ and every $x$, $y$, $v\in V$, 
the mappings
$s\mapsto \langle A(t,x+sy), v\rangle$, $[0,1]\to\R$, 
are continuous.

(iv) Boundedness: There exist a function $a_1\in L^q(0,T)$ and a number $c_1\ge 0$ such that,
for every $x\in V$ and a.e.~$t\in (0,T)$,
\begin{align*}
\norm{A(t,x)}_\ast \le a_1(t)+c_1\norm{x}^{p-1}.
\end{align*}

(v) Coercivity: There exists a number $c_2> 0$  such that, for every  $x\in V$ and a.e.~$t\in (0,T)$,
\begin{align*}
\langle A(t,x),x\rangle \ge c_2 \norm{x}^p.
\end{align*}

The following trajectory spaces  are ubiquitous  in the present work.
They are the natural analogues to the spaces of Lipschitz continuous functions
with common Lipschitz constant 
in the finite-dimensional theory.

\begin{definition}\label{D:trajectorySpace}
Given  $L\ge 0$ and $(t_0,x_0)\in [0,T)\times C([0,T],H)$, put
\begin{equation*}
\boxed{
\begin{split}
\mathcal{X}^L(t_0,x_0):=\{x\in C([0,T],H):\,  
\text{$\exists \wx\in W_{pq}(t_0,T):\exists f^x\in L^2(t_0,T;H)$}:\\
 \mathbf{x}^\prime(t)+A(t,\mathbf{x}(t))=f^x(t) \text{ a.e.~on $(t_0,T)$, }
x=x_0\text{ on $[0,t_0]$, }\\
 \qquad
 x=\wx\text{ a.e.~on $(t_0,T)$, and }
 \abs{f^x(t)}\le L(1+\sup_{0\le s\le t} \abs{x(s)})
\text{ a.e.~on $(t_0,T)$}\}.
\end{split}}
\end{equation*}
\end{definition}
\begin{remark}
Given $x\in\mathcal{X}^L(t_0,x_0)$, the function $\wx$ is 
actually the unique element of $W_{pq}(t_0,T)$ 
that  coincides with $x$ a.e.~on $(t_0,T)$ 
(in particular, uniqueness of $\wx^\prime$ in this context
 follows
from \cite[Proposition~23.20~(a)]{ZeidlerIIA})
and also
$f^x$ is uniquely determined in  $L^2(t_0,T;H)$
(actually, even in $L^q(t_0,T;V^\ast)$), which can be seen  by setting
$f^x(t)=\wx^\prime(t)+A(t,\wx(t))$.
\end{remark}


Fix $x_\ast\in H$.  Considering $x_\ast$ as a constant function on $[0,T]$,
we set 
\begin{align*}
\text{$\Omega^L:=\mathcal{X}^L(0,x_\ast)$, $L\ge 0$, and $\Omega:=\cup_{L\ge 0} \Omega^L$.}
\end{align*}
Those path spaces will serve as our primary domains for solutions
of our path-dependent PDE. In particular, $\Omega$ can be used
as  \emph{universal path space} for all Hamiltonians $F$
satisfying the hypothesis \textbf{H}($F$) below, 
whereas the spaces $\Omega^L$ can only be used for Hamiltonians
that are $L$-Lipschitz in the gradient (in the sense of condition
\textbf{H}($F$)(ii) with $L_0$ replaced by $L$).

\begin{remark}
The sets $\mathcal{X}^L(t_0,x_0)$ are not stable under stopping, i.e.,
from $x\in\mathcal{X}^L(t_0,x_0)$ one cannot deduce
that $x(\cdot\wedge t)\in\mathcal{X}^L(t_0,x_0)$ if $t<T$.
(Recall that $s\wedge t=\min\{s,t\}$.)
\end{remark}

\begin{remark}\label{R:equivPseuoMetrics}
On $\mathcal{X}^L(t_0,x_0)$, the pseudo-metrics
\begin{align*}
(x,y)\mapsto \sup_{s\le t} \abs{x(s)-y(s)}\text{ and } 
(x,y)\mapsto \left(\int_0^t \abs{x(s)-y(s)}^2\,ds\right)^{1/2},\, t\in (t_0,T],
\end{align*}
are topologically (but not pseudo-metrically) equivalent. The non-trivial direction follows from  
 \eqref{E:CompactEndProof}.
\end{remark}

The next proposition is crucial. Note that
compactness
of the sets $\mathcal{X}^L(t_0,x_0)$ (see, e.g., \cite{TT1999NoDEA} for a more general
result in the context of evolution inclusions) 
ultimately relies on $V$ being compactly 
embedded in $H$ (Assumption~\ref{A:compact}) and its ramifications
(see, in particular, Remark~\ref{R:compact}). 
Due to its importance and for the sake of completeness,
a proof will be provided in the appendix. 
Let us also mention that compactness of solution sets of monotone
evolution inclusions is  addressed in 
\cite{Migorski96Edinburgh_compactness}, 
\cite{PS97AMO}, and
\cite{HPHandbookII}.

\begin{proposition}\label{P:compact}
The sets $\mathcal{X}^L(t_0,x_0)$ are non-empty, compact, and path-connected
in $C([0,T],H)$.
\end{proposition}

\begin{proof}
Non-emptiness follows from \cite[Theorem~30.A, p.~771]{ZeidlerIIB},
compactness from Lemma~\ref{L:CompactW} and
Lemma~\ref{L:Compact},
and path-connectedness from Lemma~\ref{L:Connected}.
\end{proof}

As the next example shows, coercivity of $A$ is necessary for
$\mathcal{X}^L(t_0,x_0)$ to be compact in $C([0,T],H)$.

\begin{example} [$A=0$]
Let $H=\{ v\in L^2(0,2\pi): v(0)=v(2\pi)=0\}$ and $V=H^1_0(0,2\pi)$, which
is the completion of $C^\infty_0(0,2\pi)$ under $\norm{\cdot}_{H^1(0,2\pi)}$.
Consider the sequence $(x_n)_n$  in $W_{2,2}(0,1)$ defined by
\begin{align*}
x_n(t)\,(\xi):=t f_n(t)\,(\xi):= \frac{t}{\sqrt{\pi}} \sin(n\xi),\quad t\in [0,1],\, \xi\in (0,2\pi).
\end{align*}
Then $x^\prime_n(t)=f_n(t)$, $x_n(0)=0$. Clearly, $x_n\in\mathcal{X}^1(0,0)$
for every $n\in\N$. Since $x_n(1)\xrightarrow{w} 0$ in $H$
but $\abs{x_n(1)}=1$ for every $n\in\N$, the sequence
$(x_n)_n$  cannot have a subsequence that converges in $C([0,T],H)$.
Also note that $\sup_n \norm{x_n}_{C([0,1],H)}\le 1$
but $\sup_n \norm{x_n}_{L^2(0,1;V)}=\infty$, i.e., 
$(x_n)_n$ is not bounded in $W_{2,2}(0,1)$.
\end{example}

The following continuous dependence type result 
will play a central role in carrying out Perron's method
to prove our main existence result 
(see the proof of Lemma~\ref{L:u0-+SemiSol}).
It is also needed to prove stability.

\begin{proposition}\label{P:CompactII}
Let $L\ge \tilde{L}\ge 0$. 
Let $(t_n,x_n)\to (t_0,x_0)$ in $[0,T)\times\Omega^L$.
Consider a sequence $(\tilde{x}_n)_n$ in $\Omega^L$ with
 $\tilde{x}_n\in\mathcal{X}^{\tilde{L}}(t_n,x_n)$,
$n\in\N$. Then $(\tilde{x}_n)_n$ has a converging subsequence 
with a limit in $\mathcal{X}^{\tilde{L}}(t_0,x_0)$.
\end{proposition}

\begin{proof}
By Proposition~\ref{P:compact}, $(\tilde{x}_n)_n$ has a subsequence, which, with slight abuse of notation
is still denoted by
$(\tilde{x}_n)_n$, that converges to some $\tilde{x}_0\in\Omega^L$. 
To show that $\tilde{x}_0\in\mathcal{X}^{\tilde{L}}(t_0,x_0)$ we need to verify that
\begin{align}\label{E:CompactII:1}
\tilde{x}_0&=x_0&&\text{on $[0,t_0]$ and that}\\ \label{E:CompactII:2}
\abs{f^{\tilde{x}}(t)}&\le \tilde{L}(1+\sup_{s\le t}\abs{\tilde{x}(s)})&&\text{a.e.~on $(t_0,T)$.} 
\end{align}

First, we show \eqref{E:CompactII:1}.
To this end, let $t\in [0,t_0]$. 
Since, for every $n\in\N$,
\begin{align*}
\abs{\tilde{x}_0(t)-x_0(t)}\le\abs{\tilde{x}_0(t)-\tilde{x}_n(t)}+
\abs{\tilde{x}_n(t)-x_n(t)}+\abs{x_n(t)-x(t)}
\end{align*}
and $\abs{\tilde{x}_0(t)-\tilde{x}_n(t)}+\abs{x_n(t)-x(t)}\to 0$ as $n\to\infty$, we only need to deal
with the term $\abs{\tilde{x}_n(t)-x_n(t)}$. Note that in the case $t_n\ge t$, this term equals $0$
because $\tilde{x}_n\in\mathcal{X}^{\tilde{L}}(t_n,x_n)$. Thus, we shall assume  from now on that $t_n<t$.
By the a-priori estimates from Lemma~\ref{L:Apriori}
applied to $\Omega^L=\mathcal{X}^L(0,x_\ast)$, 
there exists a $C\ge 0$ independent from $n$
such that
\begin{align*}
\abs{\tilde{x}_n(t)-x_n(t)}^2&\le \abs{x_n(t_n)-x_n(t_n)}^2+
\int_{t_n}^t \abs{f^{\tilde{x}_n}(s)-f^{x_n}(s)} \cdot
\abs{\tilde{x}_n(s)-x_n(s)}\,ds\\
&\le 4\tilde{L}(1+C)C(t_0-t_n).
\end{align*}
 Letting $n\to\infty$  yields \eqref{E:CompactII:1}.
 
 Next, we show  \eqref{E:CompactII:2}. Keep in mind that
 $f^{\tilde{x}_n}\xrightarrow{w} f^{\tilde{x}}$ in $L^2(0,T;H)$ and that
 \begin{align*}
 \abs{f^{\tilde{x}}(t)}&\le L(1+\sup_{s\le t}\abs{\tilde{x}(s)}) &&\text{a.e.~on $(0,T)$,}\\
 \abs{f^{\tilde{x}_n}(t)}&\le L(1+\sup_{s\le t}\abs{\tilde{x}_n(s)}) &&\text{a.e.~on $(0,t_n)$,}\\
  \abs{f^{\tilde{x}_n}(t)}&\le \tilde{L}(1+\sup_{s\le t}\abs{\tilde{x}_n(s)}) &&\text{a.e.~on $(t_n,T)$.}
 \end{align*}
 Without loss of generality, we distinguish between the following two cases.
 
 \textit{Case~1:} Let $t_n\uparrow t_0$. In this case, one can proceed as in the proof
 of Lemma~\ref{L:Compact}. 
 
 \textit{Case~2:} Let $t_n\downarrow t_0$. Following the lines of the proof of 
  of Lemma~\ref{L:Compact}, we can infer that
  \begin{align*}
  f^{\tilde{x}}(t)\in\overline{\mathrm{conv}}
  \text{ $w$-$\varlimsup$} \{f^{\tilde{x}_n}(t)\}_{n\in\N}=:\tilde{E_1}\qquad\text{a.e.~on $(t_0,T)$.}
\end{align*}
Let $t\in (f^{\tilde{x}})^{-1}(\tilde{E}_1)\cap (t_0,T)=:E_1$. Then there exists an $n_0\in\N$ such that,
for all $n\ge n_0$, we have $t_0\le t_n< t$. Moreover, suppose that, in addition, we have
\begin{align}\label{E:CompactII:3}
t\in \bigcap_{n\ge n_0} \left\{s\in (t_0,T):\,\abs{f^{\tilde{x}_n}(s)}\le \tilde{L} (1+\sup_{r\le s} 
\abs{\tilde{x}_n(r)})\right\}=:E_2.
\end{align}
Next, let $h\in\text{ $w$-$\varlimsup$} \{f^{\tilde{x}_n}(t)\}_{n\in\N}$.
Then there exists a subsequence $(f^{\tilde{x}_{n_k}}(t))_k$ of $(f^{\tilde{x}_{n}}(t))_n$
that converges weakly to $h$ in $H$.  Thus, by \cite[Proposition~21.23 (c), p.~258]{ZeidlerIIA} (Banach-Steinhaus), by \eqref{E:CompactII:3},
   and since $\tilde{x}_n\to \tilde{x}$ in $C([0,T],H)$, 
 \begin{align*}
 \abs{h}\le \varliminf\limits_k\abs{f^{\tilde{x}_{n_k}}(t)} 
 \le \varliminf\limits_k  \tilde{L}(1+\sup_{s\le t} \abs{\tilde{x}_{n_{k}}(s)})
 =\tilde{L}(1+\sup_{s\le t} \abs{\tilde{x}(s)}).
 \end{align*}
Since  $(E_1\cap E_2)^c$ is a null set, we can deduce that \eqref{E:CompactII:2} holds.
\end{proof}

We shall also need the following result concerning evolution equations 
with more general right-hand side than in \cite{ZeidlerIIB}.

\begin{proposition}\label{P:EU_ODE}
Let $(t_0,x_0)\in [0,T)\times C([0,T],H)$ and $L\ge 0$. 
Consider the initial-value problem
\begin{equation}\label{E:IVP_ODE}
\begin{split}
\wx^\prime(t)+A(t,\wx(t))=f(t,x)\quad\text{a.e.~on $(t_0,T)$,}\quad
{x}=x_0\quad\text{on $[0,t_0]$,}
\end{split}
\end{equation}
where $f:[t_0,T]\times C([0,T],H)\to H$ is non-anticipating
and measurable.

(i) If, for a.e.~$t\in (t_0,T)$, the mapping $x\mapsto f(t,x)$, $C([0,T],H)\to H$, is
continuous, 
and if, for a.e~$t\in (0,T)$ and every $x\in C([0,T],H)$,
\begin{align*}
 \abs{f(t,x)}\le L(1+\sup_{s\le t} \abs{x(s)}),
 \end{align*}
 then
\eqref{E:IVP_ODE} has a solution in $\mathcal{X}^L(t_0,x_0)$.

(ii) If, in addition to the assumptions in (i), there exists an $l\ge 0$ such that,
for a.e.~$t\in (t_0,T)$ and every $x$, $y\in\mathcal{X}^L(t_0,x_0)$,
\begin{align*}
\abs{f(t,x)-f(t,y)}\le l \sup_{s\le t}\abs{x(s)-y(s)},
\end{align*}
then \eqref{E:IVP_ODE} has a unique solution in $\mathcal{X}^L(t_0,x_0)$.
\end{proposition}

\begin{proof}
(i) First, let $M\in\N$ and define $f^M:[t_0,T]\times C([0,T],H)\to H$ by
\begin{align*}
f^M(t,x):=\begin{cases}
f(t,x) &\text{ if $\sup_{s\le t}\abs{x(s)}\le M$,}\\
f\left(t,\frac{M}{\sup_{s\le t}\abs{x(s)}}\, x\right)
&\text{ if $\sup_{s\le t}\abs{x(s)}> M$.}
\end{cases}
\end{align*}
{\color{black}
Next, consider the set
\begin{equation*}
\begin{split}
\tilde{\mathcal{X}}^{L(1+M)}(t_0,x_0):=\{x\in C([0,T],H):\,  
\text{$\exists \wx\in W_{pq}(t_0,T):\exists f^x\in L^2(t_0,T;H)$}:\\
 \mathbf{x}^\prime(t)+A(t,\mathbf{x}(t))=f^x(t) \text{ a.e.~on $(t_0,T)$, }
x=x_0\text{ on $[0,t_0]$, }\\
 \qquad
 x=\wx\text{ a.e.~on $(t_0,T)$, and }
 \abs{f^x(t)}\le L(1+M)
\text{ a.e.~on $(t_0,T)$}\}.
\end{split}
\end{equation*}
and define a  mapping  
\begin{align*}
\Gamma^M: C([0,T],H) \to C([0,T],H),\quad x\mapsto \Gamma^M x=y\in
\mathcal{X}^{L(1+M)}(t_0,x_0),
\end{align*}
 by
\begin{align*}
\wy^\prime(t)+A(t,\wy(t))=f^M(t,x)\quad\text{a.e.~on $(t_0,T)$},\quad
y=x_0\quad\text{ on $[0,t_0]$}.
\end{align*}
This is possible according to  \cite[Theorem~30.A, p.~771]{ZeidlerIIB}.
Note that $\tilde{\mathcal{X}}^{L(1+M)}(t_0,x_0)$ is compact (the proof is similar 
and actually slightly easier than the proof of Proposition~\ref{P:compact}).
Moreover, $\Gamma^M$ is continuous as $x_n\to x$ in $C([0,T],H)$ implies,
for  $y_n:=\Gamma x_n$, $y:=\Gamma x$, and a.e.~$t\in (t_0,T)$, 
\begin{align*}
&\langle (\wy-\wy_n)^\prime(t),(y-y_n)(t)\rangle\\
&\qquad =-\langle A(t,\wy(t))-A(t,\wy_n(t)),
(\wy-\wy_n)(t)\rangle+ (f(t,x)-f(t,x_n),(y-y_n)(t))\\
&\qquad \le C\abs{f(t,x)-f(t,x_n)}
\end{align*}
thanks to the monotonicity of $A$
and thus (cf.~part (ii) of this proof)
\begin{align*}
\norm{\Gamma^M x-\Gamma^M x_n}^2_\infty \le \int_{t_0}^T 2C \abs{f^M(t,x)-f^M(t,x_n)}\,dt
\to 0
\end{align*}
by Proposition~\ref{P:parts} and the dominated convergence theorem. Here, $C$  is the constant from 
standard a-priori estimates similarly to the ones of Lemma~\ref{L:Apriori}.
Thus, by Schauder's fixed point theorem, there exists a   fixed point 
$x^M$
 of $\Gamma^M$.
Moreover, note that  $x^M\in\mathcal{X}^L(t_0,x_0)$
because, for  a.e.~$t\in (t_0,T)$, 
\begin{align*}
\abs{f^M(t,x^M)}\le L(1+\sup_{s\le t} \abs{x^M(s)}).
\end{align*}}
Hence, we can apply Proposition~\ref{P:compact} to deduce that
$(x^M)_{M}$ has a subsequence, which we still denote by $(x^M)_{M}$, that
converges to some  $x\in\mathcal{X}^{L}(t_0,x_0)$. 
Furthermore, by Lemma~\ref{L:CompactW},
  $f^M(\cdot,x^M)\xrightarrow{w} f^x$ in $L^2(t_0,T;H)$.
 Since, for sufficiently large $M$,
$\norm{x^M}_\infty\le \norm{x}_\infty+1\le M$ and   
$f^{M}(\cdot,x^M)
=f(\cdot,x^M)$,
 it follows
that $f^M(t,x^M)\to f(t,x)$ for a.e.~$t\in (t_0,T)$. Therefore,
we can conclude that
$f^x=f(\cdot,x)$, i.e., we have proven existence.

(ii) Let $x$ and $y$ be solutions of \eqref{E:IVP_ODE} in $\mathcal{X}^L(t_0,x_0)$. Put $z:=x-y$.
Then, for a.e.~$t\in (t_0,T)$, by  monotonicity of $A$,
\begin{align*}
{\color{black}\langle\wz^\prime(t),z(t)\rangle}&=- \langle A(t,\wx(t))-A(t,\wy(t)),\wz(t)\rangle
+(f(t,x)-f(t,y),z(t))
\le l\,m(t)^2,
\end{align*}
where $m(t):=\sup_{s\le t} \abs{z(s)}$. 
Using integration-by-parts (Proposition~\ref{P:parts}) yields
$\frac{1}{2}\frac{d}{dt}\abs{z(t)}^2\le l\,m(t)^2$,
which we integrate  to obtain
${\color{black}
\abs{z(t)}^2\le \int_{t_0}^t 2 l\,m(s)^2\,ds}$. 
By Gronwall's inequality, $z=0$. This concludes the proof.
\end{proof}

\subsection{Chain rule and standard derivatives}
This section
serves as motivation for the path derivatives in the following section.
The results here  will also be used in the treatment of differential games.

 A corresponding result in the second-order case in a setting close to ours can be found
in  \cite[Section~5.4.3]{Nisio_book}. In particular, slight variants of
the conditions ($F_0$) and ($F_1$)
from \cite[pp.~186 and 190]{Nisio_book} 
found their way into the following  definition.

\begin{definition}\label{D:C111}
Denote by $\mathcal{C}_V^{1,1,1}([0,T]\times\R\times H)$ the set of all functions 
$\psi=\psi(t,\xi,x)\in C^{1,1,1}([0,T]\times\R\times H)$ such that
the following holds:

(i)  For a.e.~$t\in (0,T)$ and every $\xi\in\R$,
\begin{align*}
x\in V\text{ implies } \psi_x(t,\xi,x)\in V.
\end{align*}

(ii) The functions $\abs{\psi_t}$ and $\abs{\psi_\xi}$ 
are bounded
on bounded subsets of $[0,T]\times\R\times H$.

(iii) There exists a constant $L_\psi\ge 0$ such that, 
for every $x_1$, $x_2\in V$,
\begin{align*}
\norm{\psi_x(t,\xi,x_1)-\psi_x(t,\xi,x_2)}\le L_{\psi} \norm{x_1-x_2}
\text{ and } \norm{\psi_x(t,\xi,x_1)}\le L_\psi(1+\norm{x_1})
\end{align*}
\end{definition}

\begin{proposition}\label{P:ChainRule}
If $\psi\in\mathcal{C}_V^{1,1,1}([0,T]\times\R\times H)$, then
\begin{align*}
&\psi(t_2,\xi(t_2),x(t_2))-\psi(t_1,\xi(t_1),x(t_1))\\&\qquad =
\int_{t_1}^{t_2} \psi_t(t,\xi(t),x(t))+
\psi_\xi(t,\xi(t),x(t))\,\xi^\prime(t)+
\langle\mathbf{x}^\prime(t),\psi_x(t,\xi(t),\mathbf{x}(t))\rangle\,dt 
\end{align*}
for every $\xi\in C^1([0,T])$, $\mathbf{x}\in W_{pq}(0,T)$,
$x\in C([0,T],H)$ with $\wx=x$ a.e.~on $(0,T)$,  and $t_1$, $t_2\in [0,T]$ with $t_1\le t_2$. 
\end{proposition}

\begin{proof}
Fix $\xi\in C^1([0,T])$, $\mathbf{x}\in W_{pq}(0,T)$, and $x\in C([0,T],H)$
with $\wx=x$  a.e.~on $(0,T)$. 
We use a standard approximation argument (cf.~\cite[Proof of Theorem~1.12~(d)]{HPHandbookII}).
Since $C^1([0,T],V)$ is dense in $W_{pq}(0,T)$,
we can choose a sequence $(x_n)_n$ in $C^1([0,T],V)$ that converges to $\mathbf{x}$ in $W_{pq}(0,T)$
and thus also to $x$ in $C([0,T],H)$. 
Let $0\le t_1\le t_2\le T$. Note that, for every $n\in\N$,
\begin{align*}
&\psi(t_2,\xi(t_2),x_n(t_2))-\psi(t_1,\xi(t_1),x_n(t_1))\\&\qquad=
\int_{t_1}^{t_2} \psi_t(t,\xi(t),x_n(t))+ 
\psi_\xi(t,\xi(t),x_n(t))\,\xi^\prime(t)+
( \psi_x(t,\xi(t),x_n(t)), x_n^\prime(t))\,dt\\
&\qquad=
\int_{t_1}^{t_2} \psi_t(t,\xi(t),x_n(t))+ 
\psi_\xi(t,\xi(t),x_n(t))\,\xi^\prime(t)+
\langle x_n^\prime(t), \psi_x(t,\xi(t),x_n(t))\rangle\,dt.
\end{align*}
Without loss of generality, $\norm{x_n}_\infty\le \norm{x}_\infty+1$ for every $n\in\N$.
Thus, by condition~(ii) of Definition~\ref{D:C111},
\begin{align*}
&\lim_{n\to\infty} \int_{t_1}^{t_2} \psi_t(t,\xi(t),x_n(t))+ 
\psi_\xi(t,\xi(t),x(t))\,\xi^\prime(t)\,dt\\&\qquad=
\int_{t_1}^{t_2} \psi_t(t,\xi(t),x(t))+ 
\psi_\xi(t,\xi(t),x(t))\,\xi^\prime(t)\,dt.
\end{align*}
Also, without loss of generality, $\norm{x_n}_{L^p(0,T;V)}\le \norm{x}_{L^p(0,T;V)}+1$ for every $n\in\N$.
Thus, by condition~(ii) of Definition~\ref{D:C111}, there exists a constant $C>0$ such that,
\begin{align*}
&\abs{\int_{t_1}^{t_2} \langle x_n^\prime(t), \psi_x(t,\xi(t),x_n(t))\rangle
-\langle\mathbf{x}^\prime(t),\psi_x(t,\xi(t),\mathbf{x}(t))\rangle\,dt}\\
 &\qquad \le
\norm{x_n^\prime-\mathbf{x}^\prime}_{L^q(0,T;V^\ast)} \cdot
\norm{\psi_x(\cdot,\xi(\cdot),x_n(\cdot))}_{L^p(0,T;V)}
\\ &\qquad\qquad +
\norm{\mathbf{x}^\prime}_{L^q(0,T;V^\ast)}\cdot
\norm{\psi_x(\cdot,\xi(\cdot),x_n(\cdot))-\psi_x(\cdot,\xi(\cdot),\mathbf{x}(\cdot))}_{L^p(0,T;V)}\\
&\qquad \le
L_\psi C \norm{x_n-\mathbf{x}}_{W_{pq}} \cdot 
\left(1+\norm{\mathbf{x}}_{L^p(0,T;V)}
\right)\\
&\qquad\qquad +
L_\psi\norm{\mathbf{x}^\prime}_{L^q(0,T;V^\ast)} \cdot
\norm{x_n-\mathbf{x}}_{W_{pq}}\\
&\qquad \to 0\quad(n\to\infty). 
\end{align*}
It is now straight-forward to complete the proof.
\end{proof}

\subsection{Functional chain rule and path derivatives}\label{S:path_derivatives}
Fix $n\in\N$. We shall use the spaces
\begin{align*}
C([0,T],H)^{n}&=\underbrace{ C([0,T],H)\times \cdots \times C([0,T],H)}_{\text{$n$ times}},\\
W_{pq}(t_0,T)^{ n}& =
\underbrace{ W_{pq}(t_0,T)\times \cdots \times W_{pq}(t_0,T)}_{\text{$n$ times}},\quad t_0\in [0,T).
\end{align*}

\begin{definition}\label{PDC11}
Let $t_\ast\in [0,T)$.
Denote by $\mathcal{C}_V^{1,1}([t_\ast,T]\times C([0,T],H)^{ n})$  the set of all 
 functions $\varphi\in C([t_\ast,T]\times C([0,T],H)^{ n})$ for which there exist
 functions, called \emph{path derivatives},  
 \begin{align*}
 &\partial_t\varphi\in C([t_\ast,T]\times C([0,T],H)^{ n}),\quad\text{ and}\\
&\partial_{x^j}\varphi\in C([t_\ast,T]\times C([0,T],H)^{ n},H),\quad j=1, \ldots, n,
\end{align*}
  such that the following holds:
 
\begin{enumerate}
\renewcommand{\labelenumi}{(\roman{enumi})}
\renewcommand{\theenumi}{(\roman{enumi})}
\item \label{C11cond1} 
For every  $x=(x^j)_{j=1}^n\in C([0,T],H)^{ n}$,
and $i\in \{1,\ldots,n\}$, 
\begin{align*}
 x^i(t)\in V\quad\text{implies}\quad\partial_{x^i}\varphi (t,x)\in V\text{ a.e.~on $(t_\ast,T)$}.
 \end{align*}
\item For every $t_0\in [t_\ast,T)$, every $x_0\in C([0,T],H)^{ n}$,
every  $\mathbf{x}=(\mathbf{x}^j)_{j=1}^n\in W_{pq}(t_0,T)^{ n}$, 
every $x=(x^j)_{j=1}^n\in C([0,T],H)^{ n}$ with $x=x_0$ on $[0,t_0]$ and
$x=\wx$ a.e.~on $(t_0,T]$,
 and every $t\in [t_0,T]$, the \emph{functional chain rule} holds, i.e., 
\begin{equation}\label{E:ChainRule}
\begin{split}
\varphi(t,x)-\varphi(t_0,x_0)&=\int_{t_0}^t \partial_t\varphi(s,x)+
\sum_{j=1}^n
\langle (\wx^j)^\prime(s),\partial_{x^j} \varphi(s,x)\rangle\,ds.
\end{split}
\end{equation}
\end{enumerate} 
\end{definition}

\begin{remark}
 The path derivatives $\partial_t \varphi$ and $\partial_{x^j} \varphi$ are uniquely determined on $[t_\ast,T)$.
 Indeed, for every $t_0\in [t_\ast, T)$, $x_0\in C([0,T],H)^n$ and $j\in\{1,\ldots,n\}$,
\begin{align}\label{partialt_rep}
\partial_t \varphi(t_0,x_0)&=\lim_{t\downarrow t_0} \frac{1}{t-t_0} 
\left[\varphi(t,x_0(\cdot\wedge t))-\varphi(t_0,x_0)\right], 
\end{align}
and, if additionally $x_0^m(t_0)\in V$ for every $m\in\{1,\ldots,n\}$, then
\begin{align}
 \label{partialx_rep}
(e_i,\partial_{x^j} \varphi(t_0,x_0))&=\lim_{t\downarrow t_0} \frac{1}{t-t_0}
\left[ \varphi(t,x_i)-\varphi(t_0,x_0)\right]-\partial_t \varphi(t_0,x_0).
\end{align}
Here, $\{e_i\}_i$ an orthonormal basis of $H$  with  elements in $V$
and the paths $x_i=(x_i^m)_{m=1}^n \in C([0,T],H)^{ n}$, $i\in\N$, are defined by 
\begin{align*}
x_i^m(t)=\begin{cases}
x_0^j(t) &\text{ if $t\le t_0$ and $m=j$,}\\
(t-t_0)e_i+x_0^j(t_0) &\text{ if $t>t_0$ and $m=j$,}\\
x_0^m(t\wedge t_0) &\text{ if $m\neq j$.}
\end{cases}
\end{align*}
Note that the restrictions $x_i^m\vert_{[t_0,T]}$ are $V$-valued and belong to
$W_{pq}(t_0,T)$.
Since $V$ is dense in $H$ and $\partial_{x^j}\varphi$ is continuous, we can deduce
uniqueness for $\partial_{x^j}\varphi(t_0,x_0)$ also for the general case
(i.e., without requiring $x_0^m(t_0)\in V$ for all $m$).
\end{remark}

\begin{example}
Let $n=1$ and $\varphi(t,x)=\abs{x(t)}^2+\int_0^t \psi(s,x)\,ds$, $t\in [0,T]$, 
$x\in C([0,T],H)$. If $\psi\in C([0,T]\times C([0,T],H))$, then,
by the integration-by-parts formula (Proposition~\ref{P:parts}),
 $\varphi\in\mathcal{C}_V^{1,1}
([0,T]\times C([0,T],H))$ with $\partial_t\varphi(t,x)=\psi(t,x)$
and $\partial_x\varphi(t,x)=2 x(t)$.
\end{example}

\section{Path-dependent Hamilton-Jacobi equations}
\renewcommand{\theequation}
{\textrm{TVP}}
We consider the terminal-value problem
\begin{equation}\label{E:PPDE2}
\begin{split}
&\partial_t u-\langle A(t,x(t)),\partial_x u\rangle+F(t,x,\partial_x u)=0,\,
 (t,x)\in [0,T)\times C([0,T],H),\\
&u(T,x)=h(x),\quad x\in C([0,T],H).
\end{split}
\end{equation}
\setcounter{equation}{0}
\renewcommand{\theequation}{\thesection.\arabic{equation}}
For the functions $h$ and $F$, we impose the following standing hypotheses.

\textbf{H}($h$): The function $h: C([0,T],H)\to\R$ is continuous.

\textbf{H}($F$): The function $F:[0,T]\times C([0,T],H)\times H\to\R$ satisfies the following:

(i) $F$ is continuous.

(ii) There exists an $L_0\ge 0$  such that,
for every $t\in [0,T]$, every $x\in C([0,T],H)$, and every $z$, $\tilde{z}\in H$,
\begin{align*}
\abs{F(t,x,z)-F(t,x,\tilde{z})}&\le L_0(1+\sup_{0\le s\le t} \abs{x(s)})
\abs{z-\tilde{z}},
\end{align*}

(iii) For every $L\ge 0$ and $(t_0,x_0)\in [0,T)\times C([0,T],H)$,
there exists a modulus of continuity $m_{L,t_0,x_0}$ such that,
for every $\eps>0$, every $t\in [t_0,T]$, and every $x$, $y\in\mathcal{X}^L(t_0,x_0)$,
\begin{equation}\label{E:HF3}
\begin{split}
&F(t,x,\eps^{-1} [x(t)-y(t)])-F(t,y,\eps^{-1}[x(t)-y(t)])\\ &\qquad\le
m_{L,t_0,x_0}\left(
\eps^{-1} \abs{x(t)-y(t)}^2+\sup_{s\le t}\abs{x(s)-y(s)}
\right).
\end{split}
\end{equation}

\begin{remark}
(i) It is more common to impose 
the stronger condition
\begin{equation}\label{E:HF3stronger}
\begin{split}
&\abs{F(t,x,z)-F(t,y,z)}
\le \tilde{m}_{L,t_0,x_0}
\left(\sup_{s\le t}\abs{x(s)-y(s)}
\cdot (1+\abs{z})\right),
\end{split}
\end{equation}
where the modulus $\tilde{m}_{L,x_0,t_0}$ is independent from $z\in H$.
For further discussion regarding corresponding conditions in the non-path-dependent case,
 we refer to \cite[Remark~1]{CrandallLions86NA}.

(ii) In contrast to the conditions in \cite{Lukoyanov03a} {\color{black} or \cite{RTZ17}}, which only
allow distributed delays,
we can immediately cover concentrated delays, i.e., $F$ can be of the form
$F(t,x,z)=\tilde{F}(t,x(t/2),z)$ or $F(t,x,z)=\tilde{F}(t,\bfone_{[1,T]}(t)\, x(t-1),z)$
for some function $\tilde{F}$. (See also Remark~\ref{R:equivPseuoMetrics}.)
{\color{black} However, note that \cite{Lukoyanov03a} has been
extended in \cite{Lukoyanov10a} and \cite{Lukoyanov10b} to cover
besides distributed delays also a finite number
of concentrated delays.}
Let us also remark that our treatment of differential games in Section~\ref{S:DiffGames}
 will still require distributed delays.
\end{remark}

\subsection{Minimax solutions: Global version of the notion}
The approach in this and the next subsection is strongly motivated by \cite{Lukoyanov03a},
where the finite-dimensional case is treated.

\begin{definition} Let $u:[0,T]\times\tilde{\Omega}\to\R$ for some
subset $\tilde{\Omega}$ of $C([0,T],H)$.

(i) Let $L\ge 0$. We call $u$  a \emph{minimax $L$-subsolution}
(resp.~\emph{$L$-supersolution}, resp.~\emph{$L$-solution}) of \eqref{E:PPDE2} 
on $[0,T]\times\tilde{\Omega}$ if
\begin{itemize}
 \item $u\in\mathrm{USC}([0,T]\times\tilde{\Omega})$ (resp.~$\mathrm{LSC}([0,T]\times\tilde{\Omega})$,
 resp.~$C([0,T]\times\tilde{\Omega})$), 
 \item $u(T,\cdot)\le\text{ (resp.~$\ge$, resp.~$=$) } h$ on $\tilde{\Omega}$, 
 \item for every
$(t_0,x_0,z)\in [0,T)\times\tilde{\Omega}\times H$, there exists an
$x\in\mathcal{X}^L(t_0,x_0)\cap\tilde{\Omega}$  such that,
for every $t\in [t_0,T]$,
\begin{align*}
u(t_0,x_0)\le \text{ (resp.~$\ge$, resp.~$=$) }\int_{t_0}^t
[(-f^x(s),z)+F(s,x,z)]\,ds+u(t,x).
\end{align*}
\end{itemize}

%

(ii) We call $u$ a minimax subsolution (resp.~supersolution, resp.~solution) 
of \eqref{E:PPDE2} on 
 $[0,T]\times\tilde{\Omega}$ if $u$ is a minimax $L$-subsolution
 (resp.~$L$-supersolution, resp.~$L$-solution) of \eqref{E:PPDE2} on 
 $[0,T]\times\tilde{\Omega}$  for some $L\ge 0$.
\end{definition}

\begin{remark}\label{R:MinimaxSolution}
Let $\tilde{\Omega}\subseteq C([0,T],H)$.
If $u$ is a minimax $L_1$-supersolution of \eqref{E:PPDE2} on $[0,T]\times\tilde{\Omega}$
 for some $L_1\ge 0$, 
then $u$ is a minimax $L$-supersolution of \eqref{E:PPDE2} on $[0,T]\times\tilde{\Omega}$
 for all $L\ge L_1$.
\end{remark}

\begin{remark}\label{R:MinimaxRestriction}
Fix $L\ge 0$. Let $\Omega^L\subseteq\Omega_1\subseteq\Omega_2\subseteq
C([0,T],H)$. If $u$ is a minimax $L$-solution of \eqref{E:PPDE2} on $[0,T]\times\Omega_2$,
then its restriction $u\vert_{[0,T]\times\Omega_1}$ is a minimax $L$-solution 
of \eqref{E:PPDE2} on $[0,T]\times\Omega_1$.
However, a minimax solution of \eqref{E:PPDE2} on $[0,T]\times\Omega_2$ is not necessarily a
minimax solution  of \eqref{E:PPDE2} on $[0,T]\times\Omega_1$.
\end{remark}

\begin{remark}\label{R:Lminimax}
Let $L\ge 0$.
A minimax solution of \eqref{E:PPDE2}
 on $[0,T]\times\Omega^L$ is a minimax $L$-solution of \eqref{E:PPDE2}
 on $[0,T]\times\Omega^L$.
\end{remark}

\begin{lemma}\label{L:EquivUpperSol}
Fix $L\ge 0$ and a set $\tilde{\Omega}$ with
$\Omega^L\subseteq\tilde{\Omega}\subseteq C([0,T],H)$.
Consider a function $u\in\mathrm{USC}([0,T]\times\tilde{\Omega})$ 
(resp.~$\mathrm{LSC}([0,T]\times\tilde{\Omega})$, resp.~$C([0,T]\times\tilde{\Omega})$).
Let $(t_0,x_0,z)\in [0,T)\times\tilde{\Omega}\times H$
Then the following statements are equivalent:

(i) There exists an $x\in\mathcal{X}^L(t_0,x_0)$ 
such that,
for every $t\in [t_0,T]$, 
\begin{align}\label{E:lowerSol}
u(t_0,x_0)\le\text{ (resp.~$\ge$, resp.~$=$) }\int_{t_0}^t
[(-f^x(s),z)+F(s,x,z)]\,ds+u(t,x).
\end{align}

(ii) For every $t\in [t_0,T]$, there exists 
 an $x\in\mathcal{X}^L(t_0,x_0)$ 
 for which  \eqref{E:lowerSol} holds.
 \end{lemma}

\begin{proof} [Proof of  Lemma~\ref{L:EquivUpperSol}] 
We show the statement only in the case  $u\in\mathrm{USC}([0,T]\times\tilde{\Omega})$.

Clearly, (i) implies (ii). Next, suppose that (ii) holds.
To  show  (i), let
 $(\pi^m)_{m\in\N}=((t_i^m)_{i=1,\ldots,n(m)})_{m\in\N}$ be
dyadic partitions of $[t_0,T]$
with $t_0=t_0^m<t_1^m<\cdots<t_{n(m)}^m=T$. For each $m\in\N$, we define an
$x^m\in\mathcal{X}^L(t_0,x_0)$ 
as follows: 
First, note that, by (ii) there exists an $x_1^m\in\mathcal{X}^L(t_0,x_0)$ 
such
that \eqref{E:lowerSol} holds with $t=t_1^m$ and $x=x^m$.
Next, assuming that $x^m_i$ has already been chosen for some $i\in\{1,\ldots,n(m)-1\}$,
note that, by (ii), there exists an $x_{i+1}^m\in\mathcal{X}^L(t_i^m,x_i^m)$ 
such that \eqref{E:lowerSol} holds with $(t_0,x_0)$ replaced by $(t_i^m,x_i^m)$
and with $t=t_{i+1}^m$ and $x=x_{i+1}^m$. Put
\begin{align*}
x^m:=x_0.\bfone_{[0,t_0]}+\sum\nolimits_{i=1}^{n(m)} x_i^m.\bfone_{(t_{i-1}^m,t_i^m]}.
\end{align*}
By construction, $x^m\in\mathcal{X}^L(t_0,x_0)$ 
and $x^m$ satisfies \eqref{E:lowerSol}
for all $t\in\{t_i^m\}_{i=0}^{m(n)}$. Since $\mathcal{X}^L(t_0,x_0)$ is compact
and since $\cup_m \{t_i^m\}_{i=0}^{m(n)}$ is dense in $[t_0,T]$,
the sequence $(x^m)_m$ has a subsequence $(x^{m_k})_k$ that converges to some
$x^0$ in $\mathcal{X}^L(t_0,x_0)$. 
Since $u$ is upper semi-continuous,
\begin{align*}
u(t_0,x_0)&\le\varlimsup_k\left[ \int_{t_0}^t
(-f^{x^{m_k}}(s),z)+F(s,x^{m_k},z)\,ds+u(t,x^{m_k})\right]\\
&\le \int_{t_0}^t
(-f^{x_{0}}(s),z)+F(s,x^{0},z)\,ds+u(t,x^{0})
\end{align*}
for every $t\in [t_0,T]$.
\end{proof}

\begin{proposition} \label{P:MinimaxUpperLower}
Fix $L\ge 0$ and a set $\tilde{\Omega}$ with
$\Omega^L\subseteq\tilde{\Omega}\subseteq C([0,T],H)$.
Then a mapping $u:[0,T]\times\tilde{\Omega}\to\R$ 
is a minimax $L$-solution of \eqref{E:PPDE2} on $[0,T]\times\tilde{\Omega}$
 if and only if
$u$ is a minimax $L$-super- as well as a minimax $L$-subsolution of \eqref{E:PPDE2}
on $[0,T]\times\tilde{\Omega}$.
\end{proposition}

\begin{proof}
If $u$ is a minimax $L$-solution of \eqref{E:PPDE2} on $[0,T]\times\tilde{\Omega}$, then it follows
immediately from the definition that $u$ is a also a minimax $L$-super- and a 
minimax $L$-subsolution
 of \eqref{E:PPDE2} on $[0,T]\times\tilde{\Omega}$.

Let now $u$ be a minimax $L$-super-
as well as a minimax $L$-subsolution of \eqref{E:PPDE2} on $[0,T]\times\tilde{\Omega}$.
Clearly, $u$ is continuous and $u(T,\cdot)=h$. Thus it suffices, by Lemma~\ref{L:EquivUpperSol},
to show that, for every $(t_0,x_0,z)\in [0,T)\times\tilde{\Omega}\times H$ and
every $t\in (t_0,T]$, there exists an $(x,y)\in\mathcal{Y}^L(t_0,x_0,z)$ (see \eqref{E:Y(t,x,z)}
for the definition of  $\mathcal{Y}^L(t_0,x_0,z)$) such
that $u(t,x)-y(t)=u(t_0,x_0)$.
To this end, fix $(t_0,x_0,z)\in [0,T)\times\tilde{\Omega}\times H$ and let $t\in (t_0,T]$.
Since $u$ is 
a minimax $L$-super- and a 
minimax $L$-subsolution
 of \eqref{E:PPDE2} on $[0,T]\times\tilde{\Omega}$,
 there exist
$(x_1,y_1)$, $(x_2,y_2)\in\mathcal{Y}^L(t_0,x_0,z)$ such that
\begin{align*}
u(t,x_1)-u(t_0,x_0)-y_1(t)\ge 0\ge u(t,x_2)-u(t_0,x_0)-y_2(t).
\end{align*}
Since $\mathcal{Y}^L(t_0,x_0,z)$ is connected in $C([0,T],H)\times C([t_0,T])$, 
which follows from
Lemma~\ref{L:Connected}, and the mapping
\begin{align*}
(x,y)\mapsto u(t,x)-u(t_0,x_0)-y(t),\quad \tilde{\Omega}\times C([t_0,T])\to\R,
\end{align*}
is continuous, there exists an $(x,y)\in \mathcal{Y}^L(t_0,x_0,z)$ such that
\begin{align*}
u(t,x)-u(t_0,x_0)-y(t)=0.
\end{align*}
This concludes the proof.
\end{proof}

\subsection{Minimax solutions: Infinitesimal version of the notion}

\begin{proposition}\label{P:Dini}
Fix $L\ge 0$ and some set $\tilde{\Omega}$ with
$\Omega^L\subseteq\tilde{\Omega}\subseteq C([0,T],H)$.

(i) Let  $u\in\mathrm{LSC}([0,T]\times\tilde{\Omega})$. Then $u$ is a
minimax $L$-supersolution of  \eqref{E:PPDE2} on $[0,T]\times\tilde{\Omega}$
 if and only if $u(T,\cdot)\ge h$  on $\tilde{\Omega}$ and,
for every $(t_0,x_0,z)\in [0,T)\times\tilde{\Omega}\times H$,
\begin{equation}\label{E:Dini}
\begin{split}
&\inf_{x\in\mathcal{X}^L(t_0,x_0)} 
\varliminf_{\delta\downarrow 0} 
\Bigl[
u(t_0+\delta,x)-u(t_0,x_0)\\&\qquad\qquad\qquad\qquad
+\int_{t_0}^{t_0+\delta} (-f^x(s),z)+F(s,x,z)\,ds
\Bigr]\delta^{-1}\le 0.
\end{split}
\end{equation}

(ii)  Let $u\in\mathrm{USC}([0,T]\times\tilde{\Omega})$ Then $u$ is a
minimax $L$-subsolution of  \eqref{E:PPDE2} on $[0,T]\times\tilde{\Omega}$
 if and only if $u(T,\cdot)\le h$ on $\tilde{\Omega}$ and,
for every $(t_0,x_0,z)\in [0,T)\times\tilde{\Omega}\times H$,
\begin{equation*}
\begin{split}
&\sup_{x\in\mathcal{X}^L(t_0,x_0)} 
\varlimsup_{\delta\downarrow 0} 
\Bigl[
u(t_0+\delta,x)-u(t_0,x_0)\\&\qquad\qquad\qquad\qquad
+\int_{t_0}^{t_0+\delta} (-f^x(s),z)+F(s,x,z)\,ds
\Bigr]\delta^{-1}\ge 0.
\end{split}
\end{equation*}
\end{proposition}

\begin{remark} See \cite[Theorem~8.1]{Lukoyanov03a} for
a corresponding result in the finite-di\-men\-sio\-nal case,
where the outer {\color{black}infimum resp.~supremum} is taken over a
slightly different  class of paths $x$. Our choice however simplifies 
the proof and does not weaken our result because
the compactness of $\mathcal{X}^L(t_0,x_0)$ implies that, for every $\eps>0$,
there exists a $\delta>0$ such that, for every $x\in\mathcal{X}^L(t_0,x_0)$,
\begin{align*}
\abs{f^x(t)}\le L(1+\sup_{s\le t_0} \abs{x_0(s)})+\eps\text{ for a.e.~$t\in (0,t_0+\delta)$.}
\end{align*}
\end{remark}

\begin{proof}[Proof of Proposition~\ref{P:Dini}]
We prove only the non-obvious direction of part~(i), i.e., we assume that,
for every $(t_0,x_0,z)\in [0,T)\times\tilde{\Omega}\times H$, we have \eqref{E:Dini}.
For the sake of a contradiction, we assume that $u$ is not a minimax $L$-supersolution
of \eqref{E:PPDE2} on $[0,T]$, i.e., by Lemma~\ref{L:EquivUpperSol}, there exist
 $(t_0,x_0,z)\in [0,T)\times\tilde{\Omega}\times H$, $t_1\in (t_0,T]$ 
such that, for every $x\in\mathcal{X}^L(t_0,x_0)$, 
\begin{align*}
u(t_1,x)-u(t_0,x_0)+\int_{t_0}^{t_1} (-f^x(s),z)+F(s,x,z)\,ds> 0.
\end{align*}
Furthermore, there exists a $\tilde{\delta}>0$ such that, for every 
$x\in\mathcal{X}^L(t_0,x_0)$, 
\begin{align*}
u(t_1,x)-u(t_0,x_0)+\int_{t_0}^{t_1} (-f^x(s),z)+F(s,x,z)\,ds>  \tilde{\delta}
\end{align*}
because otherwise lower semi-continuity of $u$, compactness of $\mathcal{X}(t_0,x_0)$,
and continuity of $F$ would lead to a contradiction.

Denote by $\tilde{t}$ the supremum of all $t\in [t_0,t_1]$ such that
\begin{equation}\label{E:Dini2}
\begin{split}
&\min_{x\in\mathcal{X}^L(t_0,x_0)}\left\{
u(t,x)-u(t_0,x_0)+\int_{t_0}^t (-f^x(s),z)+F(s,x,z)\,ds
\right\} \\&\qquad\qquad\qquad \le \frac{t-t_0}{t_1-t_0}\cdot\tilde{\delta}.
\end{split}
\end{equation}
Clearly $\tilde{t}\in [t_0,t_1)$ and, due to the compactness of $\mathcal{X}^L(t_0,x_0)$,
semicontinuity of $u$, and continuity of $F$, the supremum $\tilde{t}$ is
actually a maximum, i.e.,
there exists an $\tilde{x}\in\mathcal{X}^L(t_0,x_0)$ such that
\begin{equation}\label{E:Dini3}
u(\tilde{t},\tilde{x})-u(t_0,x_0)+\int_{t_0}^{\tilde{t}} (-f^{\tilde{x}}(s),z)+F(s,\tilde{x},z)]\,ds
\le \frac{\tilde{t}-t_0}{t_1-t_0}\cdot\tilde{\delta}.
\end{equation}
We shall use \eqref{E:Dini} to derive 
the existence of a time $t\in (\tilde{t},t_1)$ satisfying \eqref{E:Dini2}, which
would be a contradiction.  By \eqref{E:Dini}, there exist an $x\in\mathcal{X}^L(\tilde{t},\tilde{x})$
and a $\delta\in (0,t_1-\tilde{t})$ such that
\begin{equation}\label{E:Dini4}
u(\tilde{t}+\delta,x)-u(\tilde{t},\tilde{x})+\int_{\tilde{t}}^{\tilde{t}+\delta} 
(-f^x(s),z)+F(s,x,z)\,ds\le \frac{\delta}{t_1-t_0}\cdot\tilde{\delta}.
\end{equation}
Adding \eqref{E:Dini3} and \eqref{E:Dini4}, we see that
\begin{align*}
u(\tilde{t}+\delta,x)-u(t_0,x_0)+\int_{t_0}^{\tilde{t}+\delta} 
(-f^x(s),z)+F(s,x,z)\,ds\le \frac{(\tilde{t}+\delta)-t_0}{t_1-t_0} \cdot\tilde{\delta},
\end{align*}
which concludes the proof.
\end{proof}

\section{Comparison}
Given $L\ge 0$,  a set $\tilde{\Omega}$ with $\Omega^{L}
\subseteq\tilde{\Omega}\subseteq C([0,T],H)$, a non-anticipating function
$w:[0,T]\times\tilde{\Omega}\times\tilde{\Omega}\to\R$, and
 $(t_0,x_0,y_0)\in [0,T)\times\tilde{\Omega}\times\tilde{\Omega}$, put
 \begin{align*}
&\underline{\mathcal{A}}^L w(t_0,x_0,y_0):=\{\varphi\in\mathcal{C}_V^{1,1}
([t_0,T]\times
C([0,T],H)\times C([0,T],H)):\exists T_0\in (t_0,T]:\\
&\qquad 0=(\varphi-w)(t_0,x_0,y_0)=\inf_{(t,x,y)\in [t_0,T_0]\times
\mathcal{X}^{L}(t_0,x_0)
\times \mathcal{X}^{L}(t_0,y_0)} [(\varphi-w)(t,x,y)]\}.
\end{align*}

\begin{definition}\label{D:Viscsub_doubled}
Fix $L\ge 0$ and a set $\tilde{\Omega}$ with   
$\Omega^{L}\subseteq\tilde{\Omega}\subseteq C([0,T],H)$.
A function $w\in\mathrm{USC}([0,T]\times\tilde{\Omega}\times\tilde{\Omega})$ 
 is a \emph{viscosity  $L$-subsolution} of 
\begin{equation}\label{E:DPPDE} 
\begin{split}
&\partial_t w(t,x,y)-\langle A(t,x(t)),\partial_x w(t,x,y)\rangle
-\langle A(t,y(t)),\partial_y w(t,x,y)\rangle\\ 
&\qquad +F(t,x,\partial_x w(t,x,y)) - F(t,y,-\partial_y w(t,x,y))=0
\end{split}
\end{equation}
on $[0,T]\times\tilde{\Omega}\times\tilde{\Omega}$
 if, for every  $(t_0,x_0,y_0)\in [0,T)\times\tilde{\Omega}\times\tilde{\Omega}$ and
  for every test function
 $\varphi\in\underline{\mathcal{A}}^L w(t_0,x_0,y_0)$, 
there exists a pair $(x,y)\in \mathcal{X}^{L}(t_0,x_0)
\times \mathcal{X}^{L}(t_0,y_0)$  such that
\begin{equation}\label{E:DViscSub}
\begin{split}
&\partial_t\varphi(t_0,x_0,y_0)\\ &\qquad
+\varlimsup_{\delta\downarrow 0} \frac{1}{\delta} \Biggl[\int_{t_0}^{t_0+\delta}
-\langle A(t,\wx(t)),\partial_x  \varphi(t,x,y)\rangle
-\langle A(t,\wy(t)),\partial_y  \varphi(t,x,y)\rangle\,dt \Biggr]\\
&\qquad
+F(t_0,x_0,\partial_x\varphi(t_0,x_0,y_0))-F(t_0,y_0,-\partial_y\varphi(t_0,x_0,y_0))\ge 0.
\end{split}
\end{equation}
\end{definition}

\begin{theorem} [Comparison~I] \label{T:Comparison1}
Fix $L\ge 0$ and a set $\tilde{\Omega}$ with   
$\Omega^{L}\subseteq\tilde{\Omega}\subseteq C([0,T],H)$.
Let $w$ be a viscosity $L$-subsolution of \eqref{E:DPPDE} on 
$[0,T]\times\tilde{\Omega}\times\tilde{\Omega}$
with $w(T,x,x)\le 0$ for all $x\in\tilde{\Omega}$.
Then $w(t,x,x)\le 0$ for all $(t,x)\in [0,T]\times\tilde{\Omega}$.
\end{theorem}

\begin{proof}
Assume that there exists a $(t_0,x_0)\in [0,T)\times\tilde{\Omega}$ such that
\begin{align*}
M_0:=w(t_0,x_0,x_0)>0.
\end{align*}

For every $\eps>0$, define a mapping
$\Phi_\eps:[t_0,T]\times\tilde{\Omega}\times\tilde{\Omega}\to\R$ by
\begin{align*}
\Phi_\eps(t,x,y):=w(t,x,y)-\Theta_\eps(t,x,y),
\end{align*}
where
\begin{align*}
\Theta_\eps(t,x,y)&:=\frac{T-t}{T-t_0}\cdot\frac{M_0}{2} +\frac{1}{\eps} \Psi(t,x,y),\\
\Psi(t,x,y)&:= \abs{x(t)-y(t)}^2+\int_{t_0}^t \abs{x(s)-y(s)}^2\,ds,\\
&\qquad\qquad\qquad\qquad\qquad\qquad\qquad t\in [t_0,T],\quad x,y\in C([0,T],H).
\end{align*}
Since $\Phi_\eps\in\mathrm{USC}([t_0,T]\times\tilde{\Omega}\times\tilde{\Omega})$
 and  $\mathcal{X}^{L}(t_0,x_0)$ is compact,
there exists a point 
\begin{align*}
k_\eps:=(t_\eps,x_\eps,y_\eps)\in [t_0,T]\times
 \mathcal{X}^{L}(t_0,x_0)\times \mathcal{X}^{L}(t_0,x_0)=:K
\end{align*}
such that
\begin{align*}
M_\eps:=\Phi_\eps(k_\eps)=\max_{k\in K}\Phi_\eps(k)\ge 
\Phi_\eps(t_0,x_0,x_0)=\frac{M_0}{2}>0.
\end{align*}
Following the proof of \cite[Proposition~3.7]{UserGuide}, we can deduce that
$\lim_{\eps\to 0} \frac{1}{\eps}\Psi(k_\eps)=0$ and that 
\begin{equation}
\label{E:VC2}
\begin{split}
\Psi(\hat{k})=0,\quad
\lim_{\eps\to 0} M_\eps&=w(\hat{t},\hat{x},\hat{y})-\frac{T-\hat{t}}{T-t_0}\cdot\frac{M_0}{2}\\ 
&=\max_{(t,x,y)\in  K\cap\Psi^{-1}\{0\}}\left[w(t,x,y)-\frac{T-t}{T-t_0}\cdot\frac{M_0}{2}\right]\\
&=\max_{(t,x)\in [t_0,T]\times\mathcal{X}^L(t_0,x_0)}
\left[w(t,x,x)-\frac{T-t}{T-t_0}\cdot\frac{M_0}{2}\right]
\end{split}
\end{equation}
for every limit point $\hat{k}=(\hat{t},\hat{x},\hat{y})$ of $\{k_\eps\}$  in $K$ as $\eps\to 0$. Moreover,
$\hat{x}=\hat{y}$ on $[0,\hat{t}]$.
Note that the last equality in \eqref{E:VC2} follows from $w$ being non-anticipating because
$\Psi(t,x,y)=0$ together with $x$, $y\in\mathcal{X}^L(t_0,x_0)$ yields $x=y$ on $[0,t]$.
Now, from \eqref{E:VC2} we can infer that  the maxima $M_\eps$ are attained at  interior points,
i.e., we have $t_\eps <T$ for sufficiently small $\eps$; in particular, 
$w(T,x,x)-(T-t_0)^{-1}(T-T)M_0/2\le 0<M_0/2$ yields 
$\hat{t}<T$.

Next, we construct test functions. 
For every $\eps>0$, define a function $\varphi_\eps: [t_\eps,T]\times
C([0,T],H)\times C([0,T],H)
\to\R$ by
$\varphi_\eps(t,x,y):=\Theta_\eps(t,x,y)-M_\eps$.
Clearly, $\varphi_\eps\in\underline{\mathcal{A}}^L w(k_\eps)$ with
\begin{align*}
\partial_t \varphi_\eps(t,x,y)&=-\frac{M_0}{2(T-t_0)}+\frac{1}{\eps}\abs{x(t)-y(t)}^2,\\
\partial_x \varphi_\eps(t,x,y)&=-\partial_y \varphi_\eps(t,x,y)=\frac{2}{\eps} [x(t)-y(t)].
\end{align*}
Hence, there exists an $(x,y)\in\mathcal{X}^L(t_\eps,x_\eps)\times\mathcal{X}^L(t_\eps,y_\eps)$
such that, for small $\eps$,
\begin{equation}\label{E:CompFinal}
\begin{split}
0&\le -\frac{M_0}{2(T-t_0)}+
\frac{1}{\eps}\abs{x_\eps(t_\eps)-y_\eps(t_\eps)}^2\\
&\qquad +\varlimsup_{\delta\downarrow 0}\frac{1}{\delta}
\int_{t_\eps}^{t_\eps+\delta} -\frac{2}{\eps}\underbrace{\langle
A(t,\wx(t))-A(t,\wy(t)), x(t)-y(t)
\rangle}_{\ge 0}\,dt\\
&\qquad +F(t_\eps,x_\eps,\eps^{-1} 2[x_\eps(t_\eps)-y_\eps(t_\eps))
-F(t_\eps,y_\eps,\eps^{-1} 2[x_\eps(t_\eps)-y_\eps(t_\eps)])\\
&{\color{black}\le  -\frac{M_0}{2(T-t_0)}+
\frac{1}{\eps}\abs{x_\eps(t_\eps)-y_\eps(t_\eps)}^2}
\\&\qquad {\color{black}+m_{L,t_0,x_0}\left(
\eps^{-1} \abs{x_\eps(t_\eps)-y_\eps(t_\eps)}^2+\sup_{s\le t_\eps}\abs{x_\eps(s)-y_\eps(s)}\right).}
\end{split}
\end{equation}
{\color{black} Here, we obtained the last inequality from condition \textbf{H}($F$)(iii).
Now, fix some limit point $\hat{k}=(\hat{t},\hat{x},\hat{y})$ of $\{k_\eps\}$. 
Then $\hat{k}\in K$ and, with slight abuse of notation, we can assume that 
$(t_\eps,x_\eps,y_\eps)\to \hat{k}$ as $\eps\downarrow 0$.
Recall  also that $\lim_{\eps\to 0} \frac{1}{\eps}\Psi(k_\eps)=0$
and
 that $\hat{x}=\hat{y}$ on $[0,\hat{t}]$.
Thus,
letting $\eps\downarrow 0$ in \eqref{E:CompFinal} 
yields a contradiction.
(Note that $\sup_{s\le t_\eps}\abs{x_\eps(s)-y_\eps(s)} \to 0$ as $\eps\downarrow 0$ can also be deduced 
from Remark~\ref{R:equivPseuoMetrics} together with $\lim_{\eps\to 0} \frac{1}{\eps}\Psi(k_\eps)=0$.
For more details, we refer the reader to  \eqref{E:CompactEndProof} 
in the proof of Lemma~\ref{L:CompactW}.)}
\end{proof}

\begin{theorem}[Doubling theorem]\label{T:MinimaxDoubling}
Fix $L\ge 0$ and a set $\tilde{\Omega}$ with   $\Omega^L\subseteq\tilde{\Omega}\subseteq C([0,T],H)$.
Let $u$ be a minimax $L$-sub- and let $v$ be a minimax $L$-supersolution of \eqref{E:PPDE2}
on $[0,T]\times\tilde{\Omega}$.
Then  $w$ defined by $w(t,x,y):=u(t,x)-v(t,y)$ is a viscosity $L$-subsolution of \eqref{E:DPPDE}
on $[0,T]\times\tilde{\Omega}\times\tilde{\Omega}$.
\end{theorem}

\begin{proof}
Fix $(t_0,x_0,y_0)\in [0,T)\times\tilde{\Omega}\times\tilde{\Omega}$.
Let $\varphi\in\underline{\mathcal{A}}^L(t_0,x_0,y_0)$, i.e., there exists
a $T_0\in (t_0,T]$ such that,  for all $(t,x,y)\in [t_0,T_0]\times\mathcal{X}^L(t_0,x_0)
\times\mathcal{X}^L(t_0,y_0)$,
\begin{equation}\label{E:DMinimaxIsVisc1}
\begin{split}
\varphi(t,x,y)-\varphi(t_0,x_0,y_0)&\ge w(t,x,y)-w(t_0,x_0,y_0)\\ &=
[u(t,x)-u(t_0,x_0)]-[v(t,y)-v(t_0,y_0)].
\end{split}
\end{equation}
Put $z:=\partial_x\varphi(t_0,x_0,y_0)$ and $\zeta:=\partial_y\varphi(t_0,x_0,y_0)$.
Since  $u$ is a  minimax $L$-sub- 
and $v$ is a minimax $L$-supersolution of \eqref{E:PPDE2}
on $[0,T]\times\tilde{\Omega}$, there exists
an $(x,y)\in \mathcal{X}^L(t_0,x_0)\times \mathcal{X}^L(t_0,y_0)$ such that, for all $t\in [t_0,T_0]$,
\begin{equation}\label{E:DMinimaxIsVisc2}
\begin{split}
 u(t,x)-u(t_0,x_0)&\ge \int_{t_0}^t (f^x(s),z)
 -F(s,x,z)\,ds,\\
  v(t,y)-v(t_0,y_0)&\le \int_{t_0}^t (f^y(s),-\zeta)
 -F(s,y,-\zeta)\,ds.
\end{split}
\end{equation}
Thus, by the functional chain-rule applied to  \eqref{E:DMinimaxIsVisc1}
and by \eqref{E:DMinimaxIsVisc2},
\begin{align*}
&\int_{t_0}^t [\partial_t\varphi(s,x,y)-\langle A(s,\wx(s)),\partial_x \varphi(s,x,y)\rangle
 -\langle  A(s,\wy(s)),\partial_y \varphi(s,x,y)\rangle\\&\qquad\qquad\qquad
+(f^x(s),\partial_x \varphi(s,x,y))+(f^y(s),\partial_y \varphi(s,x,y)]\,ds\\ 
&\qquad \ge\int_{t_0}^t (f^x(s),z)+(f^y(s),\zeta)-F(s,x,z)+F(s,y,-\zeta)\,ds,
\end{align*}
which 
yields \eqref{E:DViscSub}, i.e., $w$ is a viscosity $L$-subsolution of \eqref{E:DPPDE} on 
$[0,T]\times\tilde{\Omega}\times\tilde{\Omega}$.
\end{proof}

Immediately, from the doubling theorem (Theorem~~\ref{T:MinimaxDoubling}) together with
Theorem~\ref{T:Comparison1}, we obtain a comparison principle for 
minimax $L$-semisolutions.

\begin{theorem}[Comparison~II] \label{T:Comparison2}
Fix  $L\ge 0$ and a set $\tilde{\Omega}$ with   $\Omega^L\subseteq\tilde{\Omega}\subseteq C([0,T],H)$.
Let $u$ be a minimax $L$-sub- and let $v$ be a minimax $L$-supersolution of \eqref{E:PPDE2}
on $[0,T]\times\tilde{\Omega}$.
Then $u\le v$ on $[0,T]\times\tilde{\Omega}$.
\end{theorem}

\begin{theorem}[Comparison~III] \label{T:Comparison3}
Fix  a set $\tilde{\Omega}$ with   $\Omega\subseteq\tilde{\Omega}\subseteq C([0,T],H)$.
Let $u$ be a minimax sub- and let $v$ be a minimax supersolution of \eqref{E:PPDE2}
on $[0,T]\times\tilde{\Omega}$.
Then $u\le v$ on $[0,T]\times\tilde{\Omega}$.
\end{theorem}

\begin{proof}
There exist $L_1$, $L_2\ge 0$ such that $u$ is a minimax 
$L_1$-sub- and $v$ is a minimax $L_2$-supersolution of \eqref{E:PPDE2}
on $[0,T]\times\tilde{\Omega}$. By Remark~\ref{R:MinimaxSolution}, we can 
replace $L_1$ and $L_2$ with $L_1\vee L_2$.
Since $\Omega=\cup_{L\ge 0}\Omega^L\subseteq\tilde{\Omega}$, we have
$\Omega^{L_1\vee L_2}\subseteq\tilde{\Omega}$. Thus we can apply
Theorem~\ref{T:MinimaxDoubling} to deduce that
the function $w$ defined by $w(t,x,y):=u(t,x)-v(t,y)$ is a viscosity $(L_1\vee L_2)$-subsolution
of \eqref{E:DPPDE} on 
$[0,T]\times\tilde{\Omega}$. By Theorem~\ref{T:Comparison1}, $u\le v$ on $[0,T]\times\tilde{\Omega}$.
\end{proof}
\section{Existence and uniqueness}
Thanks to our previous results concerning the topological properties of
the spaces $\mathcal{X}^L(t_0,x_0)$, we are able to follow the approach
of Lukoyanov  \cite[Section~7]{Lukoyanov03a}
(see also Subbotin \cite[Section~8, pp.~69-80]{Subbotin_book} for the non-path-dependent case).

Given $(t_0,x_0,z)\in [0,T)\times\Omega\times H$ and $L\ge 0$, put
\begin{equation}\label{E:Y(t,x,z)}
\begin{split}
\mathcal{Y}^L(t_0,x_0,z)&:=\{(x,y)\in\mathcal{X}^L(t_0,x_0)\times C([t_0,T]):\\
&\qquad\qquad y(t)=\int_{t_0}^t (f^x(s),z)-F(s,x,z)\,ds\text{ on $[t_0,T]$}\}.
\end{split}
\end{equation}

\begin{definition}
Let $\tilde{\Omega}\subseteq C([0,T],H)$ and put
$S(t_0):=\tilde{\Omega}\times C([t_0,T])$, $t_0\in [0,T)$.
Consider a non-anticipating function
 $u:[0,T]\times\tilde{\Omega}\to [-\infty,\infty]$.

(i) We say that $u$ is an \emph{improper minimax $L$-supersolution} of  \eqref{E:PPDE2}
on $[0,T]\times\tilde{\Omega}$  
if $u(T,\cdot)\ge h$ on $\tilde{\Omega}$ and if, for every $(t_0,x_0,z)\in [0,T)\times\tilde{\Omega}\times H$
and every $t\in (t_0,T]$,
\begin{align*}
\inf_{(x,y)\in\mathcal{Y}^L(t_0,x_0,z)\cap S(t_0)} \left\{
u(t,x)-y(t)
\right\}\le u(t_0,x_0).
\end{align*}

(ii) We say that $u$ is an \emph{improper  minimax $L$-subsolution} of  \eqref{E:PPDE2}
on $[0,T]\times\tilde{\Omega}$
if $u(T,\cdot)\le h$ on $\tilde{\Omega}$ and if, for every $(t_0,x_0,z)\in [0,T)\times\tilde{\Omega}\times H$
and every $t\in (t_0,T]$,
\begin{align*}
\sup_{(x,y)\in\mathcal{Y}^L(t_0,x_0,z)\cap S(t_0)} \left\{
u(t,x)-y(t)
\right\}\ge u(t_0,x_0).
\end{align*}
\end{definition}

Note that we use here $\inf$ and $\sup$ instead of $\min$ and $\max$
because we do not require any semicontinuity for $u$.

Next, we show that the set of improper minimax $L_0$-supersolutions of \eqref{E:PPDE2}
is not empty. To this end, define  functions
$u_+^z:[0,T]\times\Omega\to [-\infty,\infty]$, $z\in H$, by
\begin{align*}
u_+^z(t_0,x_0):=\max_{(x,y)\in\mathcal{Y}^{L_0}(t_0,x_0,z)}
\left\{ h(x)-y(T)\right\}.
\end{align*}

\begin{lemma}\label{L:ImUpper}
Let $z\in H$. The functional $u_+^z$ is an improper minimax $L_0$-supersolution of \eqref{E:PPDE2}
on $[0,T]\times\Omega$
with values in $[-\infty,\infty)$.
\end{lemma}

\begin{proof}
First, note that $u^z_+(T,\cdot)=h$.

Next, fix $(t_0,x_0,\tilde{z})\in [0,T)\times\Omega\times H$. 
By Proposition~\ref{P:EU_ODE}, the initial value problem
\begin{align*}
\tilde{\wx}^\prime(t)+A(t,\tilde{\wx}(t))=\tilde{f}(t,\tilde{x})\quad\text{a.e.~on $(t_0,T)$,}\quad
\tilde{x}=x_0\quad\text{on $[0,t_0]$,}
\end{align*}
where $\tilde{f}:[t_0,T]\times C([0,T],H)\to H$ is defined by
\begin{align*}
\tilde{f}(t,x):=\begin{cases}
\frac{F(t,x,\tilde{z})-F(t,x,z)}{\abs{\tilde{z}-z}^2}\cdot (\tilde{z}-z)&\text{ if $\tilde{z}\neq z$,}\\
0 &\text{ if $\tilde{z}=z$,}
\end{cases}
\end{align*}
has a unique solution $\tilde{x}\in\mathcal{X}^{L_0}(t_0,x_0)$.
Moreover, put 
\begin{align*}
\tilde{y}(t)&:=\int_{t_0}^t (\tilde{f}(s,\tilde{x}),\tilde{z})-F(s,\tilde{x},\tilde{z})\,ds,\quad t\in [t_0,T].
\end{align*}
Then $(\tilde{x},\tilde{y})\in\mathcal{Y}^{L_0}(t_0,x_0,z)\cap\mathcal{Y}^{L_0}(t_0,x_0,\tilde{z})$.
Since, for every $t\in (t_0,T]$,
\begin{align*}
u^z_+(t,\tilde{x})-\tilde{y}(t)&=
\max_{x\in\mathcal{X}^{L_0}(t,\tilde{x})}\left\{
h(x)-\int_t^T [(f^x(s),z)-F(s,x,z)]\,ds
\right\}\\ &\qquad\qquad -
\int_{t_0}^t [\underbrace{(f^{\tilde{x}}(s),\tilde{z})-F(s,\tilde{x},\tilde{z})}_{
=(f^{\tilde{x}}(s),{z})-F(s,\tilde{x},{z})
}]\,ds\\
&\le \max_{x\in\mathcal{X}^{L_0}(t_0,x_0)}\left\{
h(x)-\int_{t_0}^T [(f^x(s),z)-F(s,x,z)]\,ds
\right\}\\
&=u^z_+(t_0,x_0),
\end{align*}
we can conclude that $u^z_+$ is an improper minimax $L_0$-supersolution of \eqref{E:PPDE2}
on $[0,T]\times\Omega$.
The pointwise boundedness from above ($u^z_+<\infty$) follows from the
continuity of $h$ and $F$ in $x_0$ and from the compactness of the sets
$\mathcal{X}^{L_0}(t_0,x_0)$ (Proposition~\ref{P:compact}).
\end{proof}

Next, we will show that  the functions 
$u^L_0:[0,T]\times\Omega^L\to [-\infty,\infty]$, $L\ge L_0$,
defined by
\begin{align*}
u_0^L(t_0,x_0)&:=\inf\Bigl\{
u(t_0,x_0):\\ \qquad\qquad\qquad\qquad
 &\text{$u$ improper minimax $L_0$-supersolution of \eqref{E:PPDE2} on $[0,T]\times\Omega^L$}
\Bigr\},
\end{align*}
are both,  improper minimax $L_0$-super- as well as a
improper minimax $L_0$-subsolutions of \eqref{E:PPDE2} on $[0,T]\times\Omega^L$, $L\ge L_0$.

\begin{lemma}
Let $L\ge L_0$. 
Then the function $u_0^L$ is an improper minimax $L_0$-super 
as well as an improper minimax $L_0$-subsolution of \eqref{E:PPDE2}
on $[0,T]\times\Omega^L$.
Moreover, $u_0$ is real-valued.
\end{lemma}

\begin{proof}
Fix $(t_0,x_0,z)\in [0,T)\times\Omega^L\times H$.
As every improper minimax $L_0$-supersolution $u$ of \eqref{E:PPDE2} on $[0,T]\times\Omega^L$
 satisfies
\begin{align*}
u(t_0,x_0)&\ge\inf_{(x,y)\in\mathcal{Y}^{L_0}(t_0,x_0,z)}\{u(T,x)-y(T)\}
\ge\min_{(x,y)\in\mathcal{Y}^{L_0}(t_0,x_0,z)}\{h(x)-y(T)\},
\end{align*}
we have
\begin{align}\label{E:ImLBound}
u_0^L(t_0,x_0)\ge\min_{(x,y)\in\mathcal{Y}^{L_0}(t_0,x_0,z)}\{h(x)-y(T)\}>-\infty.
\end{align}
By Lemma~\ref{L:ImUpper}, $u_0^L(t_0,x_0)<\infty$. Thus $u_0^L$ is real-valued.

By the definition of $u_0^L$,  $u_0^L(T,x_0)\ge h(x_0)$.
Since $u_0^L(T,x_0)\le u^z_+(T,x_0)=h(x_0)$, we can deduce that
\begin{align}\label{E:ImULh}
u_0(T,x_0)=h(x_0).
\end{align}

Next, fix $t\in (t_0,T]$. Let $(u_n)_n$ be a sequence of improper minimax $L_0$-super\-solutions
of \eqref{E:PPDE2} on $[0,T]\times\Omega^L$ such that 
$\abs{u_n(t_0,x_0)-u_0^L(t_0,x_0)}\le n^{-1}$ and that
\begin{align*}
u_n(t_n,x_n)-y_n(t)\le u_n(t_0,x_0)+\frac{1}{n}
\end{align*}
for some $(x_n,y_n)\in\mathcal{Y}^{L_0}(t_0,x_0,z)$, $n\in\N$. Then
\begin{align*}
u_0^L(t,x_n)-y_n(t)&\le u_n(t,x_n)-y_n(t)
\le u_n(t_0,x_0)+\frac{1}{n}
\le u_0^{L}(t_0,x_0)+\frac{2}{n}
\end{align*}
for every $n\in\N$, i.e., $u_0^L$ is an improper minimax $L_0$-supersolution of \eqref{E:PPDE2}
on $[0,T]\times\Omega^L$.

To show that $u_0^L$ is also an improper minimax $L_0$-subsolution of \eqref{E:PPDE2} 
on $[0,T]\times\Omega^L$
it suffices
(because of \eqref{E:ImULh}) to verify that
\begin{align}\label{E:ImVerify}
\sup_{(x,y)\in\mathcal{Y}^{L_0}(t_0,x_0,z)}\{u_0^L(t,x)-y(t)\}\ge u_0^L(t_0,x_0).
\end{align}
To this end, we  prove that the function $u^z:[0,T]\times\Omega^L\to [-\infty,\infty]$ defined by
\begin{align*}
u^z(\tilde{t},\tilde{x}):=\begin{cases}
\sup_{(x,y)\in\mathcal{Y}^{L_0}(\tilde{t},\tilde{x},z)}\{u_0^L(t,x)-y(t)\}&\text{if $\tilde{t}\le t$,}\\
u_0^L(\tilde{t},\tilde{x})&\text{if $\tilde{t}> t$,}
\end{cases}
\end{align*}
is an improper minimax $L_0$-supersolution of \eqref{E:PPDE2} on $[0,T]\times\Omega^L$.
Then, according to the definition of $u_0^L$, we have
$u_0^L(\tilde{t},\tilde{x})\le u^z(\tilde{t},\tilde{x})$, $(\tilde{t},\tilde{x})\in [0,T)\times\Omega^L$,
which immediately yields \eqref{E:ImVerify}.

To show that $u^z$ is an improper minimax $L_0$-supersolution of \eqref{E:PPDE2}
on $[0,T]\times\Omega^L$,
let us first check whether the terminal
condition is satisfied. Let $\tilde{x}\in\Omega^L$. If $T>t$, then
$u^z(T,\tilde{x})=u_0^L(T,\tilde{x})\ge h(\tilde{x})$ because $u_0^L$ is an improper 
minimax $L_0$-supersolution of \eqref{E:PPDE2} on $[0,T]\times\Omega^L$. 
If $T=t$, then we also get immediately the same inequality.
Next, note that, by \eqref{E:ImLBound}, $u^z(\tilde{t},\tilde{x})>-\infty$ for all
$(\tilde{t},\tilde{x})\in [0,T]\times\Omega^:$. Thus, given
$(\tilde{t},\tilde{x},\tilde{z})\in [0,T)\times\Omega^L\times H$,
$\tau\in (\tilde{t},T]$, and $\eps>0$, it suffices to show that there exists an
$(x,y)\in\mathcal{Y}^{L_0}(\tilde{t},\tilde{x},\tilde{z})$ such that
\begin{align}\label{E:ImDD}
u^z(\tau,x)-y(\tau)\le u^z(\tilde{t},\tilde{x})+\eps.
\end{align}
We shall distinguish among the following three cases.

\textit{Case~1.} Let $\tilde{t}\ge t$.
Then $u^z(\tilde{t},\tilde{x})=u_0^L(\tilde{t},\tilde{x})$
and $u^z(\tau,x)=u_0(\tau,x)$ for all $x\in\mathcal{X}^{L_0}(\tilde{\tau},\tilde{x})$,
i.e., we can deduce \eqref{E:ImDD} from $u_0$ being an improper minimax $L_0$-supersolution
of $\eqref{E:PPDE2}$ on $[0,T]\times\Omega^L$.

\textit{Case~2.} Let $\tilde{t}<t$ and $\tau<t$.
By Proposition~\ref{P:EU_ODE}, the initial value problem
\begin{align*}
\hat{\wx}^\prime(s)+A(s,\hat{\wx}(s))&=\hat{f}(s,\hat{x})\text{ a.e.~on $(\tilde{t},T)$, }
\hat{x}=\tilde{x}\text{ on $[0,\tilde{t}]$,}
\end{align*}
where $\hat{f}:[\tilde{t},T]\times C([0,T],H)\to H$ is defined by
\begin{align*}
\hat{f}(s,x):=\begin{cases}
\frac{F(s,x,\tilde{z})-F(s,x,z)}{\abs{\tilde{z}-z}^2}\cdot (\tilde{z}-z)&\text{ if $\tilde{z}\neq z$,}\\
0 &\text{ if $\tilde{z}=z$,}
\end{cases}
\end{align*}
has a unique solution $\hat{x}\in\mathcal{X}^{L_0}(\tilde{t},\tilde{x})$.
Next, put
\begin{align*}
\hat{y}(s):=\int_{\tilde{t}}^s(\hat{f}(r,\hat{x}),\tilde{z})-F(r,\hat{x},\tilde{z})\,dr,\quad t\in [\tilde{t},T],
\end{align*}
so that $(\hat{x},\hat{y})\in\mathcal{Y}^{L_0}(\tilde{t},\tilde{x},z)\cap
\mathcal{Y}^{L_0}(\tilde{t},\tilde{x},\tilde{z})$. 
Since $\hat{x}\in\mathcal{X}^{L_0}(t_0,x_0)$, 
\begin{align*}
u^z(\tau,\hat{x})-\hat{y}(\tau)&=
\sup_{(x,y)\in\mathcal{Y}^{L_0}(\tau,\hat{x},z)}
\{ u_0^L(t,x)-y(t)\} - \hat{y}(\tau)\\
&=\sup_{x\in\mathcal{X}^{L_0}(\tau,\hat{x})}\left\{
u_0^L(t,x)-\int_\tau^t [(f^x(s),z)-F(s,x,z)]\,ds
\right\}\\
&\qquad\qquad 
-\int_{\tilde{t}}^\tau [\underbrace{
(f^{\hat{x}}(s),\tilde{z})-F(s,\hat{x},\tilde{z})
}_{
=(f^{\hat{x}}(s),{z})-F(s,\hat{x},{z})
}]\,ds\\
&\le \sup_{x\in\mathcal{X}^{L_0}(\tilde{t},\tilde{x})}\left\{
u_0^L(t,x)-\int_{\tilde{t}}^t [(f^x(s),z)-F(s,x,z)]\,ds
\right\}\\
&=u^z(\tilde{t},\tilde{x}),
\end{align*}
i.e., \eqref{E:ImDD} holds (even for $\eps=0$).

\textit{Case~3.} Let $\tilde{t}<t$ and $\tau\ge t$.
By Case~2, there exists an $x_2\in\mathcal{X}^{L_0}(\tilde{t},\tilde{x})$
such that
\begin{align*}
u^z(t,x_2)-\int_{\tilde{t}}^t
[(f^{x_2}(s),\tilde{z})-F(s,x_2,\tilde{z})]\,ds\le u^z(\tilde{t},\tilde{x}).
\end{align*}
By Case~1, there exists an $x_1\in\mathcal{X}^{L_0}(\tilde{t},x_2)$
such that
\begin{align*}
u^z(\tau,x_1)-\int_t^\tau [(f^{x_1}(s),\tilde{z})-F(s,x,\tilde{z})]\,ds\le u^L_0(t,x_2)+\eps.
\end{align*}
Consequently, with 
$x:=\bfone_{[0,t)}\cdot x_2+\bfone_{[t,T]}\cdot x_1$, we have
\begin{align*}
u^z(\tau,x)-\int_{\tilde{t}}^\tau [(f^x(s),\tilde{z})-F(s,x,\tilde{z})]\,ds\le u^z(\tilde{t},\tilde{x})+\eps,
\end{align*}
which completes the proof.
\end{proof}

Let $L\ge L_0$. Since $u_0^L$ is possibly not semicontinuous, consider the 
envelopes $u^{L,-}_0$, $u^{L,+}_0:[0,T]\times\Omega^L\to\ [-\infty,\infty]$ defined by
\begin{align*}
u^{L,-}_0(t_0,x_0)&:=\varliminf_{(t,x)\to(t_0,x_0)} u^L_0(t,x)
:=\sup_{\delta>0} \Inf_{(t,x)\in O^L_\delta(t_0,x_0)} u_0^L(t,x), \\
u^{L,+}_0(t_0,x_0)&:=\varlimsup_{(t,x)\to(t_0,x_0)} u^L_0(t,x)
:=\Inf_{\delta>0} \sup_{(t,x)\in O^L_\delta(t_0,x_0)} u_0^L(t,x),
\end{align*}
where
\begin{align*}
O_\delta^L(t_0,x_0):=\{
(t,x)\in [0,T]\times\Omega^L: \mathbf{d}_\infty((t,x),(t_0,x_0))<\delta\}.
\end{align*}
Clearly,
\begin{align}\label{E:u0-+}
u_0^{L,-}\le u^L_0\le u_0^{L,+}.
\end{align}

\begin{lemma}\label{L:u0-}
Let $L\ge L_0$. Then $u_0^{L,-}>-\infty$ and $u^{L,+}_0<\infty$.
\end{lemma}

\begin{proof}
Fix $\delta>0$. Then
\begin{align*}
u^{L,-}_0(t_0,x_0)&\ge \inf_{(t_1,x_1)\in O_\delta(t_0,x_0)} u_0(t_1,x_1)\\
&\ge \inf_{(t_1,x_1)\in O_\delta(t_0,x_0)}
\min_{x\in\mathcal{X}^{L_0}(t_1,x_1)} \left\{
h(x)+\int_t^T F(s,x,0)\,ds
\right\}
&& \text{by \eqref{E:ImLBound}}\\
&\ge 
\min_{(t,x)\in [0,T]\times\Omega^L} \ \left\{
h(x)+\int_t^T F(s,x,0)\,ds
\right\}\\
&>-\infty.
\end{align*}

To show that  $u^{L,+}_0<\infty$, one can 
proceed similarly as in the proof of Lemma~\ref{L:u0-}
but one should use Lemma~\ref{L:ImUpper} instead of \eqref{E:ImLBound}.
\end{proof}

\begin{remark}\label{R:LSC_USC} By \eqref{E:u0-+} and Lemma~\ref{L:u0-},
$u_0^{L,-}$ and $u_0^{L,+}$ are real-valued. Moreover,
\begin{align*}
u^{L,-}_0\in\mathrm{LSC}([0,T]\times\Omega^L)\text{ and }
u^{L,+}_0\in\mathrm{USC}([0,T]\times\Omega^L).
\end{align*}
The proof of this statement is standard.
One can, e.g., follow the argument in
 \cite[Proof of Proposition 8.7, p.~77]{Subbotin_book}.
\end{remark}

\begin{lemma}\label{L:u0-+SemiSol}
Let $L\ge L_0$. The function
$u^{L,-}_0$ is a minimax $L_0$-supersolution of \eqref{E:PPDE2} on $[0,T]\times\Omega^L$
and the function $u^{L,+}_0$ is a 
minimax $L_0$-subsolution of \eqref{E:PPDE2}  on $[0,T]\times\Omega^L$.
\end{lemma}

\begin{proof}
We only show that $u^{L,-}_0$ is a minimax $L_0$-supersolution 
of \eqref{E:PPDE2} on $[0,T]\times\Omega^L$.
To this end, fix $(t_0,x_0)\in [0,T)\times\Omega^L$ and let
$(t,z)\in (t_0,T]\times H$. By the definition of $u_0^{L,-}$, there exists a sequence
$(t_n,x_n)_n$ in $[0,T]\times\Omega^L$ such that
\begin{align*}
\mathbf{d}_\infty((t_n,x_n),(t_{\color{black} 0},x_{\color{black} 0}))<\frac{1}{n}\qquad\text{and}\qquad
u_0^{L,-}(t_0,x_0)\ge u_0^L(t_n,x_n)-\frac{1}{n}.
\end{align*}
Since $t>t_0$, we can assume that $t_n<t$ for all $n\in\N$.
Recall that $u_0^L$ is an improper minimax $L_0$-supersolution 
of \eqref{E:PPDE2} on $[0,T]\times\Omega^L$. Thus
\begin{align*}
u_0^L(t,\tilde{x}_n)-\tilde{y}_n(t)\le u_0^L(t_n,x_n)+\frac{1}{n}
\end{align*}
for some $(\tilde{x}_n,\tilde{y}_n)\in\mathcal{Y}^{L_0}(t_n,x_n,z)$.
By Proposition~\ref{P:CompactII}, we can assume that $(\tilde{x}_n,\tilde{y}_n)_n$
converges to some $(\tilde{x}_0,\tilde{y}_0)\in\mathcal{Y}^{L_0}(t_0,x_0,z)$. Hence
\begin{align*}
u_0^{L,-}(t_0,x_0)\ge\varliminf_n u_0^L(t_n,x_n)&\ge
\varliminf_n \left[u_0^L(t,\tilde{x}_n)-\tilde{y}_n(t)-\frac{1}{n}\right]\\
&\ge \varliminf_{(s,x)\to (t,\tilde{x}_0)} u_0^L(s,x)-\tilde{y}_0(t)\\
&\ge u_0^{L,-}(t,\tilde{x}_0)-\tilde{y}_0(t)
\end{align*}
and thus, by Lemma~\ref{L:EquivUpperSol}, 
$u_0^{L,-}$ satisfies the minimax $L_0$-supersolution property on $[0,T)\times\Omega^L$.
The minimax $L_0$-supersolution property for the terminal condition can be verified similarly.
Together with Remark~\ref{R:LSC_USC}, this concludes the proof.
\end{proof}

\begin{theorem}[Existence and uniqueness~I]\label{T:ExistenceMinimax1}
Let $L\ge L_0$.
Then $u_0^L$ is the unique minimax $L_0$-solution of 
\eqref{E:PPDE2} on $[0,T]\times\Omega^L$.
\end{theorem}

\begin{proof}
By Theorem~\ref{T:Comparison2} and
Lemma~\ref{L:u0-+SemiSol}, $u^{L,+}_0\le u^{L,-}_0$ on $[0,T]\times\Omega^L$.
Together with \eqref{E:u0-+}, we can conclude that
$u_0^L=u_0^{L,+}=u_0^{L,-}$. Therefore, by Proposition~\ref{P:MinimaxUpperLower},
the function $u_0^L$ is a minimax $L_0$-solution of \eqref{E:PPDE2} on $[0,T]\times\Omega^L$.
Uniqueness follows  from Theorem~\ref{T:Comparison2}.
\end{proof}

From Theorem~\ref{T:Comparison3} and Theorem~\ref{T:ExistenceMinimax1} together with
 Remark~\ref{R:MinimaxSolution} and Remark~\ref{R:MinimaxRestriction},
we obtain the following result.

\begin{theorem}[Existence and uniqueness~II] \label{T:ExistenceMinimax2}
There exists a function $u_0:[0,T]\times\Omega\to\R$ that is the unique
minimax solution 
on  $[0,T]\times\Omega$.
Moreover, 
 $u_0=u_0^L$ on $[0,T]\times\Omega^L$ for every $L\ge L_0$.
\end{theorem}

\section{Stability}

\renewcommand{\theequation}
{\textrm{TVP}~$n$}

Given $n\in\N$, consider the terminal value problem
\begin{equation}\label{E:PPDEn}
\begin{split}
&\partial_t u-\langle A(t,x(t)),\partial_x u\rangle+F_n(t,x,\partial_x u)=0,\, (t,x)\in [0,T)\times C([0,T],H),\\
&u(T,x)=h_n(x),\quad x\in C([0,T],H),
\end{split}
\end{equation}
where $h_n: C([0,T],H)\to\R$ and $F_n:[0,T]\times C([0,T],H)\times H\to\R$, respectively, are
supposed to satisfy \textbf{H}$(h)$ and \textbf{H}$(F)$, respectively, with the same
Lipschitz constants and moduli of continuity.

\setcounter{equation}{0}

\renewcommand{\theequation}{\thesection.\arabic{equation}}

Our approach for the following basic stability result is similar as in the finite-dimensional case 
(regarding PDEs, see \cite[pp.~36~ff.]{Subbotin_minimaxSol}; 
regarding PPDEs, see \cite[Theorem~6.1]{Lukoyanov01minimax}).

\begin{theorem}\label{T:StabSuper}
Let $L\ge 0$. 

(i) For every $n\in\N$, let $u^L_n$ be a minimax supersolution
of \eqref{E:PPDEn} on $[0,T]\times\Omega^L$. Then the function
$u^L:[0,T]\times\Omega^L\to\R$ defined by
\begin{align*}
u^{L,-}(t_0,x_0)&:=
\varliminf_{\substack{n\to\infty,\\ (t,x)\to (t_0,x_0)
}}  u_n^L(t,x):=
\sup\limits_{\substack{\delta> 0,\\ n\in\N}}\quad
 \inf\limits_{\substack{(t,x)\in O^L_\delta(t_0,x_0),\\m\ge n}}\quad u_m^L(t,x)
\end{align*}
is a minimax supersolution of \eqref{E:PPDE2} on $[0,T]\times\Omega^L$.

(ii) For every $n\in\N$, let $u^L_n$ be a minimax subsolution
of \eqref{E:PPDEn} on $[0,T]\times\Omega^L$. Then the function
$u^L:[0,T]\times\Omega^L\to\R$ defined by
\begin{align*}\label{E:StabSub}
u^{L,+}(t_0,x_0)&:=
\varlimsup_{\substack{n\to\infty,\\ (t,x)\to (t_0,x_0)
}}  u_n^L(t,x)
:=\inf\limits_{\substack{\delta> 0,\\ n\in\N}}\quad
 \sup\limits_{\substack{(t,x)\in O^L_\delta(t_0,x_0),\\m\ge n}}\quad u_m^L(t,x)
\end{align*}
is a minimax subsolution of \eqref{E:PPDE2} on $[0,T]\times\Omega^L$.

(iii)  For every $n\in\N$,  let $u^L_n$ be a minimax solution
of \eqref{E:PPDEn} on $[0,T]\times\Omega^L$. Then $(u_n^L)_n$ converges
to  the function $u^{L,+}$, 
which is a minimax solution of \eqref{E:PPDE2}
on $[0,T]\times\Omega^L$.
\end{theorem}


\begin{proof}[Proof of Theorem~\ref{T:StabSuper}]
We prove only (i) and (iii).

(i) First, note that $u^{L,-}\in\mathrm{LSC}([0,T]\times\Omega^L)$ (cf.~Remark~\ref{R:LSC_USC}).
It suffices to verify that $u^{L,-}$ satisfies the minimax $L$-supersolution property
for the terminal condition on $\Omega^L$ and for the PPDE in $[0,T)\times\Omega^L$.
We shall do only the latter as the former can be done similarly.

Fix $(t_0,x_0)\in [0,T)\times\Omega^L$. Let $t\in (t_0,T)$ and $z\in H$.
Consider a sequence $(t_n,x_n)_n$ in $[0,T)\times\Omega^L$ such that
$\mathbf{d}_\infty((t_n,x_n),(t_0,x_0))<n^{-1}$ and 
\begin{align*}
u^{L,-}(t_0,x_0)\ge u_n^L(t_n,x_n)-\frac{1}{n}.
\end{align*}
Without loss of generality, assume that $t_n<t$ for every $n\in\N$.
Since $u_n^L$ is a minimax $L$-supersolution of \eqref{E:PPDEn} on $[0,T)\times\Omega^L$
(cf.~Remark~\ref{R:Lminimax}),
there exists an $\tilde{x}_n\in\mathcal{X}^L(t_n,x_n)$ such that
\begin{align*}
u_n^L(t,\tilde{x}_n)+\int_{t_n}^t F_n(s,\tilde{x}_n,z)-(f^{\tilde{x}_n}(s),z)\,ds\le u_n^L(t_n,x_n).
\end{align*}
We can also assume that $(\tilde{x}_n)_n$ converges in $\Omega^L$
to some $\tilde{x}_0\in\mathcal{X}^L(t_0,x_0)$
(see Proposition~\ref{P:CompactII}). Therefore,
\begin{align*}
u^{L,-}(t_0,x_0)\ge\varliminf_n u_n^L(t_n,x_n)
&\ge\varliminf_n \left[u_n^L(t,\tilde{x}_n)+
\int_{t_n}^t F_n(r,\tilde{x}_n,z)-(f^{\tilde{x}_n}(r),z)\,dr\right]\\
&\ge\varliminf_{\substack{(s,x)\to (t,\tilde{x}_0)\\ n\to\infty}} u_n^L(s,x) +
\int_{t}^t F(r,\tilde{x}_0,z)-(f^{\tilde{x}_0}(r),z)\,dr\\
&\ge u^{L,-}(t,\tilde{x}_0)+\int_{t}^t F(r,\tilde{x}_0,z)-(f^{\tilde{x}_0}(r),z)\,dr.
\end{align*}

(iii) 
Let $(t_0,x_0)\in [0,T)\times\Omega^L$. Then
\begin{align*}
u^{L,-}(t_0,x_0)\le\varliminf_{n\to\infty} u^L_n(t_0,x_0)\le \varlimsup_{n\to\infty} u^L_n(t_0,x_0)
\le u^{L,+}(t_0,x_0).
\end{align*}
Moreover, by Comparison (Theorem~\ref{T:Comparison2}),
 $u^{L,+}(t_0,x_0)\le u^{L,-}(t_0,x_0)$.
 Thus $u_n^L\to u^{L,+}$ and $u^{L,+}=u^{L,-}$.
 Consequently, noting (i), (ii), and Proposition~\ref{P:MinimaxUpperLower},
 we can infer that $u^{L,+}$ is a minimax $L$-solution of \eqref{E:PPDE2} on 
 $[0,T]\times\Omega^L$.
\end{proof}


\section{Applications to optimal control}
Consider a controlled evolution equation of the form
\begin{equation}\label{E:ContrEq}
\begin{split}
\wx^\prime(t)+A(t,\wx(t))&=f(t,x,a(t))\quad\text{a.e.~on $(t_0,T)$,}\\
x&=x_0\quad\text{on $[0,t_0]$.}
\end{split}
\end{equation}
Here, $t_0\in [0,T]$, $x_0\in\Omega$, $a\in\mathcal{A}^{t_0}$ with
\begin{align*}
\mathcal{A}^{t_0}:=\{a:[t_0,T]\to P\text{ measurable}\}
\end{align*}
for some compact topological space $P$. A solution of \eqref{E:ContrEq} is to be understood as
an element of $\mathcal{X}^L(t_0,x_0)$ for some $L\ge 0$ and is denoted by
$x^{t_0,x_0,a}$.

Goal of the controller is to minimize a cost functional of the form
\begin{align*}
\int_{t_0}^T \ell(t,x^{t_0,x_0,a},a(t))\,dt+h(x^{t_0,x_0,a})
\end{align*}
over all $a\in\mathcal{A}^{t_0}$.

Next, we state the hypotheses for the data of our control problem.
To this end, let $L_f\ge 0$.

\textbf{H}$(f)_1$:  The function $f:[0,T]\times C([0,T],H)\times P\to H$ satisfies the following:

(i) $f$ is continuous.

(ii) For every
$(t,x,\bfp)\in [0,T]\times C([0,T],H)\times P$, 
\begin{align*}
\abs{f(t,x,\bfp)}\le L_f(1+\sup_{s\le t}\abs{x(s)}).
\end{align*}

(iii) For a.e.~$t\in (0,T)$,
for every $x$, $y\in\Omega$, and for every $\bfp\in P$,
\begin{align*}
\abs{f(t,x,\bfp)-f(t,y,\bfp)}\le L_f 
\sup_{s\le t}\abs{x(s)-y(s)}.
\end{align*}

\textbf{H}$(\ell)_1$: The function $\ell: [0,T]\times C([0,T],H)\times P\to\R$ satisfies the following:

(i) $\ell$ is continuous.

(ii) For every $(t,x,\bfp)\in [0,T]\times C([0,T],H)\times P$,
\begin{align*}
\abs{\ell(t,x,\bfp)}\le L_f(1+\sup_{s\le t}\abs{x(s)}).
\end{align*}

(iii) For a.e.~$t\in (0,T)$, for every $x$, $y\in\Omega$, and for
every $\bfp\in P$,
\begin{align*}
\abs{\ell(t,x,\bfp)-\ell(t,y,\bfp)}\le L_f 
\sup_{s\le t}\abs{x(s)-y(s)}.
\end{align*}

\textbf{H}$(h)_1$: The function $h: C([0,T],H)\to\R$ satisfies the following:

(i) $h$ is continuous.

(ii) For every $x$, $y\in\Omega$,
\begin{align*}
\abs{h(x)-h(y)}\le L_f\norm{x-y}_\infty.
\end{align*}

Immediately from Proposition~\ref{P:EU_ODE}, we obtain
the following statement.

\begin{proposition}
Suppose that \emph{\textbf{H}$(f)_1$} holds.
Let $(t_0,x_0)\in [0,T)\times\Omega$ and $a\in\mathcal{A}^{t_0}$.
Then the initial-value problem \eqref{E:ContrEq} has a unique
solution in $\mathcal{X}^{L_f}(t_0,x_0)$.
\end{proposition}

We also need to establish continuous dependence properties of our
controlled evolution equation with respect
to the initial data.

\begin{proposition}\label{P:ContrContDep}
Suppose that \emph{\textbf{H}$(f)_1$} holds. 
Let $t_0$, $t_1\in [0,T]$ and $x_0$, $y_0\in\Omega$.
Fix $a\in\mathcal{A}^{0}$ and denote its restrictions to subintervals also by $a$.

(i) For every $t\in [t_0,T]$,
\begin{align}\label{E1:ContrContDep}
\abs{x^{t_0,x_0,a}(t)-x^{t_0,y_0,a}(t)}&\le e^{L_f(t-t_0)} 
\sup_{s\le t_0} \abs{x_0(s)-y_0(s)}.
\end{align}

(ii) Let $L\ge 0$. If $x_0\in\Omega^L$, then 
\begin{align}\label{E3:ContrContDep}
\norm{x^{t_0,x_0,a}-x^{t_1,x_0,a}}_\infty\le 4\max\{L,L_f\}(1+C)e^{L_f T}\abs{t_1-t_0} 
\end{align}
for some constant $C$ that  depends only on $L$, $L_f$, and $x_\ast$.
\end{proposition}

\begin{proof}
(i) Put $x:=x^{t_0,x_0,a}$, $y:=x^{t_0,y_0,a}$, $z:=x-y$, $\wz:=\wx-\wy$, and
\begin{align*}
m(t):=\sup_{s\le t} \abs{z(s)},\quad t\in [t_0,T].
\end{align*}
By integration-by-parts (Proposition~\ref{P:parts}), for every $t\in [t_0,T]$,  
\begin{align*}
\frac{1}{2}\abs{z(t)}^2&=\frac{1}{2}\abs{z(t_0)^2}+
\int_{t_0}^t \langle \wz^\prime(s),\wz(s)\rangle\,ds \\
&\le \frac{1}{2}\abs{z(t_0)^2}+ \int_{t_0}^t 
\underbrace{\abs{f(s,x,a(s))-f(s,y,a(s))}}_{\le L_f m(s)}\,\abs{z(s)}\,ds.
\end{align*}
Here, we used the monotonicity of $A$ and  \textbf{H}$(f)_1$(iii).
Finally, Gronwall's inequality  yields, 
$m(t)^2\le e^{2L_f(t-t_0)} m(t_0)^2$, i.e., we have \eqref{E1:ContrContDep}.

(ii) Without loss of generality, let $L\ge L_f$. 
Put $x:=x^{t_0,x_0,a}$, $z:=x-x_0$, and $\wz:=\wx-\wx_0$, and 
\begin{align*}\
m(t):=\sup_{s\le t} \abs{z(s)},\quad t\in [t_0,T].
\end{align*}
Recall that $x$, $x_0\in\Omega^L=\mathcal{X}^L(0,x_\ast)$.
Thus, by the a-priori estimates of Lemma~\ref{L:Apriori}, there exists a
constant $C$ depending only on $L$ and 
$x_\ast$ such that $$\max\{\norm{x}_\infty,\norm{x_0}_\infty\}\le C.$$
Therefore, we can apply integration-by-parts (Proposition~\ref{P:parts}) and obtain
thanks to the monotonicity of $A$ and to \textbf{H}$(f)_1$(ii) 
\begin{align*}
\frac{1}{2}\abs{z(t)}^2=\int_{t_0}^t \langle \wz^\prime(s),\wz(s)\rangle\,ds 
&\le \int_{t_0}^t (f(s,x,a(s))-f^{x_0}(s),z(s))\,ds\\
&\le \int_{t_0}^t  L\left\{2+\sup_{s\le t} \left[\abs{x(s)}+\abs{x_0(s)}\right]\right\}\,\abs{z(s)}\,ds\\
&\le 2L(1+C) (t-t_0) m(t)
\end{align*}
for every $t\in [t_0,T]$.
This yields 
\begin{align}\label{E2:ContrContDep}
m(t)\le 4L(1+C)(t-t_0).
\end{align}

Next, without loss of generality, let $t_1\ge t_0$. Fix $t\in [t_1,T]$. 
By \eqref{E1:ContrContDep} and \eqref{E2:ContrContDep},
\begin{align*}
\abs{x^{t_0,x_0,a}(t)-x^{t_1,x_0,\tilde{a}}(t)}&=
\abs{x^{t_1,x^{t_0,x_0,a},\tilde{a}}(t)-x^{t_1,x_0,\tilde{a}}(t)}\\
&\le e^{L_f(t-t_1)} \sup_{s\le t_1} \abs{
x^{t_0,x_0,a}(s)-x_0(s)
}\\
&\le 4L(1+C) e^{L_f(t-t_1)}\,(t_1-t_0),
\end{align*}
which together with \eqref{E2:ContrContDep} yields \eqref{E3:ContrContDep}.
\end{proof}

The \emph{value function} $v:[0,T]\times\Omega\to\R$ for our control problem is
defined by
\begin{align*}
v(t_0,x_0):=\inf_{a\in\mathcal{A}^{t_0}} J(t_0,x_0;a),
\end{align*}
where
\begin{align*}
J(t_0,x_0;a):=\int_{t_0}^T \ell(t,x^{t_0,x_0,a},a(t))\,dt+h(x^{t_0,x_0,a}).
\end{align*}

Similarly as in the finite-dimensional case, one can prove a dynamic programming
principle.

\begin{theorem}\label{T:DPP}
Suppose that \emph{\textbf{H}$(f)_1$},  \emph{\textbf{H}$(\ell)_1$}, and  \emph{\textbf{H}$(h)$}
hold. Let $(t_0,x_0)\in [0,T]\times\Omega$. Then, for every $t\in [t_0,T]$,
\begin{align*}
v(t_0,x_0)=\inf_{a\in\mathcal{A}^{t_0}}\left[
\int_{t_0}^t \ell(s,x^{t_0,x_0,a},a(s))\,ds+v(t,x^{t_0,x_0,a})
\right].
\end{align*}
\end{theorem}

In order to show that the value function $v$ is a minimax solution of the corresponding 
Bellman equation, we need to establish that $v$ is continuous.

\begin{theorem}[Regularity of the value function]\label{T:Reg}
Suppose that \emph{\textbf{H}$(f)_1$},  \emph{\textbf{H}$(\ell)_1$}, and  \emph{\textbf{H}$(h)_1$}
hold.  Then $v$ is continuous. In particular, the following holds:

(i)  For every $t_0\in [0,T]$, $x_0$, $y_0\in\Omega$,
\begin{align}\label{E1:ContrReg}
\abs{v(t_0,x_0)-v(t_0,y_0)}&\le L_f (T-t_0+1)e^{L_f(T-t_0)}\sup_{s\le t_0} \abs{x_0(s)-y_0(s)}.
\end{align}

(ii) Let $L\ge 0$. If $x_0\in\Omega^L$, then, for every $t_0$, $t_1\in [0,T]$,
\begin{align}\label{E2:ContrReg}
\abs{v(t_0,x_0)-v(t_1,x_0)}&\le C \abs{t_1-t_0}
\end{align}
for some constant $C$ that depends only on $L$, $L_f$, $x_\ast$, and $T$.
\end{theorem}

\begin{proof}
Fix $a\in\mathcal{A}^{0}$.  

(i) By \textbf{H}$(\ell)_1$(iii), \textbf{H}$(h)_1$(ii),
and \textbf{H}$(f)_1$ together with Proposition~\ref{P:ContrContDep}~(i),
\begin{align*}
J(t_0,x_0;a)-J(t_0,y_0;a)&\le 
L_f(T-t_0+1) \norm{x^{t_0,x_0,a}-x^{t_0,y_0,a}}_\infty\\
&\le L_f(T-t_0+1) e^{L_f(T-t_0)} \sup_{s\le t_0} \abs{x_0(s)-y_0(s)},
\end{align*}
from which we get \eqref{E1:ContrReg}. 

(ii) Without loss of generality, let $t_1\ge t_0$. By \textbf{H}$(\ell)_1$, \textbf{H}$(h)_1$, 
Lemma~\ref{L:Apriori}, and \textbf{H}$(f)_1$ together with Proposition~\ref{P:ContrContDep}~(ii),
\begin{align*}
J(t_0,x_0;a)-J(t_1,x_0;a)&=
\left[
\int_{t_0}^T \ell(t,x^{t_0,x_0,a},a(t))\,dt-\int_{t_1}^T \ell(t,x^{t_1,x_0,a},a(t))\,dt
\right]\\
&\qquad +\left[
h(x^{t_0,x_0,a})-h(x^{t_1,x_0,a})
\right]\\
&\le (t_1-t_0) L_f\left(1+\sup_{t\le t_1} \abs{x^{t_0,x_0,a}(t)}\right) \\ &\qquad
+\int_{t_1}^T L_f \sup_{s\le t} \abs{
x^{t_0,x_0,a}(s)-x^{t_1,x_0,a}(s)
}\,dt\\ &\qquad
 + L_f \norm{x^{t_0,x_0,a}-x^{t_1,x_0,a}}_\infty\\
 &\le  C\abs{t_1-t_0}
\end{align*}
for some constant $C$ that depends only on $L$, $L_f$, $x_\ast$, and $T$.
This yields \eqref{E2:ContrReg}.

By parts (i) and (ii) of this proof together with $t\mapsto x_0(t)$ being continuous
for each $x_0\in\Omega$, we have continuity of $v$ under $\mathbf{d}_\infty$.
\end{proof}

\begin{remark}
Note that we have not established that $v$ is (locally) Lipschitz continuous under
$\mathbf{d}_\infty$ on the sets $\Omega^L$.
\end{remark}

Now, we are in the position to prove the main result of this section.

The  Bellman equation associated to our control problem is given by \eqref{E:PPDE2} with
\begin{align}\label{E:BellmanF}
F(t,x,z):=\min_{\bfp\in P} \left[\ell(t,x,\bfp)+(f(t,x,\bfp),z)\right].
\end{align}

\begin{theorem} 
Suppose that \emph{\textbf{H}$(f)_1$}, \emph{\textbf{H}$(\ell)_1$}, and
\emph{\textbf{H}$(h)_1$} hold.  
Let  $F$ be given by \eqref{E:BellmanF}.
Then
the value function $v$ is a minimax $L_f$-solution of \eqref{E:PPDE2} on $[0,T]\times\Omega$.
\end{theorem}

\begin{proof}
Note that  \textbf{H}$(f)_1$ and \textbf{H}$(\ell)_1$ imply \textbf{H}$(F)$.

(i) First, we show that $v$ is a minimax $L_f$-supersolution of \eqref{E:PPDE2}
on $[0,T]\times\Omega$. To this end, fix
$(t_0,x_0,z)\in [0,T]\times\Omega\times H$ and $t\in [t_0,T]$.
For every $\eps>0$, there exists,  
by Theorem~\ref{T:DPP},
a control $a^\eps\in\mathcal{A}^{t_0}$ such that, with $x^\eps:=x^{t_0,x_0,a^\eps}$,
\begin{align*}
\int_{t_0}^t \ell(s,x^\eps,a^\eps(s))\,ds+ v(t,x^\eps)-v(t_0,x_0)\le \eps.
\end{align*}
Since by \eqref{E:BellmanF} also
\begin{align*}
\int_{t_0}^t (-f^{x^\eps}(s),z)+F(s,x^\eps,z)\,ds &\le  \int_{t_0}^t   \ell(s,x^\eps,a^\eps(s))\,ds
\end{align*}
holds, we have
\begin{align*}
v(t,x^\eps)-v(t_0,x_0)\le \eps+\int_{t_0}^t \left[ -F(s,x^\eps,z)+(f^{x^\eps}(s),z)\right]\,ds.
\end{align*}
Since $\mathcal{X}^{L_f}(t_0,x_0)$ is compact
and $F$ is continuous, there exists an $x\in\mathcal{X}^{L_f}(t_0,x_0)$ depending on $t$
such that
\begin{align}\label{E1:ContrMain}
v(t,x)-v(t_0,x_0)\le \int_{t_0}^t \left[-F(s,x,z)+(f^x(s),z)\right]\,ds.
\end{align}
Recall that $v$ is continuous (Theorem~\ref{T:Reg}).
Thus we can apply Lemma~\ref{L:EquivUpperSol} to deduce the existence of
an $x\in\mathcal{X}^{L_f}(t_0,x_0)$ such that, for all $t\in [t_0,T]$, \eqref{E1:ContrMain} holds.
Since the definition of $v$ yields $v(T,\cdot)=h$, 
we can conclude that $v$ is a minimax $L_f$-supersolution of
\eqref{E:PPDE2} on $[0,T]\times\Omega$.

(ii) Next, we show that $v$ is a minimax $L_f$-solution of \eqref{E:PPDE2}
on $[0,T]\times\Omega$.

Note that, by Theorem~\ref{T:ExistenceMinimax2}, there exists 
a unique minimax solution $u$ of \eqref{E:PPDE2} on $[0,T]\times\Omega$.
Moreover, for every $L\ge L_f$, 
the restriction of $u$ to $[0,T]\times\Omega^L$ is, according to Theorem~\ref{T:ExistenceMinimax1},
the unique minimax $L_f$-solution of \eqref{E:PPDE2}  on $[0,T]\times\Omega^L$.
Thus it suffices to show that $u=v$ on $\Omega^L$ for each $L\ge L_f$. To this end,
fix $L\ge L_f$.

(ii-a) First, note that  $u\le v$ on $\Omega^L$. This follows from part~(i) of this proof
together with Theorem~\ref{T:Comparison2}.

(ii-b) Now, we show that $v\le u$ on $\Omega^L$. 
To this end, we verify that the function $w:[0,T]\times\Omega^L\times\Omega^L\to\R$
defined by $w(t,x,y):=v(t,x)-u(t,y)$ is a viscosity $L_f$-subsolution of \eqref{E:DPPDE}
on $[0,T]\times\Omega^L\times\Omega^L$. We proceed similarly as in the proof
of the doubling theorem (Theorem~\ref{T:MinimaxDoubling}). 
Let $(t_0,x_0,y_0)\in [0,T]\times\Omega^L\times\Omega^L$ and
$\varphi\in\underline{\mathcal{A}}^{L_f}w(t_0,x_0,y_0)$.
Then there exists a $T_0\in (t_0,T]$ such that,
for every $(t,x,y)\in [t_0,T]\times\mathcal{X}^{L_f}(t_0,x_0)\times\mathcal{X}^{L_f}(t_0,x_0)$,
\begin{align}\label{E1:uvContr}
\varphi(t,x,y)-\varphi(t_0,x_0,y_0)\ge w(t,x,y)-w(t_0,x_0,y_0).
\end{align}
Put $z:=-\partial_y\varphi(t_0,x_0,y_0)$. Since $u$ is a minimax $L_f$-solution
of \eqref{E:PPDE2} on $[0,T]\times\Omega^L$, there exists a $y\in\mathcal{X}^{L_f}(t_0,y_0)$
such that, for every $t\in [t_0,T]$,
\begin{align}\label{E2:uvContr}
u(t,y)-u(t_0,y_0)=\int_{t_0}^t \left[
(f^y(s),z)-F(s,y,z)
\right]\,ds.
\end{align}
Next, let $\bfp\in P$  such that
\begin{align}\label{E:ContrP}
\ell(t_0,x_0,\bfp)+(f(t_0,x_0,\bfp),\partial_x\varphi(t_0,x_0,y_0)=F(t_0,x_0,\partial_x\varphi(t_0,x_0,y_0)),
\end{align}
which is possible because $P$ is compact and $F$ is continuous.
Also put $x:=x^{t_0,x_0,a}$, where $a(t):=\bfp$, $t\in [t_0,T]$.
Then, by Theorem~\ref{T:DPP},
\begin{align}\label{E3:uvContr}
v(t,x)-v(t_0,x_0)\ge \int_{t_0}^t \left[-\ell(s,x,\bfp)\right]\,ds.
\end{align}
Applying the functional chain-rule to the left-hand side of \eqref{E1:uvContr} yields
together with \eqref{E2:uvContr} and \eqref{E3:uvContr}
\begin{align*}
&\int_{t_0}^t \Biggl[
\partial_t\varphi(s,x,y)-\langle A(s,\wx(s)),\partial_x\varphi(s,x,y)\rangle
-\langle A(s,\wy(s)),\partial_y\varphi(s,x,y)\rangle\\ &\qquad
+(f(s,x,\bfp),\partial_x \varphi(s,x,y))+(f^y(s),\partial_y\varphi(s,x,y))
\Biggr]\,ds\\
&\ge \int_{t_0}^t  \left[
-\ell(s,x,\bfp)-(f^y(s),z)+F(s,y,z)
\right]\,ds,
\end{align*}
which in turn implies
\begin{align*}
&\partial_t\varphi(t_0,x_0,y_0)\\ &\qquad
+\varlimsup_{\delta\downarrow 0} \frac{1}{\delta} \Biggl[\int_{t_0}^{t_0+\delta}
-\langle A(t,\wx(t)),\partial_x  \varphi(t,x,y)\rangle
-\langle A(t,\wy(t)),\partial_y  \varphi(t,x,y)\rangle\,dt \Biggr]\\
&\qquad
+\left[\ell(t_0,x_0,\bfp)+(f(t_0,x_0,\bfp),\partial_x\varphi(t_0,x_0,y_0) \right]
-F(t_0,x_0,-\partial_y\varphi(t_0,x_0,y_0))\ge 0.
\end{align*}
Noting \eqref{E:ContrP}, we get \eqref{E:DViscSub}, i.e.,
$w$ is $L_f$-subsolution of \eqref{E:DPPDE}
on $[0,T]\times\Omega^L\times\Omega^L$. 
Thus, by Theorem~\ref{T:Comparison1}, $v\le u$ on $\Omega^L$.
This concludes the proof.
\end{proof}

\section{Applications to differential games}\label{S:DiffGames}
It will be demonstrated that the Krasovski\u\i-Subbotin approach for differential games
(see \cite{KrasovksiiSubbotin}) can also be applied to  infinite-dimensional 
settings. We are able to adapt the methods by Lukoyanov~\cite{Lukoyanov03b},
where 
the finite-dimensional PPDE case is treated. Note that in this section,
we require stronger conditions for our data in contrast to the previous section.

Consider an evolution equation controlled by two players, which we call
the \emph{controller} $a$ and
the \emph{disturbance} $b$, of the form
\begin{equation}\label{E:GameEq}
\begin{split}
\mathbf{x}^\prime(t)+A(t,\mathbf{x}(t))&=f(t,x,a(t),b(t))\quad\text{a.e.~on $(t_0,T)$,}\\
x&=x_0\quad\text{on $[0,t_0]$,}
\end{split}
\end{equation}
where $t_0\in [0,T]$, $x_0\in\Omega$, 
$a\in\mathcal{A}^{t_0}$,  and $b\in\mathcal{B}^{t_0}$ with
\begin{align*}
\mathcal{A}^{t_0}:=\{a:[t_0,T]\to P\text{ measurable}\},\,
\mathcal{B}^{t_0}:=\{b:[t_0,T]\to Q\text{ measurable}\}
\end{align*}
for some compact topological spaces $P$ and $Q$. 
A solution of \eqref{E:GameEq} is denoted by $x^{t_0,x_0,a,b}$
and is to be understood as an element of $\mathcal{X}^L(t_0,x_0)$ for some $L\ge 0$.

Goal of the controller is to minimize via $a\in\mathcal{A}^{t_0}$ a cost functional of the form
\begin{align*}
\int_{t_0}^T \ell(t,x^{t_0,x_0,a,b},a(t),b(t))\,dt +h(x^{t_0,x_0,a,b})
\end{align*}
while assuming this functional is to be maximized by the disturbance via $b\in\mathcal{B}^{t_0}$,
thus representing uncertainty.
The precise description of the admissible actions of the controller
and the corresponding notions for the disturbance 
will be given shortly.

We proceed by stating the additional hypotheses for the data of our game.
To this end, fix a (global) constant $L_f>1$.

\textbf{H}$(f)_2$:  The function $f:[0,T]\times C([0,T],H)\times P\times Q\to H$ satisfies the following:

(i) $f$ is continuous.

(ii) For every
$(t,x,\bfp,\bfq)\in [0,T]\times C([0,T],H)\times P\times Q$, 
\begin{align*}
\abs{f(t,x,\bfp,\bfq)}\le L_f(1+\sup_{s\le t}\abs{x(s)}).
\end{align*}

(iii) For a.e.~$t\in (0,T)$
and for every $x$, $y\in\Omega$, every $\bfp\in P$, and every $\bfq\in Q$,
\begin{align*}
\abs{f(t,x,\bfp,\bfq)-f(t,y,\bfp,\bfq)}\le L_f 
\sqrt{
\abs{x(t)-y(t)}^2+\int_0^t \abs{x(s)-y(s)}^2\,ds
}.
\end{align*}
\textbf{H}$(\ell)_2$: The function $\ell: [0,T]\times C([0,T],H)\times P\times Q\to\R$ satisfies the following:

(i) $\ell$ is continuous.

(ii) For every $(t,x,\bfp,\bfq)\in [0,T]\times C([0,T],H)\times P\times Q$,
\begin{align*}
\abs{\ell(t,x,\bfp,\bfq)}\le L_f(1+\sup_{s\le t}\abs{x(s)}).
\end{align*}

(iii) For a.e.~$t\in (0,T)$, and for every $x$, $y\in\Omega$,
every $\bfp\in P$, and every $\bfq\in Q$,
\begin{align*}
\abs{\ell(t,x,\bfp,\bfq)-\ell(t,y,\bfp,\bfq)}\le L_f 
\sqrt{
\abs{x(t)-y(t)}^2+\int_0^t \abs{x(s)-y(s)}^2\,ds
}.
\end{align*}

By Proposition~\ref{P:EU_ODE}, we have immediately
the following result.

\begin{proposition}
Suppose that \emph{\textbf{H}$(f)_2$} holds.
Let $(t_0,x_0)\in [0,T)\times\Omega$, $a\in\mathcal{A}^{t_0}$, and
$b\in\mathcal{B}^{t_0}$.
Then the initial-value problem \eqref{E:GameEq} has a unique
solution in $\mathcal{X}^{L_f}(t_0,x_0)$.
\end{proposition}

Next, we present the precise setup for our differential game.
To this end, fix $t_0\in [0,T]$.
The spaces
of \emph{feedback strategies} for the controller and the disturbance are defined as follows:
\begin{align*}
\mathbb{A}^{t_0}&:=\{\mathbf{a}:[t_0,T]\times\Omega\to P\text{ non-anticipating}\},\\
\mathbb{B}^{t_0}&:=\{\mathbf{b}:[t_0,T]\times\Omega\to Q\text{ non-anticipating}\}.
\end{align*}
\begin{remark}
Note that elements of $\mathbb{A}^{t_0}$ and $\mathbb{B}^{t_0}$ 
 do not need to be measurable functions.
 \end{remark}

Let us first consider the controller's point of view. For every $\delta>0$, every
partition $\pi: t_0<t_1<\cdots<t_n=T$ with mesh size $\abs{\pi}\le\delta$,  
every $x_0\in\Omega$,
every feedback strategy $\mathbf{a}\in\mathbb{A}^{t_0}$, and every control $b\in\mathcal{B}^{t_0}$, define a 
\emph{step-by-step feedback control} $a=a(\cdot;\pi,x_0,\mathbf{a},b)\in\mathcal{A}^{t_0}$ 
recursively by\begin{align*}
a(t):=\mathbf{a}(t_i,x^{t_0,x_0,a,b}),\quad t\in [t_i,t_{i+1}),\quad i=0,\ldots, n-1.
\end{align*}
\begin{remark}
Note that $a(t)$, $t\in [t_i,t_{i+1})$, depends only on $b$ via 
the state trajectory $\{x^{t_0,x_0,a,b}_s\}_{s\in [0,t_i]}$.
\end{remark}
\noindent 
The  \emph{guaranteed result of the controller} 
given a strategy $\mathbf{a}\in\mathbb{A}^{t_0}$ and a state $x_0\in\Omega$
is defined by
\begin{align*}
J_a(t_0,x_0;\mathbf{a}):=\inf_{\delta>0} \sup_{\abs{\pi}\le\delta,\,b\in\mathcal{B}^{t_0}}
&\Biggl[
\int_{t_0}^T \ell(t,x^{t_0,x_0,a(\cdot;\pi,x_0,\mathbf{a},b),b},a(t;\pi,x_0,\mathbf{a},b),b(t))
\\ &\qquad+
 h\left(x^{t_0,x_0,a(\cdot;\pi,x_0,\mathbf{a},b),b}\right)
 \Biggr].
 \end{align*}
The \emph{optimal guaranteed result of the controller} $v_a:[0,T]\times\Omega\to\R$
is defined by
\begin{align*}
v_a(t_0,x_0):=\inf_{\mathbf{a}\in\mathbb{A}^{t_0}} J_a(t_0,x_0;\mathbf{a}).
\end{align*}

Let us next consider the disturbance's point of view. For every $\delta>0$, every
partition $\pi: t_0<t_1<\cdots<t_n=T$ with mesh size $\abs{\pi}\le\delta$, 
every $x_0\in\Omega$,
every feedback strategy $\mathbf{b}\in\mathbb{B}^{t_0}$, and every control $a\in\mathcal{A}^{t_0}$, define a 
\emph{step-by-step feedback control} $b=b(\cdot;\pi,x_0,a,\mathbf{b})\in\mathcal{B}^{t_0}$ 
recursively by
\begin{align*}
b(t):=\mathbf{b}(t_i,x^{t_0,x_0,a,b}),\quad t\in [t_i,t_{i+1}),\quad i=0,\ldots, n-1.
\end{align*}
The  \emph{guaranteed result of the disturbance} given  a strategy $\mathbf{b}\in\mathbb{B}^{t_0}$
and a state $x_0\in\Omega$
is defined by
\begin{align*}
J_b(t_0,x_0;\mathbf{b}):=\sup_{\delta>0} \inf_{\abs{\pi}\le\delta,\,a\in\mathcal{A}^{t_0}}
&\Biggl[
\int_{t_0}^T \ell(t,x^{t_0,x_0,a,b(\cdot;\pi,x_0,a,\mathbf{b})},a(t),b(t;\pi,x_0,a,\mathbf{b}))\\ &\qquad+
 h\left(x^{t_0,x_0,a,b(\cdot;\pi,x_0,a,\mathbf{b})}\right)
 \Biggr].
 \end{align*}
The \emph{optimal guaranteed result of the disturbance} $v_b:[0,T]\times\Omega\to\R$ is defined by
\begin{align*}
v_b(t_0,x_0):=\sup_{\mathbf{b}\in\mathbb{B}^{t_0}} J_b(t_0,x_0;\mathbf{b}).
\end{align*}

\begin{definition}
If $v_a=v_b$, then we call $v:=v_a=v_b$ the \emph{value} of our game.
\end{definition}

Let us also mention a perhaps more intuitive definition for the optimal guaranteed results.
Let $(t_0,x_0)\in [0,T]\times\Omega$.
Denoting by $\Pi^{t_0}$ the set of all (finite) partitions of $[t_0,T]$, put
\begin{align*}
v_a^\ast(t_0,x_0)&:=\inf_{(\mathbf{a},\pi)\in\mathbb{A}^{t_0}\times\Pi^{t_0}} 
 \sup_{b\in\mathcal{B}^{t_0}}
 \Biggl[
\int_{t_0}^T \ell(t,x^{t_0,x_0,a(\cdot;\pi,x_0,\mathbf{a},b),b},a(t;\pi,x_0,\mathbf{a},b),b(t))
\\ &\qquad\qquad\qquad\qquad\qquad\qquad+
 h\left(x^{t_0,x_0,a(\cdot;\pi,x_0,\mathbf{a},b),b}\right)\Biggr],\\
 v_b^\ast(t_0,x_0)&:=\sup_{(\mathbf{b},\pi)\in\mathbb{B}^{t_0}\times\Pi^{t_0}}
  \inf_{a\in\mathcal{A}^{t_0}}
  \Biggl[
\int_{t_0}^T \ell(t,x^{t_0,x_0,a,b(\cdot;\pi,x_0,a,\mathbf{b})},a(t),b(t;\pi,x_0,a,\mathbf{b})
)\\ &\qquad\qquad\qquad\qquad\qquad\qquad+
 h\left(x^{t_0,x_0,a,b(\cdot;\pi,x_0,a,\mathbf{b})}\right)\Biggr].
\end{align*}
\begin{proposition} \label{P:gameValue}
We have
$v_b\le v^\ast_b\le v^\ast_a\le v_a$, i.e., if our game has value, both definitions
for the optimal guaranteed results coincide.
\end{proposition}
\begin{proof}
We show only $v_b\le v_a$. The inequality $v_b^\ast\le v_a^\ast$ can be shown similarly
and the remaining inequalities $v_b\le v_b^\ast$ and $v_a^\ast\le v_a$ follow immediately
from the definitions.
Fix $(t_0,x_0)\in [0,T]\times\Omega$. Let $\mathbf{a}\in\mathbb{A}^{t_0}$,
$\mathbf{b}\in\mathbb{B}^{t_0}$. Since
for every partition $\pi$ of $[t_0,T]$,
\begin{align*}
\{x^{t_0,x_0,a,b(\cdot;\pi,x_0,a,\mathbf{b})}: a\in\mathcal{A}^{t_0}\}\cap
\{x^{t_0,x_0,a(\cdot;\pi,x_0,\mathbf{a},b),b}: b\in\mathcal{B}^{t_0}\}
\neq\emptyset,
\end{align*}
we have
\begin{align*}
&\inf_{a\in\mathcal{A}^{t_0}} 
\Biggl[
\int_{t_0}^T \ell(t,x^{t_0,x_0,a,b(\cdot;\pi,x_0,a,\mathbf{b})},a(t),b(t;\pi,x_0,a,\mathbf{b}))+
 h\left(x^{t_0,x_0,a,b(\cdot;\pi,x_0,a,\mathbf{b})}\right)\Biggr]\\
 &\le \sup_{b\in\mathcal{B}^{t_0}} 
 \Biggl[
\int_{t_0}^T \ell(t,x^{t_0,x_0,a(\cdot;\pi,x_0,\mathbf{a},b),b},a(t;\pi,x_0,\mathbf{a},b),b(t))
+ h\left(x^{t_0,x_0,a(\cdot;\pi,x_0,\mathbf{a},b),b}\right)\Biggr].
\end{align*}
Thus $J_b(t_0,x_0;\mathbf{b})\le J_a(t_0,x_0;\mathbf{a})$, and since
$\mathbf{a}$ and $\mathbf{b}$ were arbitrary feedback strategies, 
we obtain $v_b(t_0,x_0)\le v_a(t_0,x_0)$.
\end{proof}
The  Isaacs equation for our differential game is given by \eqref{E:PPDE2} with
\begin{align}\label{E:IsaacsF}
F(t,x,z):=\min_{\bfp\in P}\max_{\bfq\in Q} \left[\ell(t,x,\bfp,\bfq)+(f(t,x,\bfp,\bfq),z)\right].
\end{align}

\begin{assumption} [Isaacs condition]
For every $(t,x,z)\in [0,T]\times C([0,T],H)\times H$,
\begin{align*}
\min_{\bfp\in P}\max_{\bfq\in Q} \left[\ell(t,x,\bfp,\bfq)+(f(t,x,\bfp,\bfq),z)\right]
=\max_{\bfq\in Q}\min_{\bfp\in P} \left[\ell(t,x,\bfp,\bfq)+(f(t,x,\bfp,\bfq),z)\right].
\end{align*}
\end{assumption}

It can easily be shown that  then \textbf{H}$(\ell)_2$ and
\textbf{H}$(f)_2$ imply \textbf{H}($F$)~(i) and (ii) with $L_0=L_f$.
Moreover, similarly as in \cite[Section~11.1, pp.~116ff.]{Subbotin_book},
one can show that in this case, 
for a.e.~$t\in (0,T)$, and for every $x$, $y\in\Omega$,
and every $z\in H$,
\begin{equation}\label{E:HF4stronger}
\begin{split}
&\abs{F(t,x,z)-F(t,y,z)}
\le 
L_f(1+\abs{z})
\sqrt{\abs{x(t)-y(t)}^2+\int_0^t \abs{x(s)-y(s)}^2\,ds},
\end{split}
\end{equation}
i.e., we also have \textbf{H}($F$)~(iii).

\begin{theorem}
Suppose that \emph{\textbf{H}$(f)_2$}, \emph{\textbf{H}$(\ell)_2$}, and
\emph{\textbf{H}$(h)_1$} hold.
Let $u$ be the minimax solution of \eqref{E:PPDE2} on $[0,T]\times\Omega$
with $F$ given by \eqref{E:IsaacsF}.
Then our game has value and $u=v$.
\end{theorem}

\begin{proof} 
First, recall that, by Theorem~\ref{T:ExistenceMinimax2}, the function $u$ is actually well-defined.
Next, note that  $v_b\le v_a$ (see Proposition~\ref{P:gameValue}),
Thus, we only need to show that $v_a\le u\le v_b$. 
Fix $(t_0,x_0)\in [0,T)\times\Omega$.
We prove only $v_a(t_0,x_0)\le u(t_0,x_0)$.
To this end, it suffices to   specify a family $(\mathbf{a}^\eps)_{\eps>0}$ of strategies in $\mathbb{A}^{t_0}$
to show that $J_a(t_0,x_0;\mathbf{a}^\eps)\le u(t_0,x_0)+m_a(\eps)$, $\eps>0$
for some modulus of continuity $m_a$.
In order to construct such a family, we shall use auxiliary functions $\alpha^\eps:[t_0,T]\to\R$,
$\beta^\eps:[t_0,T]\times C([0,T],H)\to\R$, and $\nu^\eps:[t_0,T]\times C([0,T],H)\to\R$, $\eps>0$,
defined by
\begin{align*}
\alpha^\eps(t)&:=\frac{e^{-2L_f(t-t_0)}-\eps}{\eps},\\
\beta^\eps(t,x)&:=\sqrt{\eps^4+\abs{x(t)}^2+2L_f\int_0^t \abs{x(s)}^2\,ds},\\
\nu^\eps(t,x)&:=\alpha^\eps(t)\beta^\eps(t,x).
\end{align*}
Moreover, put 
\begin{align*}
\eps_0:=e^{-2L_f(T-t_0)}.
\end{align*}
Then, for every $\eps\in (0,\eps_0)$, we claim that  $\nu^\eps\in\mathcal{C}_V^{1,1}([t_0,T]\times C([0,T],H))$ with
\begin{align*}
\partial_t\nu^\eps(t,x)&=-2L_0\frac{e^{-2L_f(t-t_0)}}{\eps}\beta^\eps(t,x)
+L_f\frac{\alpha^\eps(t)}{\beta^\eps(t,x)} \abs{x(t)}^2,\\
\partial_x\nu^\eps(t,x)&=\frac{\alpha^\eps(t)}{\beta^\eps(t,x)}\,x(t).
\end{align*}
To show this claim, write $\nu^\eps(t,x)$ in the form
$\psi(t,\xi(t),x(t))$, where
\begin{align*}
\psi(t,r,v)&=\alpha^\eps(t)\sqrt{\eps^4+\abs{v}^2+2L_f r},\quad (t,r,v)\in [t_0,T]\times\R\times H,\\
\xi(t)&=\int_0^t \abs{x(s)}^2\,ds,\quad t\in [0,T].
\end{align*}
Then, one can readily verify that $\psi\in\mathcal{C}^{1,1,1}_V([t_0,T]\times\R\times H)$
and  an application of Proposition~\ref{P:ChainRule} yields our claim.


Now we are in the position to specify our strategies. 
For every $\eps\in (0,\eps_0)$, define
$u^\eps_a:[t_0,T]\times\Omega\to\R$ by
\begin{align}\label{E:uepsGame}
u^\eps_a(t,x):=\min_{\tilde{x}\in\mathcal{X}^{L_f}(t_0,x_0)} \left[
u(t,\tilde{x})+\nu^\eps(t,x-\tilde{x})
\right].
\end{align}
For every $(t,x)\in [t_0,T]\times\Omega$, choose a 
path $x^{\eps,t,x}_a\in\mathcal{X}^{L_f}(t_0,x_0)$ for which the minimum in 
\eqref{E:uepsGame} is attained.
Then we can define strategies $\mathbf{a}^\eps\in\mathbb{A}^{t_0}$  as follows.
For every $(t,x)\in [t_0,T]\times\Omega$, choose an $\mathbf{a}^\eps(t,x)\in P$ such that
\begin{align*}
&\max_{\bfq\in Q}\left[\ell(t,x,\mathbf{a}^\eps(t,x),\bfq)+
(f(t,x,\mathbf{a}^\eps(t,x),\bfq),\partial_x \nu^\eps(t,x-x^{\eps,t,x}_a))\right]\\
&\qquad\qquad=
\min_{\bfp\in P}\max_{\bfq\in Q}\left[
\ell(t,x,\bfp,\bfq)+
(f(t,x,\bfp,\bfq),\partial_x \nu^\eps(t,x-x^{\eps,t,x}_a))\right].
\end{align*}

Given $\delta>0$, a partition $\pi$ of $[t_0,T]$ with $\abs{\pi}\le\delta$, and 
a control $b\in\mathcal{B}^{t_0}$,
put $x^{\pi,b}:=x^{t_0,x_0,a(\cdot;\pi,x_0,\mathbf{a}^\eps,b),b}$.

We shall show the following statement:
\begin{quotation}
\textbf{S.} For every $i\in\{0,\ldots,n-1\}$,
\begin{align*}
\int_{t_i}^{t_{i+1}} \ell(t,x^{\pi,b},\mathbf{a}^{\eps}(t_i,x^{\pi,b}),b(t))\,dt+
u_a^\eps(t_{i+1},x^{\pi,b})-u_a^\eps(t_i,x^{\pi,b})\le m(\delta)\cdot (t_{i+1}-t_i)
\end{align*}
for some modulus of continuity $m$.
\end{quotation}

According to \eqref{E:uepsGame}, for every $\tilde{x}\in\mathcal{X}^{L_f}(t_0,x_0)$,
\begin{equation}\label{E:ueps2Game}
\begin{split}
u_a^\eps(t_{i+1},x^{\pi,b})-u_a^\eps(t_i,x^{\pi,b})&\le 
[u(t_{i+1},\tilde{x})+\nu^\eps(t_{i+1},x^{\pi,b}-\tilde{x})]\\ &\qquad -
[u(t_i,{x}_{a,i}^\eps)+\nu^\eps(t_i,x^{\pi,b}-{x}_{a,i}^\eps)],
\end{split}
\end{equation}
where
$x^\eps_{a,i}:=x^{\eps,t_i,x^{\pi,b}}_a\in\mathcal{X}^{L_f}(t_0,x_0)$
is a path for which  the minimum in 
\eqref{E:uepsGame}  with $(t,x)$ replaced by 
$(t_i, x^{\pi,b})$ is attained.

Note that
$u$ is also a minimax $L_f$-solution of \eqref{E:PPDE2} on $[0,T]\times\Omega$.
Thus we can assume that
$\tilde{x}\in\mathcal{X}^{L_f}(t_i,{x}^\eps_{a,i})$ and that
\begin{align*}
u(t_{i+1},\tilde{x})-u(t_i,{x}^\eps_{a,i})&=\int_{t_i}^{t_{i+1}}
[(f^{\tilde{x}}(t),\partial_x\nu^\eps (t_i,x^{\pi,b}-{x}_{a,i}^\eps))\\ &\qquad\qquad-
F(t,\tilde{x},\partial_x\nu^\eps (t_i,x^{\pi,b}-{x}_{a,i}^\eps))]\,dt.
\end{align*}
Then \eqref{E:ueps2Game} becomes
\begin{align*}
&u_a^\eps(t_{i+1},x^{\pi,b})-u_a^\eps(t_i,x^{\pi,b})\\&\qquad\le
[u(t_{i+1},\tilde{x})-u(t_i,x^\eps_{a,i})]+
[\nu^\eps(t_{i+1},x^{\pi,b}-\tilde{x})-\nu^\eps(t_i,x^{\pi,b}-x^\eps_{a,i})]\\
&\qquad\le \int_{t_i}^{t_{i+1}} [(f^{\tilde{x}}(t),\partial_x\nu^\eps(t_i,x^{\pi,b}-x^{\eps}_{a,i}))
-F(t,\tilde{x},\partial_x\nu^\eps(t_i,x^{\pi,b}-x^{\eps}_{a,i}))]\,dt\\
&\qquad\qquad +
\int_{t_i}^{t_{i+1}} [\partial_t\nu^\eps(t,x^{\pi,b}-\tilde{x})+
\langle A(t,\tilde{\wx}(t))-A(t,\wx^{\pi,b}(t)),
\partial_x \nu^\eps(t,x^{\pi,b}-\tilde{x})
\rangle\\
&\qquad\qquad\qquad +
(f(t,x^{\pi,b},\mathbf{a}^\eps(t_i,x^{\pi,b}),b(t))
-f^{\tilde{x}}(t)),\partial_x\nu^\eps(t,x^{\pi,b}-\tilde{x}))]\,dt,
\end{align*}
i.e.,
\begin{align*}
&u_a^\eps(t_{i+1},x^{\pi,b})-u_a^\eps(t_i,x^{\pi,b})\le
\int_{t_i}^{t_{i+1}} [\partial_t\nu^\eps(t,x^{\pi,b}-\tilde{x})\\
&\qquad +(f(t,x^{\pi,b},\mathbf{a}^\eps(t_i,x^{\pi,b}),b(t)),\partial_x\nu^\eps(t,x^{\pi,b}-\tilde{x}))\\
&\qquad - F(t,\tilde{x},\partial_x\nu^\eps(t_i,x^{\pi,b}-{x}^\eps_{a,i}))\\
&\qquad + (f^{\tilde{x}}(s),\partial_x\nu^\eps(t_i,x^{\pi,b}-{x}^\eps_{a,i})-
\partial_x\nu^\eps(t,x^{\pi,b}-\tilde{x}))\\
&\qquad + \langle A(t,\tilde{\wx}(t))-A(t,\wx^{\pi,b}(t)),\partial_x\nu^\eps(t,x^{\pi,b}-\tilde{x})\rangle]\,dt.
\end{align*}
\allowdisplaybreaks
Since $f$ and $\ell$ as well as suitable compositions of $\nu^\eps$,
$\partial_t \nu^\eps$, and $\partial_x \nu^\eps$ are uniformly continuous
on the compact set $\Omega^{L_f}$ (and also on finite Cartesian products of said set),
since $\mathcal{X}^{L_f}(t_0,x_0)$ is equicontinuous
(see \cite[Theorem~7.16, p.~233]{Kelley}),
since $A$ is monotone, and since
\begin{align*}
&\ell(t_i,x^{\pi,b},\mathbf{a}^\eps(t_i,x^{\pi,b}),b(t))
+(f(t_i,x^{\pi,b},\mathbf{a}^\eps(t_i,x^{\pi,b}),b(t)),\partial_x\nu^\eps(t_i,x^{\pi,b}-\tilde{x}))\\
&\quad \le
\max_{q\in Q} 
\left[\ell(
t_i,x^{\pi,b},\mathbf{a}^\eps(t_i,x^{\pi,b}),q
)+
(f(t_i,x^{\pi,b},\mathbf{a}^\eps(t_i,x^{\pi,b}),q),
\partial_x\nu^\eps(t_i,x^{\pi,b}-\tilde{x}))
\right]\\
&\quad =
\min_{p\in P}
\max_{q\in Q} \left[
\ell(t_i,x^{\pi,b},p,q)+
(f(t_i,x^{\pi,b},p,q),
\partial_x\nu^\eps(t_i,x^{\pi,b}-\tilde{x})
\right]\\
&\quad=F(t_i,x^{\pi,b}, \partial_x\nu^\eps(t_i,x^{\pi,b}-\tilde{x})),
\end{align*}
there exists a modulus $m$ such that
\begin{align*}
&\int_{t_i}^{t_{i+1}} \ell(t,x^{\pi,b},\mathbf{a}^{\eps}(t_i,x^{\pi,b}),b(t)\,dt+
u_a^\eps(t_{i+1},x^{\pi,b})-u_a^\eps(t_i,x^{\pi,b}) \\&\le
(t_{i+1}-t_i)\cdot\{m(\delta)+
[\partial_t\nu^\eps(t_i,x^{\pi,b}-x^\eps_{a,i})+
F(t_i,x^{\pi,b}, \partial_x\nu^\eps(t_i,x^{\pi,b}-{x^\eps_{a,i}}))\\ &\qquad\qquad-
F(t_i,x^\eps_{a,i},\partial_x\nu^\eps(t_i,x^{\pi,b}-{x}^\eps_{a,i}))]\}\\
&\le  (t_{i+1}-t_i)\cdot\Biggl\{m(\delta)-2L_f\frac{
e^{-2L_f(t-t_0)}
}{
\eps
}
\beta^\eps(t,y)+L_f\frac{\alpha^\eps(t)}{\beta^\eps(t,y)}\abs{y(t)}^2\\
&\qquad +L_f\left(
1+\frac{\alpha^\eps(t)}{\beta^\eps(t,y)}\abs{y(t)}
\right)\sqrt{\abs{y(t)}^2+\int_0^t \abs{ y(s)}^2\,ds} \Biggr\}\qquad\qquad\qquad
\text{by \eqref{E:HF4stronger}}\\
&\le  (t_{i+1}-t_i)\cdot\Biggl\{m(\delta) -2L_f\frac{
e^{-2L_f(t-t_0)}
}{
\eps
}
\beta^\eps(t,y)\\
&\qquad +L_f\left(
1+2 \frac{\alpha^\eps(t)}{\beta^\eps(t,y)}\abs{y(t)}
\right)\sqrt{\abs{y(t)}^2+\int_0^t \abs{y(s)}^2\,ds}\Biggr\}\\
&\le  (t_{i+1}-t_i)\cdot\Biggl\{m(\delta)-2L_f\frac{
e^{-2L_f(t-t_0)}
}{
\eps
}
\beta^\eps(t,\delta x) 
\\&\qquad +L_f\left(1+2\alpha^\eps(t)\right)
\sqrt{\abs{y(t)}^2+\int_0^t \abs{y(s)}^2\,ds}\Biggr\}\\
&\le (t_{i+1}-t_i)\cdot\Biggl\{m(\delta) + L_f \beta^\eps(t,x)\\
&\qquad\cdot \left[-2\frac{
e^{-2L_f(t-t_0)}
}{
\eps
}
+\frac{
e^{-2L_f(t-t_0)}
}{
\eps
}
+\frac{
e^{-2L_f(t-t_0)}
}{
\eps
}
-1
\right]\Biggr\}\\
&\le (t_{i+1}-t_i)\,m(\delta),
\end{align*}
where $y:=x^{\pi,b}-x_{a,i}^\eps$,
and thus Statement \textbf{S} has been proven. 

Next, note that
\begin{align*}
J_a(t_0,x_0;\mathbf{a}^\eps)&\le
\sup_{\abs{\pi}\le\delta,\,b\in\mathcal{B}^{t_0}} \left[
\int_{t_0}^T \ell(t,x^{\pi,b},\mathbf{a}^\eps(t,x^{\pi,b}),b(t))\,dt+
h(x^{\pi,b})
\right]\\ &=
\sup_{\abs{\pi}\le\delta,\,b\in\mathcal{B}^{t_0}}\left[
\int_{t_0}^T \ell(t,x^{\pi,b},\mathbf{a}^\eps(t,x^{\pi,b}),b(t))\,dt+
 u(T,x^{\pi,b})\right].
\end{align*}
Assuming for a moment that there exists a modulus of continuity
$m_a$ independent from $\pi$ and $b$ such that
\begin{align}\label{E:ueps3Game}
u(T,x^{\pi,b})\le u_a^\eps(T,x^{\pi,b}) +m_a(\eps)
\end{align}
is valid, we can from the definition of $\nu^\eps$ and from \textbf{S} infer that, for every $\delta>0$,
\begin{align*}
&\int_{t_0}^T \ell(t,x^{\pi,b},\mathbf{a}^\eps(t,x^{\pi,b}),b(t))\,dt
+u(T,x^{\pi,b})\\
&\qquad\le u_a^\eps(t_0,x_0)+m_a(\eps)\\
&\qquad\qquad +\sum_{i=0}^{n-1}
\left[
\int_{t_i}^{t_{i+1}} \ell(t,x^{\pi,b},\mathbf{a}^{\eps}(t_i,x^{\pi,b}),b(t)\,dt+
u_a^\eps(t_{i+1},x^{\pi,b})-u_a^\eps(t_i,x^{\pi,b})\right]\\
&\qquad \le [u(t_0,x_0)+(1-\eps)\eps]+m(\delta)\cdot (T-t_0)+m_a(\eps),
\end{align*}
which would imply that 
\begin{align*}
J_a(t_0,x_0;\mathbf{a}^\eps)&\le u(t_0,x_0)+(1-\eps)\eps+m_a(\eps)
\end{align*}
and this would conclude the proof.

Therefore it remains to find a modulus $m_a$ such that \eqref{E:ueps3Game} holds.
We have
\begin{align*}
u_a^\eps(T,x)&\le u(T,x)+\nu^\eps (T,x-x)\le
\max_{\tilde{x}\in\mathcal{X}(t_0,x_0)} u(T,\tilde{x})+(1-\eps)\eps,\\
u_a^\eps(T,x)&= u(T,x^\eps)+\nu^\eps(T,x-x^\eps)\ge
-\max_{\tilde{x}\in\mathcal{X}(t_0,x_0)} u(T,\tilde{x})+\nu^\eps(T,x-x^\eps),
\end{align*}
where $x^\eps$ is an appropriate minimizer.
Thus
\begin{align*}
\nu^\eps(T,x-x^\eps)&\le u_a^\eps(T,x)+\max_{\tilde{x}\in\mathcal{X}(t_0,x_0)} u(T,\tilde{x})\\
&\le 2\max_{\tilde{x}\in\mathcal{X}^{L_f}(t_0,x_0)} u(T,\tilde{x})+(1-\eps)\eps\\
&=: C_1+(1-\eps)(1-\eps),
\end{align*}
which implies, by the definition of $\nu^\eps$,
\begin{align*}
[\beta^\eps(t,x)]^2&=\eps^4+\abs{x(T)-x^\eps(T)}^2+2L_f \int_0^T\abs{x(t)-x^\eps(t)}^2\,dt\\ &\le
\left[\frac{C_1+(1-\eps)(\eps)}{\alpha^\eps(T)}\right]^2
\end{align*}
and therefore, proceeding along the lines of 
\eqref{E:CompactEndProof}, we obtain
\begin{align*}
\norm{x-x^\eps}_\infty &\le 2C \sqrt{\norm{x-x^\eps}_{L^2(t_0,T;H)}}\\
&\le 2C\sqrt{
(2L_f)^{-1/2} \left[
\frac{
\eps(C_1+(1-\eps)(\eps))
}{
e^{-2L_f(T-t_0)}-\eps
}
\right]
},
\end{align*}
where $C=C(t_0,x_0,L_f)$ is the constant from the a-priori estimates of Lemma~\ref{L:Apriori}.
Hence, together with $h$ being uniformly continuous on the compact set $\Omega^{L_f}$,
we can infer the existence of a modulus $m_a$ such that
\begin{align*}
u_a^\eps(T,x)&=u(T,x^\eps)+\nu^\eps(T,x-x^\eps)\ge u(T,x^\eps)=h(x^\eps)\ge h(x)-m_a(\eps),
\end{align*}
i.e., we have \eqref{E:ueps3Game} and thus our proof is complete.
\end{proof}

\appendix

\section{Properties of solution sets of evolution equations} 

We shall use the Nemytsky operator $\hat{A}:L^p(t_0,T;V)\to L^q(t_0,T;V^\ast)$
corresponding to $A$, which is defined
by 
\begin{align*}
(\hat{A}\mathbf{x})(t):=A(t,\mathbf{x}(t)), \quad\mathbf{x}\in
L^p(t_0,T;V), \quad t\in (t_0,T).
\end{align*}

The  derivation of a-priori estimates like the following is standard in the literature.
Nevertheless, for the convenience of the reader, we provide a complete proof.
\begin{lemma}\label{L:Apriori}
Fix $(t_0,x_0)\in [0,T)\times C([0,T],H)$ and $L\ge 0$.
There exists a  constant $C=C(t_0,x_0,L)>0$  such that,
for all $x\in\mathcal{X}^L(t_0,x_0)$,
\begin{align*}
\norm{\mathbf{x}}_{W_{pq}(t_0,T)}+
\norm{x}_{\infty}+
\norm{\hat{A}\mathbf{x}}_{L^q(t_0,T;V^\ast)}+
\norm{f^x}_{L^2(t_0,T;H)}\le C.
\end{align*}
\end{lemma}

\begin{proof}
Let $x\in\mathcal{X}^L(t_0,x_0)$. Then
\begin{align}\label{E:LCompact1}
\langle\mathbf{x}^\prime(t),\mathbf{x}(t)\rangle+
\langle A(t,\mathbf{x}(t)),\mathbf{x}(t)\rangle=(f^x(t),x(t))\quad\text{a.e.}
\end{align}
By  integration-of-parts, the coercivity of $A$, and Young's inequality, 
for every $\eps>0$,
\begin{equation}\label{E:LCompact2}
\begin{split}
\frac{1}{2}\frac{d}{dt}\abs{x(t)}^2+c_2\norm{\mathbf{x}(t)}^p&\le \abs{f^x(t)}\,\abs{x(t)} 
\le \frac{1}{q\eps^q} \abs{f^x(t)}^q+\frac{\eps^p \,C_1^p}{p} \norm{\mathbf{x}(t)}^p\quad\text{a.e.}
\end{split}
\end{equation}
Here, the constant $C_1>0$ comes from the continuous embedding of $V$ into $H$,
i.e., $\abs{v}\le C_1\norm{v}$ for all $v\in V$.
Now, choose $\eps>0$ such that $c_2=p^{-1}\,\eps^p\, C_1^p$. Then
\begin{align}\label{E:LCompact3}
\frac{1}{2}\frac{d}{dt}\abs{x(t)}^2\le
\frac{L^q}{q\eps^q}\left(1+\sup_{s\le t}\abs{x(s)}\right)^q.
\end{align}
Put
$m(t):=\sup_{s\le t} \abs{x(s)}$, 
$t\in [t_0,T]$.
Then integrating \eqref{E:LCompact3} 
yields
\begin{align*}
x(t)^2&\le m(t_0)^2+\int_{t_0}^t \frac{2 L^q}{q\eps^q} 
[1+m(s)]^q
\,ds.
\end{align*}
Noting that $p\ge 2$ implies $q\le 2$ and thus
$[1+m(s)]^q\le [1+m(s)]^2$, we get
\begin{align*}
m(t)^2&\le m(t_0)^2+\int_{t_0}^t \frac{2 L^q}{q\eps^q} 
[ 2+2m(s)^2]\,ds\\
&\le \left[
m(t_0)+\frac{4L^q (T-t_0)}{q\eps^q}
\right]+\int_{t_0}^t \frac{4L^q}{q\eps^q} m(s)^2\,ds.
\end{align*}
By Gronwall's inequality,
\begin{align}\label{E:LCompact4}
m(t)^2\le\left[
m(t_0)+\frac{4L^q (T-t_0)}{q\eps^q}
\right]\,\exp\left(
\frac{4L^q(T-t_0)}{q\eps^q}
\right)=:C_2^2,
\end{align}
i.e., $\norm{x}_{\infty}\le C_2=C_2(t_0,x_0,L)$.
Furthermore,
\begin{align*}
&\int_{t_0}^T\norm{\mathbf{x}(t)}^p\,dt
\le \frac{1}{c_2}
\int_{t_0}^t \langle A(t,\mathbf{x}(t)),x(t)\rangle\,dt
&&\text{since $A$ is coercive}\\
&\qquad=\frac{1}{c_2}\left[
\int_{t_0}^T (f^x(t),x(t))\,dt+\frac{1}{2}( \abs{x_0(t_0)}^2-\abs{x(T)}^2)
\right] 
&& \text{by \eqref{E:LCompact1}}\\
&\qquad\le \frac{1}{c_2} \left[
\int_{t_0}^T  L(1+m(t))\,m(t)\,dt+\frac{1}{2}m(t_0)^2
\right]&& \\
&\qquad \le \frac{(T-t_0)L(1+C_2)C_2+\frac{1}{2} C_2^2}{c_2}=:C_3^p
&&\text{by \eqref{E:LCompact4},}
\end{align*}
i.e.,
\begin{align}\label{E:LCompact5}
\norm{\mathbf{x}}_{L^p(t_0,T;V)}\le C_3=
C_3(t_0,x_0,L).
\end{align}
Moreover, by the boundedness of $A$ and since $(p-1)q=p$,
\begin{align*}
\int_{t_0}^T\norm{A(t,\mathbf{x}(t)}^q_\ast\,dt&\le
\int_{t_0}^T\abs{a_1(t)}^q+c_1\norm{\mathbf{x}(t)}^{(p-1)q}\,dt &&\\
&\le \norm{a_1}^q_{L^q}+c_1 C_3^p=:C_4^q &&\text{by \eqref{E:LCompact5}.}
\end{align*}
Finally,
\begin{align*}
\left[
\int_{t_0}^T\norm{\mathbf{x}^\prime(t)}_\ast^q\,dt
\right]^{1/q}&\le
\left[\int_{t_0}^T \norm{A(t,\mathbf{x}(t))}_\ast^q
\right]^{1/q}+\left[
\int_{t_0}^T \norm{f^x(t)}_\ast^q
\right]^{1/q}\\
&\le C_4+\left[\int_{t_0}^T C_5 \abs{f^x(t)}^q\,dt\right]^{1/q}\\
&\le C_4+C_5^{1/q} (T-t_0)^{1/q}\left[(1+m(T))L\right]\\
&\le C_4+C_5^{1/q} (T-t_0)^{1/q}\left[(1+C_2)L\right]=:C_6
&&\text{by \eqref{E:LCompact4}}
\end{align*}
and
\begin{align*}
\left(\int_{t_0}^T \abs{f^x(t)}^2\,dt\right)^{1/2}\le (T-t_0)^{1/2} L(1+C_2)=:C_7\text{ by \eqref{E:LCompact4}.}
\end{align*}
Here, the constant $C_5>0$ comes from the continuous embedding of $H$ into $V^\ast$,
i.e., $\norm{h}_\ast\le C_5\abs{h}$ for all $h\in H$.
\end{proof}

\begin{lemma}\label{L:CompactW}
Fix $(t_0,x_0)\in [0,T)\times C([0,T],H)$ and $L\ge 0$.
Consider a sequence $(x_n)_n$ in $\mathcal{X}^L(t_0,x_0)$.
Then there exist a pair $(\wx,f)\in W_{pq}(t_0,T)\times L^2(t_0,T;H)$
with $\wx^\prime+\hat{A}\wx=f$ and a subsequence $(x_{n_k})_k$ of $(x_n)_n$
such that $\wx_{n_k}\xrightarrow{w} \wx$  in $W_{pq}(t_0,T)$,
$x_{n_k}\to x$ in $C([0,T],H)$, and $f^{x_n}\xrightarrow{w} f$ in $L^2(t_0,T;H)$. 
Here, $x$ is the unique element 
of $C([0,T],H)$ such that $x=\wx$ a.e.~on $(t_0,T)$ and $x=x_0$ on $[0,t_0]$.
\end{lemma}

\begin{proof}
Major parts of this proof are slight modifications of the proofs of 
\cite[Lemma~30.5 and Lemma~30.6, p.~776f.]{ZeidlerIIB},
\cite[Assertion~2.1]{Tolstonogov99Sbornik}, and
\cite[Proposition~3.1]{TT1999NoDEA} (cf.~also the proofs of \cite[Theorem~4.1]{ZCYC13}
and \cite[Lemma~4.3]{LiuLiu14}).
First, recall that each $x_n$ satisfies  $\mathbf{x}^\prime_n+\hat{A}\mathbf{x}_n=f_n$,
where $f_n:=f^{x_n}$.
Also note that $W_{pq}(t_0,T)$ is reflexive (see, e.g.~\cite[Theorem~1.12 (a), p.~4]{HPHandbookII}).
Thus, using the the a-priori estimates from Lemma~\ref{L:Apriori}, we can assume,
by  the Banach-Alaoglu theorem 
and  the Eberlein-\v{S}mulian theorem,
that
$(\mathbf{x}_n)_n$ converges weakly to some $\mathbf{x}\in W_{pq}(t_0,T)$.
Then $(\mathbf{x}_n)_n$ converges also strongly to $\mathbf{x}$ in 
$L^p(t_0,T;H)$ because $W_{pq}(t_0,T)$ is compactly embedded into $L^p(t_0,T;H)$ 
(see \cite[Th\'eor\`eme~5.1, p.~58]{JLLions1969}).
Similarly, we can also assume that $(\hat{A}\mathbf{x}_n)_n$ converges weakly
to some $w\in L^q(t_0,T;V^\ast)$, that $(f_n)_n$ converges weakly to some
$f\in L^2(t_0,T;H)$, and that $(x_n(T))_n$ converges weakly to some $z$ in $H$.

Next, we show that
\begin{align}\label{E:WC1}
\mathbf{x}^\prime+w&=f \qquad\text{in $L^q(t_0,T;V^\ast)$,}\\ \label{E:WC2}
x(T)&=z.
\end{align}
To this end, let $\psi\in L^p(t_0,T;V)$. Then
\begin{align*}
0=\int_{t_0}^T \langle\mathbf{x}^\prime_n(t)+A(t,\mathbf{x}_n(t))-f_n(t),\psi(t)\rangle\,dt
\to \int_{t_0}^T \langle\mathbf{x}^\prime(t)+w(t)-f(t),\psi(t)\rangle\,dt
\end{align*}
as $n\to\infty$, i.e.,  \eqref{E:WC1} holds. In particular, we could use
$f_n\xrightarrow{w} f$ in $L^q(t_0,T;V^\ast)$ because
$L^2(t_0,T;H)\subset L^q(t_0,T;V^\ast)$.
To show \eqref{E:WC2}, let $v\in V$ and consider $\psi\in C^1([t_0,T],V)$
defined by $\psi(t):=\varphi(t) v$, $t\in [t_0,T]$, for some $\varphi\in C^\infty([t_0,T])$
with $\varphi(T)=1$ and $\varphi(t_0)=0$.
Then, by the integration-by-parts formula, 
\begin{align*}
(x(T),\psi(T))&=\int_{t_0}^T\langle\mathbf{x}^\prime(t),\psi(t)\rangle+
\langle\psi^\prime(t),\mathbf{x}(t)\rangle\,dt\\
&=\lim_n\int_{t_0}^T\langle\mathbf{x}_n^\prime(t),\psi(t)\rangle+
\langle\psi^\prime(t),\mathbf{x}_n(t)\rangle\,dt\\
&=\lim_n (x_n(T),\psi(T))\\
&=(z,\psi(T)).
\end{align*}
Since $\psi(T)=v$,  the vector $v$ was  arbitrarily chosen in $V$, and 
$V$ is dense in $H$, we have \eqref{E:WC2}.

Now, we show that $w=\hat{A}\mathbf{x}$. The proof of this statement is
nearly identical to the proof of \cite[Lemma~30.6, p.~777f.]{ZeidlerIIB} up
to very minor modifications. First, note that, by
\cite[Theorem~30.A~(c), p.~771]{ZeidlerIIB},
the operator $\hat{A}$ is monotone and
hemicontinuous. According to the monotonicity trick
(see, e.g., \cite[p.~474]{ZeidlerIIB}), it suffices to show that
\begin{align}\label{E:WC3}
\varlimsup_n \langle \hat{A}\mathbf{x}_n,\mathbf{x}_n\rangle_{L^p(t_0,T;V)}\le
\langle w,\mathbf{x}\rangle_{L^p(t_0,T;V)}.
\end{align}
To this end, we apply the integration-by-parts formula to obtain
\begin{align*}
\langle \hat{A}\mathbf{x}_n,\mathbf{x}_n\rangle_{L^p(t_0,T;V)}&=
\langle f_n,\mathbf{x}_n\rangle_{L^p(t_0,T;V)}+
\frac{1}{2}\left[\abs{x_0(t_0)}^2-\abs{x_n(T)}^2\right].
\end{align*}
By \eqref{E:WC2}, $x_n(T)\xrightarrow{w} x(T)$ in $H$ and thus together with
 \cite[Proposition~21.23~(c)]{ZeidlerIIA}, 
$\abs{x(T)}\le\varliminf_n\abs{x_n(T)}$. Moreover, since
$\langle f_n,\mathbf{x}_n\rangle_{L^p(t_0,T;V)}
=\langle f_n,x_n\rangle_{L^2(t_0,T;H)}$,
$f_n\xrightarrow{w} f$ in $L^2(t_0,T;H)$, and $x_n\to x$ in
$L^2(t_0,T;H)$,  we have 
\begin{align*}
\langle f_n,\mathbf{x}_n\rangle_{L^p(t_0,T;V)}\to 
\langle f,x\rangle_{L^2(t_0,T;H)}.
\end{align*}
Thus, by \eqref{E:WC1} and again by the  integration-by-parts formula,
\begin{align*}
\varlimsup_n\langle \hat{A}\mathbf{x}_n,\mathbf{x}_n\rangle_{L^p(t_0,T;V)}&\le
\langle f,x\rangle_{L^2(t_0,T;H)}+\frac{1}{2}[\abs{x_0(t_0)}^2-\abs{x(T)}^2]
=\langle w,\mathbf{x}\rangle_{L^p(t_0,T;V)},
\end{align*}
i.e., we have shown \eqref{E:WC3}. 

It remains to show that $x_n\to x$ in $C([0,T],H)$.
Since $\mathbf{x}^\prime+\hat{A}\mathbf{x}=f$ and $\hat{A}$ is monotone,
 the integration-by-parts formula yields, for every $t\in (t_0,T]$,
\begin{equation}\label{E:CompactEndProof}
\begin{split}
\frac{1}{2}\abs{x_n(t)-x(t)}^2&=
\int_{t_0}^t \langle\mathbf{x}_n^\prime(s)-\mathbf{x}^\prime(s),
\mathbf{x}_n(s)-\mathbf{x}(s)\rangle\,ds\\
&=\int_{t_0}^t -\langle A(s,\mathbf{x}_n(s))-A(s,\mathbf{x}(s)),
\mathbf{x}_n(s)-\mathbf{x}(s)\rangle\\ &\qquad\qquad+
(f_n(s)-f(s),x_n(s)-x(s))\,ds\\
&\le \norm{f_n-f}_{L^2(t_0,t;H)}\,\norm{x_n-x}_{L^2(t_0,t;H)}\\
&\le 2 C\norm{x_n-x}_{L^2(t_0,T;H)}.
\end{split}
\end{equation}
As $\norm{x_n-x}_{L^2(t_0,T;H)}\to 0$ (by \cite[Th\'eor\`eme~5.1, p.~58]{JLLions1969}), we obtain
$\norm{x_n-x}_{\infty}\to 0$,
which concludes the proof.
\end{proof}

\begin{lemma}\label{L:Compact}
Fix $(t_0,x_0)\in [0,T)\times C([0,T],H)$ and $L\ge 0$. Then
$\mathcal{X}^L(t_0,x_0)$ is closed in $C([0,T],H)$.
\end{lemma}
\begin{proof}
Let $(x_n)_n$ be a sequence in $\mathcal{X}^L(t_0,x_0)$ that converges
to some $x\in C([0,T],H)$. By  Lemma~\ref{L:CompactW}, there exists
a pair $(\wx,f)\in W_{pq}(t_0,T)\times L^2(t_0,T;H)$
with $\wx^\prime+\hat{A}\wx=f$ such that $x=\wx$ a.e.~on $(t_0,T)$ and  
$f^{x_{n_k}}\xrightarrow{w} f$ in $L^2(t_0,T;H)$ 
for some subsequence $(x_{n_k})_k$ of $(x_n)_n$. Put $f_k:=f^{x_{n_k}}$, $k\in\N$.
It suffices to show that
\begin{align}\label{E:WC4}
\abs{f(t)}\le L(1+\sup_{s\le t}\abs{x(s)})\qquad\text{a.e.~on $(t_0,T)$.}
\end{align}
By Lemma~\ref{L:Apriori}, there exists a $C\ge 0$ such that, for every $k\in\N$,
\begin{align}\label{E:WC5}
\abs{f_k(t)}\le L(1+\sup_{s\le t}\abs{x_{n_k}(s)})\le L(1+C)\qquad\text{a.e.~on $(t_0,T)$.}
\end{align}
Since $\{h\in H: \abs{h}\le L(1+C)\}$ is weakly compact, non-empty, and convex,
\cite[Theorem~3.1]{P90JMAAOnAPaper}  yields 
\begin{align}\label{E:WC6}
f(t)\in\overline{\mathrm{conv}}\text{ $w$-$\varlimsup$} \{f_k(t)\}_{k\in\N}\qquad\text{a.e.~on $(t_0,T)$.}
\end{align}
(This result can also be found in 
\cite[Proposition~3.9, p.~694]{HPHandbookI} 
and in \cite[Theorem~3.1]{P87IJMMS}.)
In~\eqref{E:WC6}, 
$\text{$w$-$\varlimsup$} \{f_k(t)\}_{k\in\N}$ is the set of all $h\in H$
for which some subsequence $(f_{k_l}(t))_l$ of $(f_k(t))_k$ converges weakly to $h$ in $H$
and
$\overline{\mathrm{conv}} \text{ $w$-$\varlimsup$} \{f_k(t)\}_{k\in\N}$ is the closure of the 
convex hull of $\text{$w$-$\varlimsup$} \{f_k(t)\}_{k\in\N}$.
Now, fix an $h\in\text{ $w$-$\varlimsup$} \{f_k(t)\}_{k\in\N}$ with a corresponding subsequence
 $(f_{k_l}(t))_l$. Then, by \cite[Proposition~21.23 (c), p.~258]{ZeidlerIIA} (Banach-Steinhaus),
  by \eqref{E:WC5}, and since $x_n\to x$ in $C([0,T],H)$, 
 \begin{align*}
 \abs{h}\le \varliminf\limits_l\abs{f_{k_l}(t)} \le \varliminf\limits_l L\left(1+\sup_{s\le t} \abs{x_{n_{k_l}}(s)}\right)
 =L(1+\sup_{s\le t} \abs{x(s)}).
 \end{align*}
 Consequently,  \eqref{E:WC6} yields \eqref{E:WC4}.
 \end{proof}

\begin{remark}\label{R:Compact}
If, instead of $\mathcal{X}^L(t_0,x_0)$, we consider 
\begin{equation*}
\begin{split}
\tilde{\mathcal{X}}^{L}(t_0,x_0):=\{x\in C([0,T],H):\,  
\text{$\exists \wx\in W_{pq}(t_0,T):\exists f^x\in L^2(t_0,T;H)$}:\\
 \mathbf{x}^\prime(t)+A(t,\mathbf{x}(t))=f^x(t) \text{ a.e.~on $(t_0,T)$, }
x=x_0\text{ on $[0,t_0]$, }\\
 \qquad
 x=\wx\text{ a.e.~on $(t_0,T)$, and }
 \abs{f^x(t)}\le L
\text{ a.e.~on $(t_0,T)$}\}.
\end{split}
\end{equation*}
 then the 
inequality corresponding to
\eqref{E:WC4} can be obtained from Mazur's lemma immediately (i.e., without invoking
\cite[Theorem~3.1]{P90JMAAOnAPaper})  as follows.
Let $f_n\xrightarrow{w} f$ in $L^2(t_0,T;H)$. By
Mazur's lemma (see, e.g., \cite[Korollar~III.3.9, p.~108]{WernerFA}),
there exists, for each $n\in\N$, a convex linear combination $\tilde{f}_n$ of
members of $(f_m)_m$, say
$\tilde{f}_n=\sum\nolimits_{i=1}^{N(n)} \lambda_i^{(n)} f_i$,
where $\lambda_i^{(n)}\in (0,1)$, $i=1$, $\ldots$, $N(n)$, and 
$\sum_{i=1}^{N(n)}\lambda_i^{(n)}=1$, such that
$\tilde{f}_n\to f$ in $L^2(t_0,T;H)$. Then
there exists a subsequence of $(\tilde{f}_n)_n$, which we still denote by
 $(\tilde{f}_n)_n$, that converges to $f$ a.e.~on $(t_0,T)$. Hence, for 
 a.e.~$t\in (t_0,T)$ and for every $\eps>0$, there exists an $n_0\in\N$
 such that, for all $n\ge n_0$,
 \begin{align*}
 \abs{f(t)}\le \sum_{i=1}^{N(n)} \lambda_i^{(n)} \abs{f_i(t)}  +\eps\le
 L+\eps.
 \end{align*}
\end{remark}

For the next statement, recall that  $\mathcal{Y}^L(t_0,x_0,\tilde{z})$ is
defined by \eqref{E:Y(t,x,z)}.

\begin{lemma}\label{L:Connected}
Fix $(t_0,x_0,\tilde{z})\in [0,T)\times C([0,T],H)\times H$ and let $L\ge 0$. Then
$\mathcal{X}^L(t_0,x_0)$ is path-connected in $C([0,T],H)$ and 
$\mathcal{Y}^L(t_0,x_0,\tilde{z})$ 
 is path-connected in $C([0,T],H)\times C([t_0,T])$.
\end{lemma}

\begin{proof}
Path-connectedness of $\mathcal{X}^L(t_0,x_0)$ follows immediately from 
\cite[Theorem~2]{PS97AMO}
(see also \cite[Theorem~3.5, p.~50f.]{HPHandbookII}). For the sake of completeness,
we provide a shortened proof in our less general setting.

Let $x$, $y\in\mathcal{X}^L(t_0,x_0)$.
Given $s\in [t_0,T]$, denote by $z_s$ the unique element of $\mathcal{X}(t_0,x_0)$
that satisfies
\begin{align*}
\mathbf{z}_s^\prime(t)+A(t,\mathbf{z}_s(t))=f^x(t).\bfone_{(t_0,s)}(t)+f^y(t).\bfone_{(s,T)}(t)
\text{ a.e.~on $(t_0,T)$.}
\end{align*}
It suffices to show that the mapping $s\mapsto z_s$, $[t_0,T]\to C([0,T],H)$, is continuous.
We show only right-continuity. Left-continuity can be shown similarly. 
To this end, let $s_n\downarrow s$ in $[t_0,T)$. 
Since $\mathcal{X}^L(t_0,x_0)$ is compact (Lemmas~\ref{L:CompactW} and \ref{L:Compact}), 
we can assume without loss of generality that
$z_{s_n}\to z$ in $C([0,T],H)$ 
 for some $z\in\mathcal{X}^L(t_0,x_0)$.
Note that $z=z_s$ on $[0,s]$. Since $A$ is monotone,
we have, for every $N\in\N$, a.e.~$t\in[s_N,T]$, and every $n\ge N$,
\begin{align*}
\langle\mathbf{z}_{s_n}^\prime(t)-\mathbf{z}_s^\prime(t),
\mathbf{z}_{s_n}(t)-\mathbf{z}_s(t)\rangle&=
-\langle A(t,\mathbf{z}_{s_n}(t))-A(t,\mathbf{z}_s(t)),\mathbf{z}_{s_n}(t)-\mathbf{z}_s(t)\rangle
\\ &\qquad+(f^y(t)-f^y(t),{z}_{s_n}(t)-{z}_s(t))\\&\le 0,
\end{align*}
hence $(1/2) (d/dt) \abs{z_{s_n}(t)-z_s(t)}^2\le 0$ and thus, for every $t\in [s_N,T]$,
\begin{align*}
\abs{z_{s_n}(t)-z_s(t)}^2\le\abs{z_{s_n}(s_N)-z_s(s_N)}^2.
\end{align*}
Consequently, for every $t\in [s,T]$ and sufficiently large $N\in\N$,
\begin{align*}
\abs{z_{s_n}(t)-z_s(t)}\le\underbrace{\abs{z_{s_n}(s_N)-z(s_N)}}_{
\le \norm{z_{s_n}-z}_\infty \to 0\quad (n\to\infty)
}+\underbrace{\abs{z(s_N)-z_s(s_N)}}_{\to 0\quad (N\to\infty)}.
\end{align*}
Thus we have shown that $\mathcal{X}^L(t_0,x_0)$ is path-connected in
$C([0,T],H)$.
To finish the proof, it suffices to verify that the mapping
$s\mapsto (z_s,u_s)$, 
$[t_0,T]\to C([0,T],H)\times C([t_0,T])$,
with $u_s$ being defined by
\begin{align*}
u_s(t):=\int_{t_0}^t (\bfone_{(t_0,s)}(r)\cdot f^x(r)+\bfone_{(s,T)}(r)\cdot f^y(r),\tilde{z})-F(r,z_s,\tilde{z})\,dr,
\quad t\in [t_0,T],
\end{align*}
is continuous. This can be done in the same manner as the proof of continuity for the 
mapping $s\mapsto z_s$, $[t_0,T]\to C([0,T],H)$.
\end{proof}

\section{Other notions of solutions} \label{App:Discussion}

\subsection{Classical solutions}
We introduce here a new notion of classical solution for our setting.
This notion is compatible with our notion of minimax solution solutions.
Moreover, it is a faithful extension of the usual classical solution in the case when
$H=V$ and $A$ is continuous. We refer to
Proposition~\ref{P:ClassicalEquivClassical} 
for more details regarding the last point.

\begin{definition}\label{D:Classical}
Let $u\in\mathcal{C}_V^{1,1}([0,T]\times C([0,T],H))$.
Let $L\ge 0$ and let $\tilde{\Omega}$ be some set with 
$\Omega^L\subseteq\tilde{\Omega}
\subseteq C([0,T],H)$.

(i) We say that the function $u$ is a \emph{classical $L$-supersolution} of \eqref{E:PPDE2} 
on $[0,T]\times\tilde{\Omega}$
 if $u(T,\cdot)\ge h$,
and if, for every $(t_0,x_0,z)\in [0,T)\times\Omega\times H$,
there exists a trajectory $x\in\mathcal{X}^L(t_0,x_0)$ such that 
\begin{equation}\label{E:classUpper}
\begin{split}
&\partial_t u(t_0,x_0)+F(t_0,x_0,z)\\ &+
\varlimsup_{\delta\downarrow 0} \frac{1}{\delta}\left[\int_{t_0}^{t_0+\delta} (f^x(t),\partial_x u(t,x)-z)- \langle A(t,\wx(t)),\partial_x  u(t,x)\rangle\,dt\right]
\le 0.
\end{split}
\end{equation}

(ii) We say that the function $u$ is a \emph{classical $L$-subsolution} of \eqref{E:PPDE2}
on $[0,T]\times\tilde{\Omega}$ if 
 $u(T,\cdot)\le h$,
and if, for every $(t_0,x_0,z)\in [0,T)\times\Omega\times H$,
there exists a trajectory $x\in\mathcal{X}^L(t_0,x_0)$ such that 
\begin{equation*}
\begin{split}
&\partial_t u(t_0,x_0)+F(t_0,x_0,z)\\ &+
\varliminf_{\delta\downarrow 0} \frac{1}{\delta}\left[\int_{t_0}^{t_0+\delta} (f^x(t),\partial_x u(t,x)-z)- \langle A(t,\wx(t)),\partial_x  u(t,x)\rangle\,dt\right]
\ge 0.
\end{split}
\end{equation*}

(iii) We say that $u$ is a \emph{classical $L$-solution} of \eqref{E:PPDE2} if 
 $u(T,\cdot) = h$,
and if, for every $(t_0,x_0,z)\in [0,T)\times\Omega\times H$,
there exists an $x\in\mathcal{X}^L(t_0,x_0)$ such that 
\begin{equation*}
\begin{split}
&\partial_t u(t_0,x_0)+F(t_0,x_0,z)\\ &+
\lim_{\delta\downarrow 0} \frac{1}{\delta}\left[\int_{t_0}^{t_0+\delta} (f^x(t),\partial_x u(t,x)-z)- \langle A(t,\wx(t)),\partial_x  u(t,x)\rangle\,dt\right]
= 0.
\end{split}
\end{equation*}
\end{definition}

The following statement justifies our slightly involved notion of classical solutions,
in particular, the reason that, in the case of classical $L$-supersolutions,
we want \eqref{E:classUpper} to hold for all
$z\in H$ and not only for $\partial_x u(t_0,x_0)$.

\begin{theorem}[Compatibility] \label{T:ConsMinimaxClass}
Let $u\in\mathcal{C}_V^{1,1}([0,T]\times C([0,T],H))$. 
Let $L\ge 0$ and let $\tilde{\Omega}$ be some set with 
$\Omega^L\subseteq\tilde{\Omega}
\subseteq C([0,T],H)$.
Then
$u$ is a minimax  $L$-solution (resp.~minimax $L$-subsolution, 
resp.~minimax $L$-supersolution)  of \eqref{E:PPDE2}
on $[0,T]\times\tilde{\Omega}$
if and only if $u$ is a classical $L$-solution (resp.~classical $L$-subsolution, 
resp.~classical $L$-supersolution) 
of \eqref{E:PPDE2} on $[0,T]\times\tilde{\Omega}$..
\end{theorem}

\begin{proof}
We prove only the case of equivalence between minimax $L$-supersolutions
 and classical $L$-supersolutions. To this end, fix $(t_0,x_0,z)\in [0,T)\times\tilde{\Omega}\times H$.

(i) Let $u$ be a classical $L$-supersolution of \eqref{E:PPDE2} on $[0,T]\times\tilde{\Omega}$.
Then there exists an $x\in\mathcal{X}^L(t_0,x_0)$ such that
\eqref{E:classUpper} holds. To show that
$u$ is an upper solution, it suffices, by Proposition~\ref{P:Dini},
to check that \eqref{E:Dini} holds. Indeed, 
applying the functional chain rule~\eqref{E:ChainRule}
as well as \eqref{E:classUpper} yields
\begin{align*}
&\varliminf_{\delta\downarrow 0} \frac{1}{\delta}\left[
u(t_0+\delta,x)-u(t_0,x_0)+\int_{t_0}^{t_0+\delta} 
(-f^x(t),z)+F(t,x,z)\,dt
\right]\\
& \le
\varlimsup_{\delta\downarrow 0} \frac{1}{\delta}\left[
\int_{t_0}^{t_0+\delta} 
\partial_t u(t,x)+\langle \wx^\prime(t),\partial_x u(t,x)\rangle 
-(f^x(t),z)+F(t,x,z)\,dt
\right]\\
&\le \partial_t u(t_0x_0)+F(t_0,x_0,z)\\
&\qquad+\varlimsup_{\delta\downarrow 0} \frac{1}{\delta}\left[
\int_{t_0}^{t_0+\delta} 
(f^x(t),\partial_x u(t,x)-z)-
\langle A(t,\wx(t)),\partial_x u(t,x)\rangle 
\,dt
\right]\\
&\le 0.
\end{align*}

(ii) Let $u$ be a minimax $L$-supersolution of \eqref{E:PPDE2}.
Then there exists an $x\in\mathcal{X}^L(t_0,x_0)$ such that, for every
$\delta\in (0,T-t_0]$,
\begin{align*}
u(t_0+\delta,x)-u(t_0,x_0)+\int_{t_0}^{t_0+\delta} F(t,x,z)-(f^x(t),z)\,dt\le 0.
\end{align*}
Hence,
\begin{align*}
&\partial_t u(t_0,x_0)+F(t_0,x_0,z)=
\lim_{\delta\downarrow 0}\frac{1}{\delta} \left[\int_{t_0}^{t_0+\delta}
\partial_t u(t,x)+F(t,x,z)\,dt
\right]\\
&=\varliminf_{\delta\downarrow 0}\frac{1}{\delta} \left[
u(t_0+\delta,x)-u(t_0,x_0)+\int_{t_0}^{t_0+\delta}
\langle -\wx^\prime(t),\partial_x u(t,x)\rangle+
F(t,x,z)
\,dt
\right]\\
&\le \varliminf_{\delta\downarrow 0}\frac{1}{\delta} \left[
\int_{t_0}^{t_0+\delta}
(f^x(t),z)+\langle  A(t,\wx(t)-f^x(t),\partial_x u(t,x)\rangle
\,dt
\right],
\end{align*}
i.e, \eqref{E:classUpper} is satisfied.
\end{proof}

\begin{proposition}\label{P:ClassicalEquivClassical}
Suppose that $V=H$ and that the map $(t,v)\mapsto A(t,v)$, $[0,T]\times H\to H$, is continuous.
Let $u\in\mathcal{C}^{1,1}_ V([0,T]\times\ C([0,T],H))$ with $u(T,\cdot)\ge h$.
Let $\Omega^{L_0}\subseteq \tilde{\Omega}\subseteq C([0,T],H)$.
Then $u$ is a classical  $L_0$-supersolution of \eqref{E:PPDE2} on $[0,T]\times\tilde{\Omega}$
 in the sense of
 Definition~\ref{D:Classical} if and only if $u$ is a classical supersolution 
 of \eqref{E:PPDE2} on $[0,T]\times\tilde{\Omega}$
  in the usual sense, i.e., for every $(t_0,x_0)\in [0,T)\times\tilde{\Omega}$,
 \begin{align*}
 \partial_t u(t_0,x_0)-(A(t_0,x_0(t_0)),\partial_x u(t_0,x_0))+F(t_0,x_0,\partial_x u(t_0,x_0))\le 0.
 \end{align*}
\end{proposition}

\begin{proof}
We show only the nontrivial direction. To this end, let $u$ be a classical supersolution of
\eqref{E:PPDE2} in the usual sense on $[0,T]\times\tilde{\Omega}$
and fix $(t_0,x_0,z)\in [0,T)\times\tilde{\Omega}\times H$.
Furthermore, let
$\tilde{\mathbf{x}}\in W_{pq}(t_0,T)$ be the solution of
\begin{align*}
 \tilde{\mathbf{x}}^\prime(t)+ A(t,\tilde{\mathbf{x}}(t))&=\tilde{f}(t,\tilde{x})
\text{ a.e.~on $(t_0,T)$,}\quad
\tilde{x}=x_0\text{ on $[0,t_0]$,}
\end{align*}
where  $\tilde{f}:[t_0,T]\times C([0,T],H)\to H$ is defined by
\begin{align*}
\tilde{f}(t,x):=
\begin{cases}\frac{F(t,x,\partial_x u(t,x))-F(t,x,z)}{\abs{\partial_x u(t,x)-z)}^2}
(\partial_x u(t,x)-z)\text{ if $\partial_x u(t,x)\neq z$,}\\
0\text{ if $\partial_x u(t,x)=z$.}
\end{cases}
\end{align*}
Then $\tilde{x}\in\mathcal{X}^{L_0}(t_0,x_0)$ with $f^{\tilde{x}}=\tilde{f}(\cdot,\tilde{x})$
and
\begin{align*}
&\partial_t u(t_0,x_0)+F(t_0,x_0,z)\\ &\qquad\quad+\varlimsup_{\delta\downarrow 0}
\frac{1}{\delta}\left[\int_{t_0}^{t_0+\delta} (f^{\tilde{x}}(t),\partial_x u(t,\tilde{x})-z)
-(A(t,\tilde{x}),\partial_x u(t,\tilde{x}))\,dt\right]\\
&=\partial_t u(t_0,x_0)+F(t_0,x_0,z)\\&\qquad\quad+\varlimsup_{\delta\downarrow 0}
\frac{1}{\delta}\left[\int_{t_0}^{t_0+\delta} F(t,\tilde{x},\partial_x(t,\tilde{x}))-F(t,\tilde{x},z)
-(A(t,\tilde{x}),\partial_x u(t,\tilde{x}))\,dt\right]\\
&=\partial_t u(t_0,x_0)-(A(t_0,x_0(t_0)),\partial_x u(t_0,x_0))+F(t_0,x_0,\partial_x u(t_0,x_0))\le 0.
\end{align*}
I.e., $u$ is a classical $L_0$-supersolution of \eqref{E:PPDE2} on 
$[0,T]\times\tilde{\Omega}$ in the sense of Definition~\ref{D:Classical} .
\end{proof}

\subsection{Viscosity solutions}
Given $L\ge 0$,  a set $\tilde{\Omega}$ with $\Omega^{L}
\subseteq\tilde{\Omega}\subseteq C([0,T],H)$, a non-anticipating function
$u:[0,T]\times\tilde{\Omega}\times\to\R$, and
 $(t_0,x_0)\in [0,T)\times\tilde{\Omega}$, put
 \begin{align*}
&\underline{\mathcal{A}}^L u(t_0,x_0):=\{\varphi\in\mathcal{C}_V^{1,1}
([t_0,T]\times
C([0,T],H)):\exists T_0\in (t_0,T]:\\
&\qquad 0=(\varphi-u)(t_0,x_0)=\inf_{(t,x)\in [t_0,T_0]\times
\mathcal{X}^{L}(t_0,x_0)} [(\varphi-u)(t,x)]\}.
\end{align*}

\begin{definition}
Fix a set $\tilde{\Omega}$ with   
$\Omega^{L}\subseteq\tilde{\Omega}\subseteq C([0,T],H)$.

(i) Let $L\ge 0$. A function $u\in\mathrm{USC}([0,T]\times\tilde{\Omega})$ 
 is a \emph{viscosity  $L$-subsolution} of \eqref{E:PPDE2}
on $[0,T]\times\tilde{\Omega}$
 if, for every  $(t_0,x_0)\in [0,T)\times\tilde{\Omega}$ and
  for every test function
 $\varphi\in\underline{\mathcal{A}}^L u(t_0,x_0)$, 
there exists an $x\in \mathcal{X}^{L}(t_0,x_0)$  such that
\begin{align*}
&\partial_t\varphi(t_0,x_0)
+\varlimsup_{\delta\downarrow 0} \frac{1}{\delta} \Biggl[\int_{t_0}^{t_0+\delta}
-\langle A(t,\wx(t)),\partial_x  \varphi(t,x)\rangle\,dt \Biggr]
+F(t_0,x_0,\partial_x\varphi(t_0,x_0))
\ge 0.
\end{align*}
(ii) Let $L\ge 0$. A function $u\in\mathrm{LSC}([0,T]\times\tilde{\Omega})$ 
 is a \emph{viscosity  $L$-supersolution} of \eqref{E:PPDE2}
on $[0,T]\times\tilde{\Omega}$
 if, for every  $(t_0,x_0)\in [0,T)\times\tilde{\Omega}$ and
  for every test function
 $\varphi\in\underline{\mathcal{A}}^L u(t_0,x_0)$, 
there exists an $x\in \mathcal{X}^{L}(t_0,x_0)$  such that
\begin{align*}
&\partial_t\varphi(t_0,x_0)
+\varlimsup_{\delta\downarrow 0} \frac{1}{\delta} \Biggl[\int_{t_0}^{t_0+\delta}
-\langle A(t,\wx(t)),\partial_x  \varphi(t,x)\rangle\,dt \Biggr]
+F(t_0,x_0,\partial_x\varphi(t_0,x_0))
\le 0.
\end{align*}
(iii) A function $u\in\mathrm{USC}([0,T]\times\tilde{\Omega})$ 
(resp.~$\in\mathrm{LSC}([0,T]\times\tilde{\Omega})$ is a 
\emph{viscosity  subsolution} (resp.~\emph{viscosity supersolution}) of \eqref{E:PPDE2}
on $[0,T]\times\tilde{\Omega}$ if it is a viscosity $L$-subsolution 
(resp.~viscosity $L$-supersolution) of \eqref{E:PPDE2}
on $[0,T]\times\tilde{\Omega}$  for some $L\ge 0$.

(iv) A function $u\in C([0,T]\times\tilde{\Omega})$ is a \emph{viscosity solution} 
of \eqref{E:PPDE2}
on $[0,T]\times\tilde{\Omega}$ if it is both, a viscosity sub- as well as a viscosity supersolution
of \eqref{E:PPDE2}
on $[0,T]\times\tilde{\Omega}$.
\end{definition}

Similarly as in the proof of Theorem~\ref{T:MinimaxDoubling}, one can show that
a minimax solution is a viscosity solution. Thus, together with 
Theorem~\ref{T:ExistenceMinimax2}, we obtain the following result.

\begin{theorem}[Existence]
There exists a viscosity solution 
of \eqref{E:PPDE2}
on $[0,T]\times\Omega$.
\end{theorem}

\bibliographystyle{amsplain}
\bibliography{PDE}
\end{document}